\documentclass[generic,square]{imsart}

\RequirePackage[OT1]{fontenc}
\usepackage[table]{xcolor}
\RequirePackage{amsmath,amssymb,amsthm,mathrsfs, enumerate,bm,makeidx,xcolor,multirow,pbox,tikz}
\RequirePackage[colorlinks,linkcolor=red,citecolor=blue,urlcolor=blue]{hyperref}
\RequirePackage{natbib}
\usepackage{enumitem}

\RequirePackage{xr}

\numberwithin{equation}{section}

\renewcommand{\epsilon}{\varepsilon}

\DeclareMathOperator*{\argmin}{argmin}
\DeclareMathOperator*{\argmax}{argmax}

\newcommand{\beq}{\begin{equation}}
  \newcommand{\eeq}{\end{equation}}
\newcommand{\beqa}{\begin{equation} \begin{aligned}}
    \newcommand{\eeqa}{\end{aligned} \end{equation}}
\newcommand{\beqas}{\begin{equation*} \begin{aligned}}
    \newcommand{\eeqas}{\end{aligned} \end{equation*}}

\newcommand{\bit}{\begin{itemize}}
  \newcommand{\eit}{\end{itemize}}
\newcommand{\bmat}{\begin{bmatrix}}
  \newcommand{\emat}{\end{bmatrix}}

\makeindex

\theoremstyle{definition}\newtheorem{problem}{Problem}[section]
\theoremstyle{definition}
\theoremstyle{remark}\newtheorem{assumption}{Assumption}
\theoremstyle{remark}\newtheorem{remark}{Remark}[section]
\theoremstyle{definition}\newtheorem{example}[problem]{Example}
\theoremstyle{plain}\newtheorem{theorem}[problem]{Theorem}
\theoremstyle{plain}\newtheorem{lemma}[problem]{Lemma}
\theoremstyle{plain}\newtheorem{proposition}[problem]{Proposition}
\theoremstyle{plain}\newtheorem{corollary}[problem]{Corollary}
\theoremstyle{plain}

\newcommand{\DD}{{\mathbb D}}
\newcommand{\FF}{ {\mathbb F}}

\newcommand{\PP}{{\mathbb P}}
\newcommand{\RR}{{\mathbb R}}

\newcommand{\VV}{{\mathbb V}}
\newcommand{\XX}{{\mathbb X}}

\usepackage{xparse} %
\usepackage{etoolbox} %
\NewDocumentCommand\DDnTwo{O{}}{
  \ifstrempty{#1}{
    \mathbb{D}_{n, t_1, t_2}
  }{
    \mathbb{D}_{n, {#1}}
  }
}
\NewDocumentCommand\twoArgs{O{}}{
  \ifstrempty{#1}{
    t_1, t_2
  }{
    {#1}
  }
}
\NewDocumentCommand\oneArg{O{}}{
  \ifstrempty{#1}{
    t
  }{
    {#1}
  }
}

\global\long\def\RnTwo{R_{\xn,1, t_1,t_2}}

\usepackage{versions}
\usepackage{mathtools}

\global\long\def\FF{\mathbb{F}} %
\global\long\def\vvo{\vp_{0}}
\global\long\def\ffo{f_{0}} %
\global\long\def\FFo{F_{0}}
\newcommand{\lv}{\left\vert} %
  \newcommand{\rv}{\right\vert}
\newcommand{\lp}{\left(} %
  \newcommand{\rp}{\right)}
\newcommand{\lb}{\left\{} %
  \newcommand{\rb}{\right\}}

\global\long\def\vvna{\widehat{\vp}_{\xn}^{0}}
\global\long\def\vvnat#1{\widehat{\vp}_{\xn,#1}^{0}}
\global\long\def\FFn{\widehat{F}_{\xn}}
\global\long\def\ffn{\widehat{f}_{\xn}}
\global\long\def\ffna{\widehat{f}_{n}^{0}}
\global\long\def\FFna{\widehat{F}_{n}^{0}}
\global\long\def\vvn{\widehat{\vp}_{\xn}}
\global\long\def\vvnt#1{\widehat{\vp}_{\xn,#1}}
\global\long\def\vp{\varphi}
\global\long\def\xn{n}
\global\long\def\Fn{\mathbb{F}_{\xn}}

\excludeversion{mylongform}

\newenvironment{longform}{%
  \par \vspace{.2cm} \color{green} { \bf  Notes to self:}%
}{%
  $\blacktriangle$  \vspace{.1cm}
}

\begin{document}

\begin{frontmatter}
  \title{Inference for the mode of a log-concave density}  %
  \runtitle{Inference for the mode}

  \begin{aug}
    \author{\fnms{Charles R.}
      \snm{Doss}\thanksref{t1}\ead[label=e1]{cdoss@stat.umn.edu}\ead[label=u2,url]{http://users.stat.umn.edu/\textasciitilde cdoss/}}
    \and
    \author{\fnms{Jon A.} \snm{Wellner}\thanksref{t2}\ead[label=e2]{jaw@stat.washington.edu}}
    \ead[label=u1,url]{http://www.stat.washington.edu/jaw/}
    \thankstext{t1}{Supported in part by NSF Grant DMS-1104832  and a University of Minnesota Grant-In-Aid grant.}
    \thankstext{t2}{Supported in part by NSF Grants DMS-1104832 and DMS-1566514,  NI-AID grant 2R01 AI291968-04,  and by the Isaac Newton Institute for Mathematical Sciences program
{\sl Statistical Scalability},  EPSRC Grant Number LNAG/036 RG91310.}
    \runauthor{Doss and Wellner}
    \affiliation{University of Minnesota; \ University of Washington\thanksmark{m1}}
    \address{School of Statistics \\University of Minnesota\\Minneapolis,  MN 55455\\
      \printead{e1}\\
      \printead{u2}}
    \address{Department of Statistics, Box 354322\\University of Washington\\Seattle, WA  98195-4322\\
      \printead{e2}\\
      \printead{u1}}
  \end{aug}

  \begin{abstract}
    We study a likelihood ratio test for the location of the mode of a log-concave density.
    Our test is based on comparison of the log-likelihoods corresponding to the unconstrained
    maximum likelihood estimator of a log-concave density and the constrained maximum likelihood
    estimator where the constraint is that the mode of the density is fixed, say at $m$.
    The constrained estimation problem is studied in detail in
\cite{Doss-Wellner:2016ModeConstrained}.
    Here the results of that paper are used to
    show that,  under the null hypothesis (and strict curvature of $-\log f$ at the mode),
    the likelihood ratio statistic is asymptotically pivotal:  that is, it converges in distribution
    to a limiting distribution which is free of nuisance parameters, thus playing the role of the
    $\chi_1^2$ distribution in classical parametric statistical problems.
    By inverting this family of tests we obtain new (likelihood ratio based) confidence intervals
    for the mode of a log-concave density $f$.  These new intervals do not depend on any
    smoothing parameters.
    We study the new confidence intervals via Monte Carlo methods and illustrate them with
    two real data sets.
    The new intervals seem to have several advantages over existing procedures.
    Software implementing the test and confidence intervals is available in the R
    package \verb+logcondens.mode+.
  \end{abstract}

  \begin{keyword}[class=AMS]
    \kwd[Primary ]{62G07} %
    \kwd[; secondary ]{62G15} %
    \kwd{62G10} %
    \kwd{62G20} %
  \end{keyword}

  \begin{keyword}
    \kwd{mode}
    \kwd{empirical processes}
    \kwd{likelihood ratio}
    \kwd{pivot}
    \kwd{convex optimization}
    \kwd{log-concave}
    \kwd{shape constraints}
  \end{keyword}

\end{frontmatter}

\newpage

\tableofcontents

\newpage

\section{Introduction and overview:  inference for the mode}
\label{sec:Intro}

Let ${\cal P}$ denote the class of all log-concave densities $f$ on $\RR$.
It is well-known since Ibragimov (1956) that all log-concave densities $f$ are strongly unimodal,
and conversely;  see \cite{DJ1988UCA}  %
for an exposition of the basic theory.
Of course, ``the mode'' of a log-concave density $f$ may not be a single point.
It is, in general, the modal interval $MI (f) \equiv \{ x \in \RR : \ f(x) = \sup_{y \in \RR} f(y) \}$,
and to describe ``the mode'' completely we need to choose a specific element of
$MI(f)$, for example $M(f) \equiv \inf \{ x \in MI (f) \}$.
For a large sub-class of log-concave densities the set reduces to a single point.
Our focus here is on the latter case and, indeed, on inference concerning $M(f)$
based on i.i.d. observations $X_1, \ldots , X_n$ with density $f_0 \in {\cal P}$.
We have restricted to log-concave densities for several reasons:
\begin{description}
\item[(a)] It is well known that the MLE over the class of all unimodal densities does not exist;
  see e.g. Birg\'e (1997).
\item[(b)]
  On the other hand, MLE's do exist for the class ${\cal P}$ of log-concave densities if $n\ge 2$:
  see, for example,
  \cite{MR2459192}, %
  \cite{RufibachThesis},   %
  \cite{DR2009LC}.
\item[(c)]
  Moreover the MLE's for the class of log-concave densities have remarkable stability and continuity
  properties under model miss-specification:  see e.g.
  \cite{DSS2011LCreg}.  %
\end{description}

Before proceeding with our overview, it will be helpful to introduce some notation
for derivatives.  (Further notation and terminology will be given in Subsection~\ref{sec:notation}.)
In particular, we let $f^{\prime} $ denote the derivative of a differentiable function $f$, and we
write $f^{\prime\prime}$ for the second derivative.  We also use the notation $f^{(i)}$ for the
$i$th derivative of $f$, particularly for higher derivatives.

Concerning estimation of the mode,
\cite{BRW2007LCasymp} %
showed that if $f_0 = e^{\varphi_0}$ where the concave function $\varphi_0$ has
second derivative $\varphi_0^{(2)} \equiv \varphi_0^{\prime\prime}$
at the mode $m_0 = M(f_0)$ of $f_0$ satisfying
$\varphi_0^{(2)} (m_0) < 0$,
then the MLE $M(\widehat{f}_n)$ satisfies
\begin{eqnarray}
  n^{1/5} \left ( M(\widehat{f}_n ) - M(f_0) \right ) \rightarrow_d
  \left ( \frac{(4!)^2 f_0 (m_0)}{f_0^{^{(2)}} (m_0)^2} \right )^{1/5} M (H_2 ^{(2)} )
  \label{LimitDistribMLEofMode}
\end{eqnarray}
where $M(H_2^{(2)})$ has a universal distribution (not depending on $f_0$).
Here $\{ H_2(t) : \ t \in \RR \}$
is  the ``invelope''  process on $\RR$ defined in terms of
the ``driving process'' $\{ Y(t) : \ t \in \RR\}$ defined by
$Y(t) = - t^4 + \int_0^t  W(s) ds$ for $t \in \RR$.
Thus with $X(t) \equiv Y^{(1)} (t) = -4 t^3 + W(t)$,
\begin{eqnarray}
  d X(t) = g_0 (t) dt + d W(t),
  \label{WhiteNoiseCanonicalConcave}
\end{eqnarray}
where $W$ is two-sided Brownian motion on $\RR$ and $g_0 (t) \equiv - 12 t^2$.
The process $H_2$ and its concave second derivative $H_2^{(2)}$
first appeared in \cite{MR1891741,MR1891742} in the study
of other nonparametric estimation problems involving convex or concave functions;
see also \cite{BRW2007LCasymp}.

The limit distribution (\ref{LimitDistribMLEofMode})
gives useful information about the behavior of $M(\widehat{f}_n)$, but
it is somewhat difficult to use for inference because of the constant \\
$( (4!)^2 f_0 (m_0) / f_0^{(2)} (m_0)^2 )^{1/5}$ which involves
the unknown density through
the second derivative
$f_0^{(2)} (m_0)$.  This can be estimated via smoothing methods,
but because we wish to avoid the consequent problem of choosing bandwidths
or other tuning parameters, we take a different
approach to inference here.

Instead, we first consider the following testing problem:   test
\begin{eqnarray*}
  H : \ M(f) = m \ \ \ \mbox{versus} \ \ \ K: \ M(f) \not= m
\end{eqnarray*}
where $m \in \RR$ is fixed.
To construct a likelihood ratio test of $H$ versus $K$ we first need to construct both the unconstrained
MLE's $\widehat{f}_n$ and the mode-constrained MLE's $\widehat{f}_n^0$.
The unconstrained MLE's $\widehat{f}_n$ are available from the results of
\cite{MR2459192}, %
\cite{RufibachThesis},   %
and \cite{DR2009LC} cited above.
Corresponding results concerning the existence and properties of the mode-constrained MLE's
$\widehat{f}_n^0$ are given in the companion paper
\cite{Doss-Wellner:2016ModeConstrained}.  %
Global convergence rates for both estimators are given in \cite{DossWellner:2016a}. %
Once both the unconstrained estimators $\widehat{f}_n$ and the constrained
estimators $\widehat{f}_n^0$ are available, then we can consider
the natural likelihood ratio test of $H$ versus $K$:  reject  the null hypothesis $H$ if
\begin{eqnarray*}
  2 \log \lambda_n \equiv 2 \log \lambda_n (m)  \equiv 2 n \PP_n  ( \log \widehat{f}_n  - \log \widehat{f}_n^0 )
  = 2 n \PP_n (\widehat{\varphi}_n - \widehat{\varphi}_n^0)
\end{eqnarray*}
is ``too large''
where $\widehat{f}_n = \exp ( \widehat{\varphi}_n ) $, $\widehat{f}_n^0  = \exp ( \widehat{\varphi}_n^0 ) $,
$\PP_n = \sum_{i=1}^n \delta_{X_i} / n$ is the empirical measure, and $\PP_n(g) = \int g d\PP_n$.
To carry out this test we need to know how large is ``too large''; i.e.\ we need to know
the (asymptotic) distribution  of $2 \log \lambda_n$ when $H $ is true.
Thus the primary goal of this paper is to prove the following theorem:

\begin{theorem}
  \label{LRasympNullDistribution}
  If $X_1, \ldots , X_n$ are i.i.d. $f_0 = e^{\varphi_0}$ with mode $m$
  where $\varphi_0$ is concave,
  twice continuously differentiable at $m$,
  and $\varphi_0^{(2)}(m) < 0$, %
  then
  \begin{eqnarray*}
    2 \log \lambda_n \rightarrow_d \DD
  \end{eqnarray*}
  where $\DD$ is a universal limiting distribution (not depending on $f_0$);  thus
  $2 \log \lambda_n$ is asymptotically pivotal under the assumption $\varphi_0^{(2)}(m)<0$.
\end{theorem}

\noindent
With Theorem~\ref{LRasympNullDistribution} in hand, our likelihood ratio test with (asymptotic) size $\alpha \in (0,1)$ becomes:
``reject $H$ if  $2 \log \lambda_n > d_{\alpha}$'' where $d_{\alpha}$ is chosen so that\\
$P(\DD > d_{\alpha} ) = \alpha$.  Furthermore, we can then form confidence intervals for $m$
by inverting the family of likelihood ratio tests:
let
\begin{eqnarray}
  J_{n,\alpha} \equiv \{ m  \in \RR : \ 2\log \lambda_n (m) \le d_{\alpha} \} .
  \label{LRConfInterval}
\end{eqnarray}
Then it follows
that for $f_0 \in {\cal P}_m = \{ f \in {\cal P}  : \ M(f ) = m \}$ with $(\log f_0)^{(2)} (m) <0$, we have
$$
P_{f_0} ( m \in  J_{n,\alpha} ) \rightarrow P( \DD \le d_{\alpha} ) = 1- \alpha.
$$
This program is very much analogous to the methods for pointwise inference for
nonparametric estimation of monotone increasing or decreasing functions
developed by \cite{MR1891743}  %
and
\cite{MR2341693}.   %
Those methods have recently been extended to include pointwise inference for nonparametric estimation
of a monotone density
by \cite{MR3375875}.  %
Theorem~\ref{LRasympNullDistribution} says that $2 \log \lambda_n$ is (asymptotically) pivotal
over the class of all log-concave densities $f_0$ satisfying $(\log f_0)^{(2)} (m) < 0$.
(That log-likelihood ratios are frequently asymptotically pivotal is sometimes known as the ``Wilks phenomenon'' in honor of the classical result in this direction in regular parametric models by
\cite{wilks1938large}.)  We can specify more about the form of the limit random variable $\DD$; see Remark~\ref{rem:rem:DD-formula}.

\medskip

A secondary goal of this paper is to begin a study of the likelihood ratio statistics $2 \log \lambda_n$ under fixed alternatives.
We leave the study of the log likelihood ratio statistic under local (contiguous) alternatives for future work.
Our second theorem concerns the situation when $f \in {\cal P}$ has mode $M(f) \not= m$.

\begin{theorem}
  \label{LR-limit-FixedAlternative}
  Suppose that $f_0 \in {\cal P}$ with $m \notin MI(f_0)$.  %
  Then
  \begin{eqnarray}
    \frac{2}{n} \log \lambda_n (m)
    & \rightarrow_p & 2 K(f_0, f_m^0 ) \nonumber  \\
    & = & 2 \inf \{ K(f_0, g) : \ g \in {\cal P}_m \} > 0  \label{LR-limit-FixedAlt}
  \end{eqnarray}
  where $ f_m^0 \in {\cal P}_m$ achieves the infimum in \eqref{LR-limit-FixedAlt} and
  \begin{eqnarray*}
    K(f,g) \equiv \left \{ \begin{array}{c l} \int f (x) \log \frac{f(x)}{g(x)} d x , & \ \ \mbox{if} \ \ f \prec \prec g \\  \infty, &  \ \ \mbox{otherwise} .
      \end{array} \right .
  \end{eqnarray*}
\end{theorem}
Here $f \prec \prec g$ means $f = 0$ whenever $g = 0$ except perhaps on a set
of Lebesgue measure $0$.  The proof of
Theorem~\ref{LR-limit-FixedAlternative}
is given in Subsection~\ref{ssec:PfSketchThm2},
and
relies on the methods used by
\cite{MR2645484} %
and \cite{DSS2011LCreg},  %
in combination with the results of
\cite{DossWellner:2016a}.
Theorem~\ref{LR-limit-FixedAlternative} implies consistency of the likelihood
ratio test based on the critical values from
Theorem~\ref{LRasympNullDistribution}.  That is:   let $d_{\alpha}$ satisfy $P(
\DD> d_{\alpha}) = \alpha$ for $0 < \alpha < 1$, and suppose we reject $H: \
M(f) = m$ if $2 \log \lambda_n (m) > d_{\alpha}$.

\begin{corollary}  If the hypotheses of Theorem~\ref{LR-limit-FixedAlternative} hold, then the
  likelihood ratio test  ``reject $H$ if $2\log \lambda_n(m)  > d_{\alpha}$'' is consistent:  if $f \notin {\cal P}_m$, then
  \begin{eqnarray*}
    P_f ( 2 \log \lambda_n (m) > d_{\alpha} ) \rightarrow 1 .
  \end{eqnarray*}
\end{corollary}

Here is an explicit example:

\begin{example}
  Suppose that $f $ is the Laplace density given by
  $$
  f(x) = (1/2) \exp ( - | x |) .
  $$
  First we note that $M(f) = 0$ so that $f \notin {\cal P}_1$.   Thus for testing
  $H : \ M(f) = 1$ versus $K : M(f) \not= 1$,  the Laplace density $f$ satisfies $f \in {\cal P} \setminus {\cal P}_1$.
  So we have (incorrectly) hypothesized that $M(f) = 1 \equiv m$.  In this case
  the constrained MLE $\widehat{f}_n^0$ satisfies $\int | \widehat{f}_n^0 - f^0 | dx \rightarrow_{a.s.} 0$
  where $f^0 \equiv g^* \in {\cal P}_1$ is determined by
  Theorem~\ref{DSS-thm2.7m} %
  which is
  the population analogue of
  Theorem 2.10 of \cite{Doss-Wellner:2016ModeConstrained}.  %
  It also satisfies (\ref{LR-limit-FixedAlt}) in Theorem 1.2.  In the present case, $g^*  = g_{a^*} $ where
  $\{ g_a : \ a \in (0,1]\}$ is the family of densities given by
  \begin{eqnarray*}
    g_a (x) = \left \{ \begin{array}{l c} (1/2) e^x, & - \infty < x \le -a \\ (1/2) e^{-a} , & -a \le x \le 1, \\ (1/2) e^{-a} e^{-c(x-1)}, & \ \ 1 \le x < \infty ,
      \end{array} \right .
  \end{eqnarray*}
  where $c \equiv c(a) = 1/(2e^a - (2+a))$ is chosen so that $\int g_a (x)dx =1$.
  Here it is not hard to show that $a^* \approx .490151 \ldots $ satisfies $c(a^*)^2 = \exp ( - (a^* -1))$, while
  $K(f,f^0) = K(f,g_{a^*} ) \approx 0.03377\ldots $.
\end{example}

Although the basic approach here has points in common with the developments in \cite{MR1891743} and \cite{MR2341693}, the details of the proofs require several new tools and techniques due to the relative lack of development of theory for the mode-constrained log-concave MLEs.  Furthermore, the proof of Theorem~\ref{LRasympNullDistribution} is significantly more complicated than corresponding proofs in \cite{MR1891743}, \cite{MR2341693}, or \cite{MR3375875}:  in the present context, the mode-constrained estimator and the unconstrained estimator are not identically equal to each other away from the constraint, whereas in many monotonicity-based cases, the corresponding constrained and unconstrained estimators are indeed equal away from the constraint.  In the case of monotone density estimation studied by
\cite{MR3375875}, the constrained and unconstrained estimators are not identically equal away from the constraint, but the differences can be handled using the so-called min-max formula (see e.g., Lemma~3.2 of \cite{MR3375875}),  %
which does not have an analog for concavity-based problems.  Thus, beyond being interesting in its own right, the proof of Theorem~\ref{LRasympNullDistribution} is useful for opening the door to the study of likelihood ratios in other concavity/convexity-based problems.  These could be likelihood ratios for locations of extrema or likelihood ratios for the values (heights) of functions in concavity/convexity-based problems.
We present some discussion of possible extensions in Section~\ref{sec:conclusions}.

To prove Theorem~\ref{LRasympNullDistribution} we first prepare the way by reviewing the local asymptotic distribution theory for the unconstrained estimators $\widehat{f}_n$ and $\widehat{\varphi}_n$ developed by \cite{BRW2007LCasymp}
and asymptotic theory for $\widehat{f}_n^0$ and $\widehat{\varphi}_n^0$ developed by \cite{Doss-Wellner:2016ModeConstrained}.  These results are stated in Section~\ref{sec:estimator-asymptotics}.

Section~\ref{sec:PfSketches} contains an outline of our proof of
Theorem~\ref{LRasympNullDistribution} and the full proof of  Theorem \ref{LR-limit-FixedAlternative}.  The complete details of the long proof of Theorem~\ref{LRasympNullDistribution} are deferred to Subsections \ref{ssec:ProofsLocalRTs} and \ref{ssec:GlobalRTs}.
In Subsection \ref{ssec:ProofsLocalRTs} we treat remainder terms in a local neighborhood of the mode $m$, while remainder terms away from the mode are treated in Subsection \ref{ssec:GlobalRTs}.  Our proofs in Subsections \ref{ssec:ProofsLocalRTs} and \ref{ssec:GlobalRTs} rely heavily on the theory developed for the constrained estimators in \cite{Doss-Wellner:2016ModeConstrained} and on the new uniform consistency results for the constrained estimator presented in
Section~\ref{subsec:UniformConsist-Rates}
(with proofs in Section~\ref{ssec:ConsistencyProofs}).

In Section 5 we present Monte-Carlo estimates of quantiles of the distribution of $\DD$ and provide empirical evidence supporting
the universality of the limit distribution (under the assumption that $\varphi_0^{(2)}(m) < 0$).  We  illustrate the likelihood
ratio confidence sets with
Monte Carlo evidence demonstrating the
coverage probabilities of our proposed intervals are near the nominal levels.
Further simulation studies and application to two data sets can be found in
\cite{Doss-Wellner:2016ModeInference-arxiv-v2}.
Section 6 gives a brief description of further problems and potential developments.
We also discuss connections with the results of
\cite{MR947566},    %
\cite{MR964293},    %
\cite{MR1105839},  %
and
\cite{MR1671670, MR1747496}.  %
In Subsection~\ref{sec:notation} we discuss notation and terminology.

\subsection{Notation and terminology}
\label{sec:notation}

Several classes of concave functions will play a central role in this paper.
\begin{eqnarray}
  {\cal C} := \{ \varphi : \ \RR \rightarrow [-\infty, \infty) \ | \ \varphi \ \ \mbox{is concave, closed, and proper} \}
  \label{UnconstrainedConcaveClass}
\end{eqnarray}
and, for any fixed $ m \in \RR$,
\begin{eqnarray}
  {\cal C}_m := \{ \varphi \in {\cal C} \ | \ \varphi(m) \ge \varphi (x) \ \ \mbox{for all} \ \ x \in \RR \} .
  \label{ConstrainedConcaveClass}
\end{eqnarray}
Here proper and closed concave functions are as defined in
\cite{MR0274683}, %
pages 24 and 50.  We will follow  the convention that all concave functions $\varphi$
are defined on all of $\RR$ and take the value $-\infty$ off of their effective domains  $\mbox{dom}(\vp)$  where
$\mbox{dom} (\varphi ) := \{ x \ : \ \varphi (x) > -\infty \}$
(\cite{MR0274683}, %
page 40). %
Recall from the previous section that  %
the classes of unconstrained and constrained log-concave densities are then
\begin{eqnarray*}
  && {\cal P} := \left \{ e^{\varphi}  \ : \ \int e^{\varphi} d \lambda = 1, \ \ \varphi  \in {\cal C} \right \} , \ \ \ \mbox{and}\\
  && {\cal P}_m := \left \{ e^{\varphi} \ : \ \int e^{\varphi} d \lambda =1, \ \ \varphi \in {\cal C}_m \right \}
\end{eqnarray*}
where $\lambda$ is Lebesgue measure on $\RR$.
We let $X_1, \ldots , X_n$ be the observations, independent and identically distributed with density
$f_0$ with respect to Lebesgue measure.
Here we assume throughout that $f_0 \in {\cal P}$ and frequently that $f_0 = e^{\varphi_0} \in {\cal P}_m$
for some $m\in \RR$.
We let $X_{(1)} < \cdots < X_{(n)}$ denote the order statistics of the $X_i$'s,
let $\PP_n = n^{-1} \sum_{i=1}^n \delta_{X_i}$ denote the empirical measure, and let
$\FF_n (x) = n^{-1} \sum_{i=1}^n 1_{(-\infty,x]} (X_i )$ denote the empirical distribution function.
We define the log-likelihood criterion function $\Psi_n : {\cal C} \rightarrow \RR$ by
\begin{eqnarray}
  \Psi_n (\varphi )
  = \frac{1}{n} \sum_{i=1}^n \varphi (X_i)  - \int_{\RR} e^{\varphi (x)} dx
  = \PP_n \varphi  - \int_{\RR} e^{\varphi} d\lambda
  \label{eqn:AdjLogLikCriterion}
\end{eqnarray}
where we have used the standard device of including the Lagrange term $\int_{\RR} e^{\varphi (x)} dx$
in $\Psi_n$ so that
we can maximize $\Psi_n$ over all concave functions ${\cal C} $ or ${\cal C}_m$
(rather than maximizing over classes corresponding to density functions).
This is as in
\cite{MR663433}.  %
We will denote the unconstrained MLEs of $\varphi_0$,  $f_0$, and $F_0$ by
$\widehat{\varphi}_n$, $\widehat{f}_n$, and $\widehat{F}_n$ respectively.
These
exist uniquely by Proposition~1 of \cite{MR1941467}. %
The corresponding
constrained estimators with mode $m$ will be denoted by
$\widehat{\varphi}_n^0$, $\widehat{f}_n^0$, and $\widehat{F}_n^0$.
These
exist uniquely by
Theorem 2.6 of \cite{Doss-Wellner:2016ModeConstrained}
(or Lemma~2.0.3 of
\cite{Doss:2013}). %
Thus
\begin{eqnarray*}
  \widehat{\varphi}_n \equiv \argmax_{\varphi \in {\cal C}} \Psi_n (\varphi),\ \ \ \mbox{and} \ \ \
  \widehat{\varphi}_n^0  \equiv \argmax_{\varphi \in {\cal C}_m } \Psi_n (\varphi) .
\end{eqnarray*}

\section{Uniform consistency and rates}
\label{subsec:UniformConsist-Rates}

Here we recall the uniform  rate-consistency theorem of \cite{DR2009LC}, and give a
partial analogue
for the mode-constrained MLE.
The new result,
given in Theorem~\ref{thm:GlobalConsistencyWithRates}
Part~\ref{thm:GlobalConsistencyWithRatesConstrained} below,
is of interest in its own right for describing the theoretical behavior of the mode-constrained MLE.  Additionally,
the proof of Theorem~\ref{LRasympNullDistribution} relies on (both parts of) Theorem~\ref{thm:GlobalConsistencyWithRates}.
It should be mentioned that
Theorem~\ref{thm:GlobalConsistencyWithRates}
Part~\ref{thm:GlobalConsistencyWithRatesConstrained}
is a non-trivial extension of the theorem of  \cite{DR2009LC}, with a fairly difficult proof.

To state the uniform results we define ${\cal H}^{\beta, L} (I)$ to be the collection of real-valued functions
$g$ on the closed interval $I$ satisfying
$| g(y) - g(x) | \le L | y -x |$ if $\beta =1$ and $| g' (y) - g' (x) | \le L | y - x|^{\beta-1}$
if $\beta > 1$, for all $x,y \in I$.    We let $\rho_n \equiv n^{-1} \log n$.

\begin{theorem}
  \label{thm:GlobalConsistencyWithRates}
  (Uniform consistency and rates of convergence.)
  \begin{enumerate}[label=\Alph*.,ref=\Alph*,leftmargin=*]
  \item  \label{thm:GlobalConsistencyWithRatesUnconstrained}
    (\cite{DR2009LC})    Suppose that $f_0 \in {\cal LC }$. %
    If $\varphi_0 \in {\cal H}^{\beta, L}(K)$ for some $1 \le \beta \le 2$, $L>0$, and
    $K = [b,c] \subset \mbox{int} (\{ f_0 > 0 \} )$, then
    \begin{align}
      \sup_{t \in K} ( \widehat{\varphi}_n - \varphi_0 )(t) & =  O_p ( \rho_n^{\beta/(2\beta+1)}) ,  \ \ \ \mbox{and}  \label{DRURup}\\
      \sup_{t \in K_n} (\varphi_0 - \widehat{\varphi}_n )(t) & = O_p ( \rho_n^{\beta/(2\beta+1)})   \label{DRURdown}
    \end{align}
    where $K_n \equiv [b + \rho_n^{1/(2\beta+1)} , c - \rho_n^{1/(2\beta+1)}]$.
    These results remain true when $\widehat{\varphi}_n$ is replaced by $\widehat{f}_n$ and $\vvo$ by $\ffo$.
  \item  \label{thm:GlobalConsistencyWithRatesConstrained}
    Suppose that $f_0 \in {\cal LC}_m$,  $\varphi_0 \in {\cal H}^{2,
      L}(K)$ for some $L>0$, $\vp_0^{(2)}(m) < 0$, and $K = [b,c] \subset
    \mbox{int} (\{ f_0 > 0 \} )$. Then the results of Part~A hold true with
    $\beta = 2$, with $\widehat{\varphi}_n$ replaced by
    $\widehat{\varphi}_n^0$ and with $\ffn$ replaced by $\ffna$.
  \end{enumerate}
\end{theorem}

\noindent
The proof of Theorem~\ref{thm:GlobalConsistencyWithRates} is given in Appendix~\ref{ssec:ConsistencyProofs}.

\section{Unconstrained and Constrained local limit processes}
\label{sec:estimator-asymptotics}

The limit distribution of $2 \log \lambda_n$, under
the hypotheses of Theorem~\ref{LRasympNullDistribution},
depends on the joint distribution of $\vvn(x)$ and $\vvna(x)$ at points
 $x$ in $n^{-1/5}$-neighborhoods of $m$.
In proving Theorem~\ref{LRasympNullDistribution} it is also helpful to know that $\vvn(x)$ and $\vvna(x)$ are asymptotically equivalent at fixed $x \ne m$ when $\vp_0^{\prime \prime}(x)<0$ and $M(f_0)=m$.
Thus, in this section we recall the limit distributions of $\vvn$ and $\vvna$ from
Theorem 2.1 of \cite{BRW2007LCasymp}
and
Theorems~5.5 and  5.7 (see also Theorem~5.8) of   \cite{Doss-Wellner:2016ModeConstrained}. %
The process giving the limit distribution of $\vvn$ was first studied by
\cite{MR1891741}.   %
Here are the assumptions we will need.

\begin{assumption}
  \label{CurvatureAtTheMode}
  (Curvature at $m$)  Suppose that $X_1, \ldots , X_n$ are i.i.d. $f_0 = e^{\varphi_0} \in {\cal P}_m$
  and that  $\varphi_0$ is twice continuously differentiable at $m$ with
  $\varphi_0^{\prime \prime} (m) < 0$.
\end{assumption}

\begin{assumption}
  \label{CurvatureAwayFromMode}
  (Curvature at $x_0 \not= m$)  Suppose that $X_1, \ldots , X_n$ are i.i.d. $f_0 = e^{\varphi_0} \in {\cal P}_m$
  and that $\varphi_0$ is twice continuously differentiable at $x_0 \not= m$ with
  $\varphi_0^{\prime \prime} (x_0) < 0$ and $f_0 (x_0) > 0$.
\end{assumption}

\begin{theorem}[\cite{BRW2007LCasymp}, \cite{Doss-Wellner:2016ModeConstrained}]
  \label{JointLimitingDistributions}
  A.  (At a point $x_0 \not= m$).
  Suppose that $\varphi_0$ and $f_0$ satisfy Assumption~\ref{CurvatureAwayFromMode}.
  Then
  \begin{eqnarray*}
    \left ( \begin{array}{l} n^{2/5} ( \widehat{\varphi}_n (x_0) - \varphi_0 (x_0)) \\
        n^{2/5} ( \widehat{\varphi}_n^0 (x_0) - \varphi_0 (x_0))
      \end{array} \right ) \rightarrow_d
    \left ( \begin{array}{l} \VV \\ \VV \end{array} \right )
  \end{eqnarray*}
  where $\VV \equiv C(x_0 , \varphi_0 ) H^{(2)} (0)$,
  where $H$ is described in
  Theorem~5.1  of
  \cite{Doss-Wellner:2016ModeConstrained}, %
 and where $C(x_0, \varphi_0 )$ is
  as given in (\ref{GammaRelationsPart2}) (but with $m$ replaced by $x_0$):
  \begin{equation*}
        C(x_0 , \varphi_0 ) = \left ( \frac{ | \varphi_0^{(2)} (x_0)|}{4!f_0 (x_0)^2} \right )^{1/5} .
    \label{ConstantReducToCanonical}
  \end{equation*}
  Consequently
  \begin{eqnarray*}
    n^{2/5} ( \widehat{\varphi}_n (x_0) - \widehat{\varphi}_n^0 (x_0) ) \rightarrow_p 0 .
  \end{eqnarray*}
  B.  (In $n^{-1/5}-$neighborhoods of $m$)
  Suppose $\varphi_0$ and $f_0$ satisfy Assumption~\ref{CurvatureAtTheMode}.
  Define processes
  $\XX_n$ and $\XX_n^0 $ by
  \begin{eqnarray*}
    \XX_n (t) \equiv n^{2/5} (\widehat{\varphi}_n (m + n^{-1/5} t) - \varphi_0 (m) ) \\
    \XX_n^0 (t) \equiv  n^{2/5} ( \widehat{\varphi}_n^0 (m + n^{-1/5} t) - \varphi_0 (m)).
  \end{eqnarray*}
  Then the finite-dimensional distributions of $(\XX_n (t), \XX_n^0 (t))$ converge in distribution to the
  finite-dimensional distributions of the processes
  \begin{eqnarray*}
    (\widehat{\varphi}_{a,\sigma} (t), \widehat{\varphi}_{a,\sigma}^0 (t) )\stackrel{d}{=}
    \frac{1}{\gamma_1 \gamma_2^2} ( \widehat{\varphi} (t/\gamma_2) , \widehat{\varphi}^0 (t/\gamma_2) )
    \equiv (\XX (t), \XX^0 (t))
  \end{eqnarray*}
  where $H, H^0, H^{(2)} = \widehat{\varphi}$, and $(H^0)^{(2)} = \widehat{\varphi}^0$
  are as described in
  Theorems~5.1 and 5.2 of
  \cite{Doss-Wellner:2016ModeConstrained}, %
  and $\gamma_i$, $i=1,2$, is described in Subsection~\ref{ssec:PrepLimitProcessesScaling}.
  Furthermore, for $p\ge 1$
  $$
  (\XX_n (t) , \XX_n^0 (t) ) \rightarrow_d (\XX(t) , \XX^0(t) )
  $$
  in ${\cal L}^p [-K,K] \times {\cal L}^p [-K,K]$ for each $K>0$.
\end{theorem}

\section{Proof sketches for Theorems~\ref{LRasympNullDistribution} and \ref{LR-limit-FixedAlternative}}
\label{sec:PfSketches}

Now we present proof sketches for our two main theorems (and make use of the results in the previous two sections).

\subsection{Proof sketch for  Theorem~\ref{LRasympNullDistribution}}
\label{ssec:PfSketchThm1}

To begin our sketch of the proof of Theorem~\ref{LRasympNullDistribution} we first give the basic decomposition we will use.
We begin by using $\int_{\RR} \widehat{f}_n (u) du = 1 = \int_{\RR} \widehat{f}_n^0 (u) du$ to write
\begin{eqnarray}
  2 \log \lambda_n
  & = & 2 n \PP_n (\widehat{\varphi}_n  - \widehat{\varphi}_n^0 ) \nonumber \\
  & = & 2 n \int_{\RR} ( \widehat{\varphi}_n - \widehat{\varphi}_n^0 ) d \FF_n - \int_{\RR} (\widehat{f}_n (u) - \widehat{f}_n^0 (u) ) du \nonumber \\
  & = & 2 n  \int_{[X_{(1)}, X_{(n)} ]}  \left ( \widehat{\varphi}_n d \widehat{F}_n  - \widehat{\varphi}_n^0  d \widehat{F}_n^0 \right )
  \label{LR-DecompStepone} \\
  && \qquad \  - \ 2 n \int_{[X_{(1)}, X_{(n)} ]}  \left ( e^{\widehat{\varphi}_n (u)} - e^{\widehat{\varphi}_n^0 (u)} \right ) du  \nonumber
\end{eqnarray}
where we have used the characterization Theorems 2.2 and 2.8 of
\cite{Doss-Wellner:2016ModeConstrained} with $\Delta = \pm \ \widehat{\varphi}_n$ and $\Delta = \pm \ \widehat{\varphi}_n^0$
respectively.  As we will see, inclusion of the second term in (\ref{LR-DecompStepone}) will be of considerable help
in the analysis.

Now we split the integrals in (\ref{LR-DecompStepone}) into two regions:
let
$D_n \equiv [t_{1}, t_{2}]$ for %
some
$t_{1} < m < t_{2}$
and then let $D_n^c \equiv [X_{(1)}, X_{(n)}] \setminus D_n$.  The set $D_n$
is the region containing the mode $m$; here the unconstrained estimator
$\widehat{\varphi}_n$ and  the constrained estimator
$\widehat{\varphi}_n^0$ tend to differ.  On the other hand, $D_n^c$ is the
union of two sets away from the mode, and on both of these sets the
unconstrained estimator $\widehat{\varphi}_n$ and the constrained estimator
$\widehat{\varphi}_n^0$ are asymptotically equivalent (or at least nearly
so).  Sometimes we will take the $t_i$, $i=1,2$, to be constant in $n$,
sometimes to be fixed or random sequences approaching $m$ as $n \to \infty$.
We will sometimes suppress the dependence of $D_n \equiv D_{n, t_{1}, t_{2}}$
on $t_{i}$, and will emphasize it when it is important.  Now, from
\eqref{LR-DecompStepone}, we can write
\begin{eqnarray*}
  2 \log \lambda_n
  & = & 2 n \left \{
    \int_{D_n} \left ( \widehat{\varphi}_n d \widehat{F}_n - \widehat{\varphi}_n^0  d \widehat{F}_n^0  \right )  \right .
  \left . -  \ \int_{D_n}  \left ( e^{\widehat{\varphi}_n (u)} - e^{\widehat{\varphi}_n^0 (u)}   \right ) du  \right . \\
  &&   \left . + \ \int_{D_n^c} \left ( \widehat{\varphi}_n d \widehat{F}_n - \widehat{\varphi}_n^0  d \widehat{F}_n^0  \right )  \right .
  \left .  -  \ \int_{D_n^c}  \left ( e^{\widehat{\varphi}_n (u)} - e^{\widehat{\varphi}_n^0 (u)} \right ) du  \right \} \\
  & = &  2 n \left \{
    \int_{D_n} \left ( ( \widehat{\varphi}_n - \varphi_0 (m)) d \widehat{F}_n
      - (\widehat{\varphi}_n^0  - \varphi_0 (m)) d \widehat{F}_n^0  \right )  \right . \\
  & &     \left . -  \ \int_{D_n}  \left ( ( e^{\widehat{\varphi}_n (u)} - e^{\varphi_0 (m)}) - ( e^{\widehat{\varphi}_n^0 (u)} - e^{\varphi_0(m)} )  \right ) du  \right \} \\
  && \hspace{2em} + \  2 n (R_{n,1} + R_{n,1}^c )
\end{eqnarray*}
where
\begin{align}
  R_{n,t_{1},t_2}    \equiv  & R_{n,1} \equiv \int_{D_n}\varphi_0 (m)  ( \widehat{f}_n (x) - \widehat{f}_n^0 (x) ) dx ,  \label{RemainTermOne}\\
  R_{n,t_{1},t_2}^c    \equiv & R_{n,1}^c \equiv \int_{D_n^c}  \left ( \widehat{\varphi}_n d \widehat{F}_n - \widehat{\varphi}_n^0  d \widehat{F}_n^0  \right )
  -  \ \int_{D_n^c}  \left ( e^{\widehat{\varphi}_n (u)} - e^{\widehat{\varphi}_n^0 (u)} \right ) du . \label{RemainTermOneCompl}
\end{align}
Now we use an expansion of the exponential function to rewrite the second part of the main term:
since
$$
e^b - e^a =  e^{a} \{ e^{b-a} -1 \} = e^a \{ (b-a) + \frac{1}{2} (b-a)^2 + \frac{1}{6} e^{a^*} (b-a)^3 \}
$$
where $|a^* - a | \le | b -a |$, we have
\begin{eqnarray*}
  \lefteqn{\int_{D_n}  \left ( ( e^{\widehat{\varphi}_n (u)} - e^{\varphi_0 (m)} ) - ( e^{\widehat{\varphi}_n^0 (u)} - e^{\varphi_0(m)} )  \right ) du } \\
  & = & \int_{D_n}  e^{\varphi_0 (m)} \left ( (\widehat{\varphi}_n (u) - \varphi_0 (m))
    + \frac{1}{2} ( \widehat{\varphi}_n (u) - \varphi_0(m))^2 )  \right ) du   + R_{n,2} \\
  &&  \ - \  \int_{D_n}  e^{\varphi_0 (m)} \left ( (\widehat{\varphi}_n^0 (u) - \varphi_0 (m))
    + \frac{1}{2} ( \widehat{\varphi}_n^0 (u) - \varphi_0(m))^2 )  \right ) du   -  R_{n,2}^0
\end{eqnarray*}
where
\begin{eqnarray}
  && R_{n,2} \equiv \int_{D_n} \frac{1}{6} f_0 (m) e^{\tilde{x}_{n,2} (u)} (\widehat{\varphi}_n (u) - \varphi_0 (m))^3 du,
  \label{RemainTermTwoUnconstrained}\\
  && R_{n,2}^0 \equiv  \int_{D_n} \frac{1}{6} f_0 (m) e^{\tilde{x}_{n,2}^0 (u)} (\widehat{\varphi}_n^0 (u) - \varphi_0 (m))^3 du .
  \label{RemainTermTwoConstrained}
\end{eqnarray}
Thus
\begin{eqnarray}
  \lefteqn{2 \log \lambda_n }    \label{LR-SecondDecomposition}\\
  & = &   n \left \{
    2 \int_{D_n} \left ( ( \widehat{\varphi}_n (u)- \varphi_0 (m))  (d \widehat{F}_n(u)  - f_0 (m) du) \right . \right . \nonumber \\
  && \qquad \left . \left .
      - \ 2\int_{D_n} (\widehat{\varphi}_n^0 (u)  - \varphi_0 (m)) ( d \widehat{F}_n^0 (u)  - f_0 (m) du)  \right )  \right .\nonumber  \\
  &&  \qquad \left .
    -  \ \int_{D_n}  \left ( (\widehat{\varphi}_n (u) - \varphi_0 (m))^2 - (\widehat{\varphi}_n^0 (u) - \varphi_0(m) )^2  \right ) f_0(m) du  \right \}
  \nonumber \\
  && + \  2 n (R_{n,1} + R_{n,1}^c + R_{n,2} - R_{n,2}^0 ).  \nonumber
\end{eqnarray}
Now we expand the first two terms in the last display, again using a two term expansion,
$e^{b} - e^{a} = (e^{b-a} -1)e^a = e^a \{ (b-a) + \frac{1}{2} (b-a)^2 e^{a^*} \}$, to find that
\begin{eqnarray*}
  \lefteqn{\int_{D_n}  ( \widehat{\varphi}_n (u)- \varphi_0 (m))  (d \widehat{F}_n(u)  - f_0 (m) du )  }\\
  & = & \int_{D_n} ( \widehat{\varphi}_n (u)- \varphi_0 (m)) ( e^{\widehat{\varphi}_n (u)  - \varphi_0 (u)} -1 ) f_0 (m) du \\
  & = & \int_{D_n} ( \widehat{\varphi}_n (u)- \varphi_0 (m)) \left ( ( \widehat{\varphi}_n (u)- \varphi_0 (m))
    + e^{\tilde{x}_{n,3} (u)} \frac{1}{2} ( \widehat{\varphi}_n (u)- \varphi_0 (m))^2 \right ) f_0(m) du \\
  & = & \int_{D_n} ( \widehat{\varphi}_n (u)- \varphi_0 (m))^2 f_0 (m) du + R_{n,3}
\end{eqnarray*}
where
\begin{eqnarray}
  R_{n,3,t_1,t_2} \equiv
  R_{n,3} = \int_{D_n} \frac{1}{2} f_0 (m) e^{\tilde{x}_{n,3} (u)} ( \widehat{\varphi}_n (u)- \varphi_0 (m))^3 du .
  \label{RemainTermThreeUnconstrained}
\end{eqnarray}
Similarly,
\begin{eqnarray*}
  \lefteqn{
    \int_{D_n}  ( \widehat{\varphi}_n^0  (u)- \varphi_0 (m))  (d \widehat{F}_n^0 (u)  - f_0 (m) du )  }\\
  & = & \int_{D_n} ( \widehat{\varphi}_n^0 (u)- \varphi_0 (m))^2 f_0 (m) du + R_{n,3}^0
\end{eqnarray*}
where
\begin{eqnarray}
  R_{n,3,t_1,t_2}^0 \equiv
  R_{n,3}^0 = \int_{D_n} \frac{1}{2} f_0 (m) e^{\tilde{x}_{n,3}^0 (u)} ( \widehat{\varphi}_n^0 (u)- \varphi_0 (m))^3 du .
  \label{RemainTermThreeConstrained}
\end{eqnarray}
If we let $t_1 = t_{n,1} = m - bn^{-1/5}$ and $t_2 = t_{n,2} = m+ bn^{-1/5}$
for $ b > 0$, then from (\ref{LR-SecondDecomposition}) we now have
\begin{eqnarray}
  2 \log \lambda_n
  & = & n \int_{D_n} f_0 (m) \left \{
    (\widehat{\varphi}_n (u) - \varphi_0 (m))^2 - (\widehat{\varphi}_n^0 (u) - \varphi_0(m) )^2  \right  \}  du \nonumber \\
  && \qquad + \ 2 n \left (R_{n,1} + R_{n,1}^c + R_{n,2} - R_{n,2}^0 + R_{n,3} - R_{n,3}^0 \right ) \nonumber \\
  & \equiv & \DD_{n,b} +  R_n .
  \label{BasicDecompositionFirstForm}
\end{eqnarray}
Now we sketch the behavior of $\DD_{n,b}$.
Let $u = m + v n^{-1/5}$;  with this change of variables and the definition of $t_{n,i}$, $i=1,2$, we can rewrite
$\DD_{n,b}$
as
\begin{align*}
  \MoveEqLeft  \DD_{n,b}
  = f_0(m) n^{4/5}  \int_{-b}^{b}
  \bigg \{
  \left (\widehat{\varphi}_n (m + n^{-1/5} v) - \varphi_0 (m) \right )^2 \\
  & \hspace{3.0cm}  - \left (\widehat{\varphi}_n^0 (m + n^{-1/5} v) - \varphi_0 (m) \right )^2
  \bigg \}  dv.
\end{align*}
By Theorem~\ref{JointLimitingDistributions}B this converges in distribution to
\begin{eqnarray}
  f_0 (m) \int_{-b}^{b}  \left \{ \left (\widehat{\varphi}_{a,\sigma} (v) \right )^2 - \left (\widehat{\varphi}_{a,\sigma}^0 (v) \right )^2 \right \} dv ,
  \label{LimitInDistributionMainTerm}
\end{eqnarray}
where the processes $(\widehat{\varphi}_{a,\sigma} , \widehat{\varphi}_{a,\sigma}^0 )$
are related to $ ( \widehat{\varphi} , \widehat{\varphi}^0)$ by the scaling relations (\ref{ScalingRelationUnConstrained})
and (\ref{ScalingRelationConstrained}).
We conclude that the limiting random variable in (\ref{LimitInDistributionMainTerm}) is equal in distribution to
\begin{eqnarray}
  \lefteqn{f_0 (m) \int_{-b}^{b} \left \{ \left ( \frac{1}{\gamma_1 \gamma_2^2}
        \widehat{\varphi} (  v/\gamma_2) \right )^2
      - \left ( \frac{1}{\gamma_1
          \gamma_2^2}
        \widehat{\varphi}^0 (
        v/\gamma_2) \right )^2
    \right \} dv } \nonumber \\
  & = & \frac{f_0 (m)}{\gamma_1^2 \gamma_2^3} \frac{1}{\gamma_2 }
  \int_{-b}^{b} \left \{ \widehat{\varphi} ( v/\gamma_2)^2
    -   \widehat{\varphi}^0
    ( v/\gamma_2)^2 \right
  \} dv \nonumber  \\
  & = & \int_{-b/\gamma_2}^{b/\gamma_2}  \left \{ \widehat{\varphi} (s)^2
    -   \widehat{\varphi}^0
    (s)^2 \right \}
  ds \label{eq:DDb-formula-gammas}
\end{eqnarray}
in view of (\ref{GammaRelationsPart1}).
This is not yet free of the parameter $\gamma_2$, but it will become so if we
let $b \to \infty$. %
If we show that this is permissible and we show that the remainder term $R_n$ in (\ref{BasicDecompositionFirstForm})
is negligible,   then the proof of
Theorem~\ref{LRasympNullDistribution} will be complete.
For details, see section~\ref{sec:ProofsPart2}.

\medskip

\begin{remark}
  \label{rem:rem:DD-formula} %
  As is suggested by \eqref{eq:DDb-formula-gammas}
  (and proved in section~\ref{sec:ProofsPart2}),
  the form of the random variable $\DD$ from Theorem~\ref{LRasympNullDistribution} is
  \begin{equation*}
    \DD = \int_{-\infty}^\infty \lb \widehat{\varphi}(u)^2 - \widehat{\varphi}^0 (u)^2 \rb \, du.
  \end{equation*}
  The form of this random variable is the same as that found in \cite{MR1891743} %
  and \cite{MR2341693}, if we replace our $\widehat{\varphi}$ and $\widehat{\varphi}^0$ with the corresponding random functions studied in the monotone case.
\end{remark}

\subsection{Proof of Theorem~\ref{LR-limit-FixedAlternative}}
\label{ssec:PfSketchThm2}

Recall
${\cal P}_m = \lb f \in {\cal P} : M(f) = m  \rb$.
We now assume that $f \in {\cal P} \setminus {\cal P}_m$.
Let  %
$\lambda$ be Lebesgue measure and let
\begin{align}
  f_m^0
  & = \mbox{argmin}_{g \in {\cal P}_m} \left \{ - \int \log g f_0 d\lambda \right \}
  \label{eq:defn:fm0} \\
  & =  \mbox{argmin}_{g\in {\cal P}_m} \int f_0 \left ( \log f_0 - \log g \right ) d \lambda
  =  \mbox{argmin}_{g \in {\cal P}_m} K(f_0,g) , \nonumber
\end{align}
where we will make
\eqref{eq:defn:fm0}
rigorous later, in Theorem~\ref{DSS-thm2.2m}.
Let $\PP_n = \sum_{i=1}^n \delta_{X_i} /n$ be the empirical measure and for a function $g$ let $\PP_n(g) = \int g d\PP_n$.
We now have
\begin{align}
  n^{-1} \log \lambda_n (m)
  & =  \PP_n ( \log \widehat{f}_n / \widehat{f}_n^0 )
  =   \PP_n \left \{ \log \frac{\widehat{f}_n}{f_0} \cdot \frac{f_0}{f^0_m}\cdot \frac{f^0_m}{\widehat{f}_n^0} \right \} \nonumber \\
  &=  \PP_n \left \{ \log \frac{\widehat{f}_n}{f_0}  \right \} +  \PP_n \left \{ \log \frac{f_0}{f_m^0} \right \}
  - \PP_n \left \{ \log \frac{\widehat{f}_n^0 }{f^0_m} \right\}. \label{eq:TLLR-alt-3terms}
\end{align}
From this we can conclude that
as $n \to \infty$,
\begin{align*}
  n^{-1} 2 \log \lambda_n(m)  =  O_p (n^{-4/5})
  +  2 \PP_n \left \{ \log \frac{f_0}{f^0_m} \right \}  - o_p (1)
  \rightarrow_p 2 K(f_0, f_m^0).
\end{align*}
That
$\PP_n \left \{ \log \frac{\widehat{f}_n}{f_0}  \right \} = O_p(n^{-4/5})$
follows from \cite{DossWellner:2016a}, Corollary 3.2, page 962.
The  convergence of $2 \PP_n \left \{ \log \frac{f_0}{f^0_m} \right \}$
to $K(f_0, f_m^0)$ follows from the weak law of large numbers.
The indicated negligibility of the third term
in \eqref{eq:TLLR-alt-3terms} follows from Theorem~\ref{DSS-thm2.15m} below
(which is a constrained analogue of Theorem 2.15 of \cite{DSS2011LCreg}).

It remains to justify
the definition given in \eqref{eq:defn:fm0},
and to show that the third term of \eqref{eq:TLLR-alt-3terms} is $o_p(1)$,
under the assumptions of Theorem~\ref{LR-limit-FixedAlternative}.
We first state three theorems.  These are mode-constrained analogues of Theorems 2.2, 2.7, and 2.15 of
\cite{DSS2011LCreg},
and are proved with methods similar to the methods used in
\cite{DSS2011LCreg}.  The full proofs will be given in a separate work on estimation and inference for modal regression functions.

Now we set
\begin{eqnarray*}
  L(\varphi , Q) \equiv \int \varphi dQ - \int e^{\varphi} d\lambda +1
\end{eqnarray*}
and define
\begin{eqnarray*}
  L_m (Q) \equiv \sup_{\varphi \in {\cal C}_m}  L( \varphi , Q)
\end{eqnarray*}
where
\begin{eqnarray*}
  {\cal C}_m
  \equiv \{ \varphi  :  \ m \in MI (\varphi), \varphi \mbox{ concave} \}
\end{eqnarray*}
and  recall
$MI (\varphi ) \equiv \{ x \in \RR : \ \varphi (x) = \sup_{y \in \RR} \varphi (y) \} .$
If for fixed $Q$ there exists $\psi_m \in {\cal C}_m$ such that
$$
L(\psi_m , Q ) = L_m (Q) \in \RR,
$$
then $\psi_m$ will automatically satisfy $\int \exp ( \psi_m (x) ) dx = 1$:
note that $\phi+c \in {\cal C}_m$
for any fixed $\phi \in {\cal C}_m$ and $c \in \RR$.  On the other hand,
\begin{eqnarray*}
  \frac{\partial}{\partial c} L({\cal C} +c , Q)
  & = & \frac{\partial}{\partial c} \left \{ \int ( \phi +c) dQ - e^c \int e^{\phi} d\lambda +1 \right \} \\
  & = & 1 - e^c \int e^{\phi} d \lambda
\end{eqnarray*}
if $L ( \phi , Q) \in \RR$.  Thus $L(\phi+c, Q)$ is maximal for $c = - \log \int e^{\phi} d\lambda $. \\
For the next theorem we need to define
\begin{equation*}
  \mbox{csupp}_m (Q) =  %
  \bigcap   \lb C \subseteq \RR^d : C \mbox{ closed, convex }, Q(C) = 1 , m \in C\rb.
\end{equation*}

\medskip

\begin{theorem}
  \label{DSS-thm2.2m}
  Let $Q$ be a measure on $\RR^d$. %
  The value of $L_m (Q)$ is real if and only if $\int |x | dQ (x) < \infty$ and
  $\mbox{int}( \mbox{csupp}_m (Q) ) \neq \emptyset$.
  In that case there exists a unique
  $\psi_m \equiv \psi_m (\cdot | Q) \in \mbox{argmax}_{\varphi \in {\cal C}_m} L( \varphi , Q)$.
  This function $\psi_m$ satisfies $\int e^{\psi_m} d \lambda =1$ and
  $$
  \mbox{int}( \mbox{csupp}_m (Q) )
  \subseteq \mbox{dom} (\psi_m) \subseteq  \mbox{csupp}_m (Q)
  $$
  where $\mbox{dom} (\psi_m) \equiv \{ x \in \RR^d : \ \psi_m (x) > -\infty \} $.
\end{theorem}
\medskip

Theorem~\ref{DSS-thm2.2m} justifies rigorously the definition given in \eqref{eq:defn:fm0}, since $f_0 \in {\cal P}$ has a finite mean and is non-degenerate.
Now, for $\psi_m \in {\cal C}_m $, let
\begin{eqnarray*}
  {\cal S}(\psi_m) =
  \{ x \in \mbox{dom} ( \psi_m ) : \ \psi_m (x) > 2^{-1} \left ( \psi_m (x- \delta) + \psi_m (x+ \delta) \right ) \ \ \mbox{for all} \ \delta > 0 \} ,
\end{eqnarray*}
and
\begin{align*}
  {\cal S}_L(\psi_m) &= \lb x \in {\cal S}(\psi_m) : \psi_m'(x-) > 0  \rb  \subseteq (-\infty, m], \\
  {\cal S}_R(\psi_m) &= \lb x \in {\cal S}(\psi_m) : \psi_m'(x+) < 0  \rb  \subseteq [m, \infty).
\end{align*}
It is possible for $m$ to be an element of either of the sets ${\cal S}_L(\psi_m)$ and
${\cal S}_R(\psi_m)$ without being a member of the other.
The following theorem is the population analogue of Theorem 2.10 of \cite{Doss-Wellner:2016ModeConstrained}.
\medskip

\begin{theorem}
  \label{DSS-thm2.7m}
  Let $Q$ be a non-degenerate distribution on $\RR$ with finite first moment and distribution
  function $G$.  Let $F_m$ be a distribution function with log-density $\varphi_m \in {\cal C}_m$.
  Then $\varphi_m = \psi( \cdot | Q)$ if and only if
  \begin{eqnarray}
    && \int_{- \infty}^x F_m (y) dy \le \int_{-\infty}^x G (y) dy \ \ \ \mbox{for all} \ \ x \le m, \ \ \mbox{and} \label{eq:pop-charzn-L} \\
    &&  \int_{x}^\infty (1-F_m (y)) dy \le \int_x^{\infty} (1-G (y)) dy  \ \ \ \mbox{for all} \ \ x \ge m , \label{eq:pop-charzn-R}
  \end{eqnarray}
  with equality in \eqref{eq:pop-charzn-L} if $x \in {\cal S}_L(\varphi_m)$ and equality in
  \eqref{eq:pop-charzn-R} if $x \in {\cal S}_R(\varphi_m)$.
\end{theorem}
\medskip

Thus, again much as in \cite{DSS2011LCreg}, for $x \in {\cal S} ( \psi_m (\cdot  | Q)) $, $x \le m$, and (small) $\delta>0$,
\begin{eqnarray*}
  && 0 \le \frac{1}{\delta} \int_{x-\delta}^x (F_m (y) - G(y)) dy \rightarrow F_m (x) - G(x-) \ \ \mbox{as} \ \delta\searrow 0, \\
  && 0 \ge \frac{1}{\delta} \int_x^{x+\delta} F_m (y) - G(y)) dy \rightarrow F_m (x) - G(x) \ \ \mbox{as} \ \ \delta\searrow 0,
\end{eqnarray*}
and hence $G(x-) \le F_m (x) \le G(x)$ for all $x \in \RR$.

Now we need to understand the properties of the maps $Q \mapsto L_m (Q)$ and $Q \mapsto \psi_m( \cdot | Q)$
on ${\cal Q}_1 \cap {\cal Q}_{0,m}$,
where we let ${\cal Q}_1  = \lb Q  : \int | x | dQ < \infty \rb$
and ${\cal Q}_{0,m} = \lb Q : \mbox{int} (\mbox{csupp}_m (Q)) \ne  \emptyset \rb$.
As in \cite{DSS2011LCreg} we show that these are both continuous with respect to Mallows
distance $D_1$:
$$
D_1 (Q, Q') \equiv \inf_{(X,X')} E | X - X' |
$$
where the infimum is taken over all pairs $(X,X')$ of random variables $X \sim Q$ and $X' \sim Q'$ on a common
probability space.  Convergence of $Q_n$ to $Q$ in Mallows distance is equivalent to
having $\int |x| dQ_n \to \int |x| dQ$ and $Q_n \Rightarrow Q$
\citep{Mallows:1972vy}.
\medskip

\begin{theorem}
  \label{DSS-thm2.15m}
  Let $\{ Q_n \}$ be a sequence of distributions on $\RR^d$ such that $D_1 (Q_n , Q) \rightarrow 0$
  for some $Q \in {\cal Q}_1$.  Then
  $$L_m (Q_n )\rightarrow L_m(Q).$$
  If $Q \in {\cal Q}_{0,m} \cap {\cal Q}_1$ the probability densities
  $f^0 \equiv \exp ( \psi_m (\cdot | Q))$ and
  $f_n^0 \equiv \exp ( \psi_m (\cdot | Q_n))$
  are well-defined for large $n$ and satisfy
  \begin{eqnarray*}
    \lim_{n \rightarrow \infty, x \rightarrow y} f_n^0 (x) & = & f^0 (y) \ \ \mbox{for all} \ \ y \in \RR^d \setminus \partial \{ f^0 > 0 \}, \\
    \limsup_{n\rightarrow \infty , x \rightarrow y} f_n^0 (x) & \le &f^0(y) \ \ \mbox{for all} \ \ y \in \partial \{ f^0 > 0 \} , \\
    \int | f_n^0 (x) - f^0 (x) | dx & \rightarrow & 0 .
  \end{eqnarray*}
\end{theorem}

We can now show that the third term of \eqref{eq:TLLR-alt-3terms} is $o_p(1)$ under the assumptions of Theorem~\ref{LR-limit-FixedAlternative}.
First note that $\PP_n $ converges weakly to $P$, the measure corresponding to $f_0 \in {\cal P}$, with
probability $1$ and $\PP_n (|x|) = n^{-1} \sum_{i=1}^n |X_i |  \rightarrow_{a.s.} \int |x| dP(x)$ by the strong law of
large numbers.  Thus $D_1 (\PP_n ,P) \rightarrow_{a.s.} 0$.
It follows from Theorem~\ref{DSS-thm2.15m} that
\begin{eqnarray*}
  \log \widehat{f}_n^0
  & = & \mbox{argmax} _{\varphi \in {\cal C}_m} \left \{ \PP_n \varphi - \int e^{\varphi} d \lambda +1 \right \} \\
  & = & \psi_m (\cdot | \PP_n ) = \widehat{\varphi}_n^0
\end{eqnarray*}
where, by the last part of Theorem~\ref{DSS-thm2.15m}, $\int | \widehat{f}_n^0 - f_m^0 | d \lambda \rightarrow_{a.s.} 0$
and by the continuity
\begin{eqnarray}
  L_m (\PP_n)
  & = & L( \psi_m ( \cdot | \PP_n ) , \PP_n ) = \PP_n \log \widehat{f}_n^0 \nonumber \\
  & \rightarrow_{a.s.} & L ( \psi_m ( \cdot  | P), P) =  P( \log f_m^0  ).  \label{CruxContinuity}
\end{eqnarray}
But then
\begin{eqnarray*}
  \PP_n \log \frac{\widehat{f}_n^0 }{f_m^0 }
  & = & \PP_n \log \widehat{f}_n^0 - P \log f_m^0 - (\PP_n \log f_m^0 - P\log f_m^0 ) \\
  & \rightarrow_{a.s.} & 0 - 0 = 0
\end{eqnarray*}
by (\ref{CruxContinuity}) for the first term and the strong law of large numbers for the second term  (using that $-\infty < L_m(P) < \infty$ by Theorem~\ref{DSS-thm2.2m}).

\section{Simulations:  some comparisons and examples}
\label{sec:Illustration-Simul}
\subsection{Monte Carlo estimates of the distribution of $\ \DD$}
\label{subsec:DD}
To implement our likelihood ratio test and the corresponding new confidence intervals,
we first  conducted a Monte-Carlo study  of the null-hypothesis limit random variable, $\DD$.
We did this by simulating $M = 5\times 10^3$ samples of $n = 10^4$ from the following distributions satisfying the
    key hypothesis ($\varphi_0^{\prime\prime} (m)< 0$) of Theorem~\ref{LRasympNullDistribution}:
    Gamma$(3,1)$, Beta$(2,3)$,  Weibull$(3/2,1)$.
    The results are shown in
    Figure~\ref{fig:figure1}.
    Figure~\ref{fig:figure1} %
also includes:
(i) a plot of the known d.f. of a chi-square random variable with $1$ degree of freedom,
(which is the limiting distribution of the likelihood ratio test statistic for testing a one-dimensional parameter in a regular
parametric model);
(ii) a plot of the empirical distribution of
$2 \log \lambda_n$ for $M = 5\times 10^3$ samples of size $n=10^4$
drawn drawn from the Laplace density $2^{-1} \exp ( - | x |)$ for which the assumption of Theorem 1 fails.
In keeping with Theorem~\ref{LRasympNullDistribution}, the empirical results for all the distributions
satisfying Theorem~\ref{LRasympNullDistribution} are tightly clustered and in fact are almost visually indistinguishable,
in spite of the fact that the various constants associated with these distributions
are quite different, as shown in Table 1; in the next to last column
$C(f_0) \equiv \left ((4!)^2 f_0 (m)/ (f_0^{\prime\prime} (m))^2) \right )^{1/5}$ from \eqref{LimitDistribMLEofMode}, and in the last column SLC stands for
``strongly log-concave'' (see e.g.\
\cite{MR3290441}).

\begin{table}[h]
  \begin{center}
    \caption{Numerical characteristics of the distributions in the null hypothesis Monte-Carlo study}
    \label{tab:1}
    \medskip
    \begin{tabular}{| l | c | c  | c  | c|  c | r |}  \hline
      distribution& $m$ & $f_0 (m)$ & $f_0^{\prime \prime} (m)$ & $\varphi_0^{\prime\prime} (m)$ & $C(f_0)$ & SLC \\ \hline
      \ & \ & \ & \ & \ & \ & \ \\
      N(0,1) & $0$ & $(2\pi)^{-1/2} = .3989\ldots$ & $-(2\pi)^{-1/2}$ & $-1$ & $4.28$ & Y \\
      \ & \ & \ & \ & \ & \ & \ \\
      Gamma$(3,1)$ & $2$ & $2e^{-2}$ & $-e^{-2}$ & $-1/2$ & $6.109$ & Y \\
      \ & \ & \ & \ & \ & \ & \ \\
      Weibull$(3/2,1)$ & $3^{-2/3}$ & $\frac{3^{2/3}}{2 e^{1/3}}$ & $- \frac{27}{8 e^{1/3}}$ & $ - \frac{9\cdot 3^{1/4}}{4}$ & $2.36$ & N \\
      \ & \ & \ & \ & \ & \ & \ \\
      Beta$(2,3)$ & $3^{-1}$ & $\frac{16}{9} $ & $-24$ & $- \frac{27}{2} $ & $1.12$ & Y \\
      \ & \ & \ & \ & \ & \ & \ \\
      Logistic & $0$ & $1/4$ & $-1/8$ & $-1/2$ & $6.207$ & N \\
      \ & \ & \ & \ & \ & \ & \ \\
      Gumbel & $0$ & $e^{-1}$ & $-e^{-1}$ & $-1$ & $4.3545 $ & N \\
      \ & \ & \ & \ & \ & \ & \ \\
      $\chi_4^2 $ & $2$ & $\frac{1}{2 e}$ & $- \frac{1}{8 e}$ & $-\frac{1}{4} $ & $8.7091$ & N \\
      \ & \ & \ & \ & \ & \ & \ \\
      \hline
    \end{tabular}
  \end{center}
\end{table}

\begin{figure}%
  \centering
  \includegraphics[width=.9\textwidth]{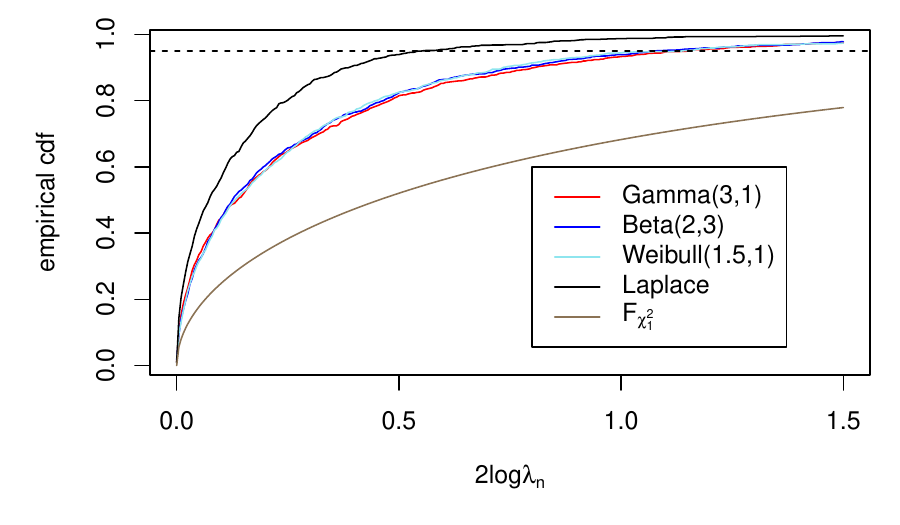}
  \caption{Monte Carlo distributions of $2 \log \lambda_n$ for four distributions,  $n=10^4$, $M = 5\times 10^3$ replications, together with the exact distribution function of $\chi_1^2$.}
  \label{fig:figure1}
\end{figure}

Now for $\alpha \in (0,1)$ let $d_{\alpha}$  satisfy $P ( \DD  > d_{\alpha} ) = \alpha $.
The following table gives
\begin{table}%
  \begin{center}
    \caption{Estimated critical values $d_{\alpha}$}
    \label{tab:2}
    \medskip
    \begin{tabular}{| l | c | c  | c  | c|  c | c |}  \hline
      $\alpha $ & $.25 $ & $.20$ & $.15$ & $.10$ & $.05$ & $.01$ \\   \hline
      $d_{\alpha}$  & $.40$ & $.49$  & $.61$  & $.79$ & $1.11$ & $1.92$  \\ \hline
    \end{tabular}
  \end{center}
\end{table}
a few estimated values for $d_{\alpha}$:
These are based on $350,000$ Monte Carlo simulations each based on simulating
$1 \times 10^6$ observations from a standard normal.  These values, and the
simulated critical values for all $\alpha \in (0,1)$, are   available
in the \verb+logcondens.mode+ package  \citep{doss:logcondens.mode}
in R \citep{R-core}.

\cite{MR1891743} study  a likelihood ratio test in the context of constraints based on monotonicity,
and find a universal limiting distribution, denote it $\DD_{\text{mono}}$, for their likelihood ratio test.
Comparison of the values in Table~\ref{tab:2} with Table~2 of
\cite{MR1891743} (particularly Method 2 in column 3 of that table) suggest,  perhaps surprisingly,
that %
$P( 2 \DD \le t) \approx P( \DD_{\text{mono}} \le t)$ for $ t \in \RR$.
It would be quite remarkable if this held exactly.
We do not have any explanation for this observed phenomenon.

\subsection{Comparisons via simulations}

Code to compute the mode-constrained log-concave MLE, implement a corresponding
test, and invert the family of tests to form confidence intervals is
available in the \verb+logcondens.mode+ package
\citep{doss:logcondens.mode}.
We can thus test our procedure and compare it to alternatives.

\cite{MR964293} %
proposed and investigated two  methods of forming confidence
intervals for the mode of a unimodal density.  His estimators of the mode
and confidence intervals were based on the classical kernel density
estimators of the density $f$ going back to \cite{MR0143282}.   %
One method, which Romano called the ``normal approximation method'',
is
based on the limiting normality of the kernel density estimator of the mode,
together with a plug-in estimator of the asymptotic variance.
Romano's second method involved bootstrapping the mode estimator,
and involved the choice of two bandwidths, one for the initial estimator
to determine the mode, and a second (larger) bandwidth for the bootstrap
sampling.
The abstract of
\cite{MR964293}  %
states:
``In summary, the results are negative in the sense that a straightforward application of a naive bootstrap
yields invalid inferences.  In particular the bootstrap fails if resampling is done from the kernel density estimate.''  That is, one must use a second (larger) bandwidth for the bootstrap resampling to achieve valid inference.  This thus necessitates selection of two tuning parameters for the bootstrap procedure.
\cite{MR964293}  %
notes in summarizing his simulation results:

\begin{quotation}
  ... but the problem of constructing a confidence interval for the mode for smaller sample sizes remains a challenging one.  In summary, the simulations reinforce the idea that generally automatic methods like the bootstrap need mathematical and numerical justification before their use can be recommended.
\end{quotation}

The bootstrap simulations that
\cite{MR964293}  %
refers to in the previous quote are based on an underlying $N(0,1)$ or a $\chi_4^2$ distribution with a sample size of $n=100$.
\cite{MR964293}  %
also performs simulations for the normal approximation method for the same underlying distributions and based on the same sample size.  For the normal approximation method,
a grid of bandwidths $h$ are used for the simulation.  For the bootstrap, a matrix of bandwidth pairs $(h,b)$ (one for estimation, one for resampling) are used.
Monte Carlo estimates of  coverage probabilities
are presented in Tables 1--4 of
\cite{MR964293}.  %

In Figure~\ref{fig:figure5} we consider the case of a true underlying $\chi_4^2$ distribution, and we plot all the estimated coverage probabilities of Romano's bootstrap CI's (blue; these are from \cite{MR964293}'s Table 4)  together with the target (ideal) coverage (green line) and the estimated coverage probabilities  of our likelihood ratio (LR) based CI's (magenta).  As can be seen, the estimated coverage probabilities of our LR-based procedure are reasonably close to the target values without requiring any bandwidth choice.

Corresponding comparison plots based on Tables 1,2, and 3 of
\cite{MR964293},  %
as well as tables of the simulated coverage probabilities,
are given in
\cite{Doss-Wellner:2016ModeInference-arxiv-v2}.  We do not include them here due to space constraints.
\cite{Doss-Wellner:2016ModeInference-arxiv-v2} also includes a Monte Carlo simulation study of lengths of the CI's in some settings.

\begin{figure}%
  \centering
\centerline{\includegraphics[height=2.8in]{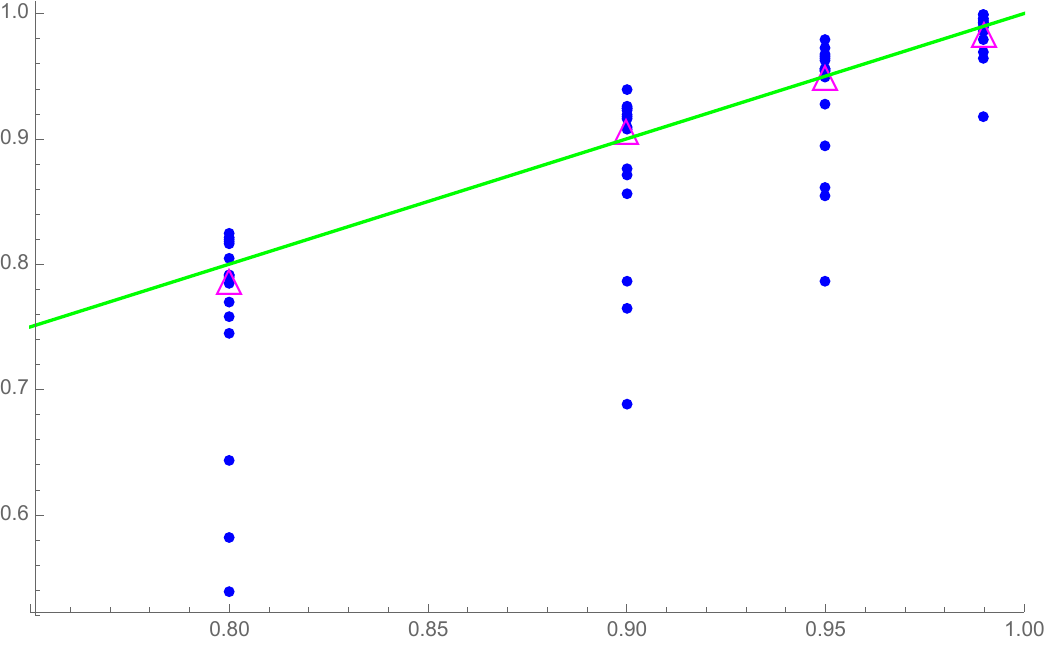}}
  \caption{Coverage probabilities, Romano's Table 4 compared with LR coverage probabilities, $\chi_4^2$ data,
    Bootstrap confidence intervals (blue dots);
    LR confidence intervals  (magenta triangles).  The green line is the nominal level.}
  \label{fig:figure5}
\end{figure}

Methods of bandwidth selection for various problems have received considerable attention in the period since
\cite{MR964293};  %
see especially
\cite{MR1089472},   %
\cite{MR1450020},    %
\cite{MR1049312},   %
\cite{MR1160479},   %
\cite{MR1839003},  %
\cite{MR1394655},  %
\cite{MR1618191},   %
and
\cite{MR2662359}.   %
Although bandwidth selection in connection with mode estimation is mentioned briefly by
\cite{MR1089472} (see their last paragraph, page 734), we are not aware of any specific proposal
or detailed study of bandwidth selection methods in the problem of confidence intervals for the mode of
a unimodal density.
For this reason, we have not undertaken a full comparative study of possible methods here.
Further comparisons of our LR based confidence intervals with methods based on kernel density estimates of the
type studied by \cite{MR964293} but incorporating current state of the art bandwidth selection procedures will be of interest.

\subsection{Comparisons via data examples}
\label{subsec:examples}

We used our procedure for formation of modal confidence intervals (CI's) on two real data sets,
the rotational velocities of stars from the Bright Star Catalogue  \citep{HW1991brightstar}  and  daily log returns for the S\&P $500$ stock  market index.
To see the former,
see Section~5.3 of \cite{Doss-Wellner:2016ModeInference-arxiv-v2}.  Here we discuss the
$1006$ daily log returns for the S\&P $500$ stock
market index from January $1$, 2003 to December $29$, 2006.
In
Figure~\ref{fig:SP500_2002_KDE-LC} we plot the data, a
kernel density estimate with bandwidth
 $.13$ \citep{Sheather:1991tp}, the log-concave MLE, and the $95$\%
confidence interval for the mode given by our likelihood ratio statistic.
We also plot the maximum likelihood Gaussian
density estimate, for comparison.
The sample mean is $0.04$,  the sample median is $0.081$, and the log-concave mode estimate is $0.17$.  A  $95$\% CI's  for the mean
is
$[-0.004, 0.09]$ and
a $95$\% CI for the median is
$[.037, .122]$,
 \citep[pages 539--540]{MR0251823}.
Our likelihood ratio CI for the mode is $[0.10, 0.21]$.
Note that our confidence interval for the mode excludes $0$ and does not intersect with the CI for the mean.  Thus, our procedure highlights some interesting features of the data and provides evidence for its non-normality.
Also note that the lengths of the mean, median, and our LR-based mode CI are
$0.094$,
$0.085$,
and $0.11$.  Thus, despite the fact that our  mode estimator does not generally have a $n^{-1/2}$ rate of convergence, the three confidence intervals are of fairly similar length on a dataset with  $1006$ observations, which is encouraging for our mode CI procedure and for any future extensions (e.g., mode regression CI's).
\begin{figure}
  \centering
  \includegraphics[scale=.40]{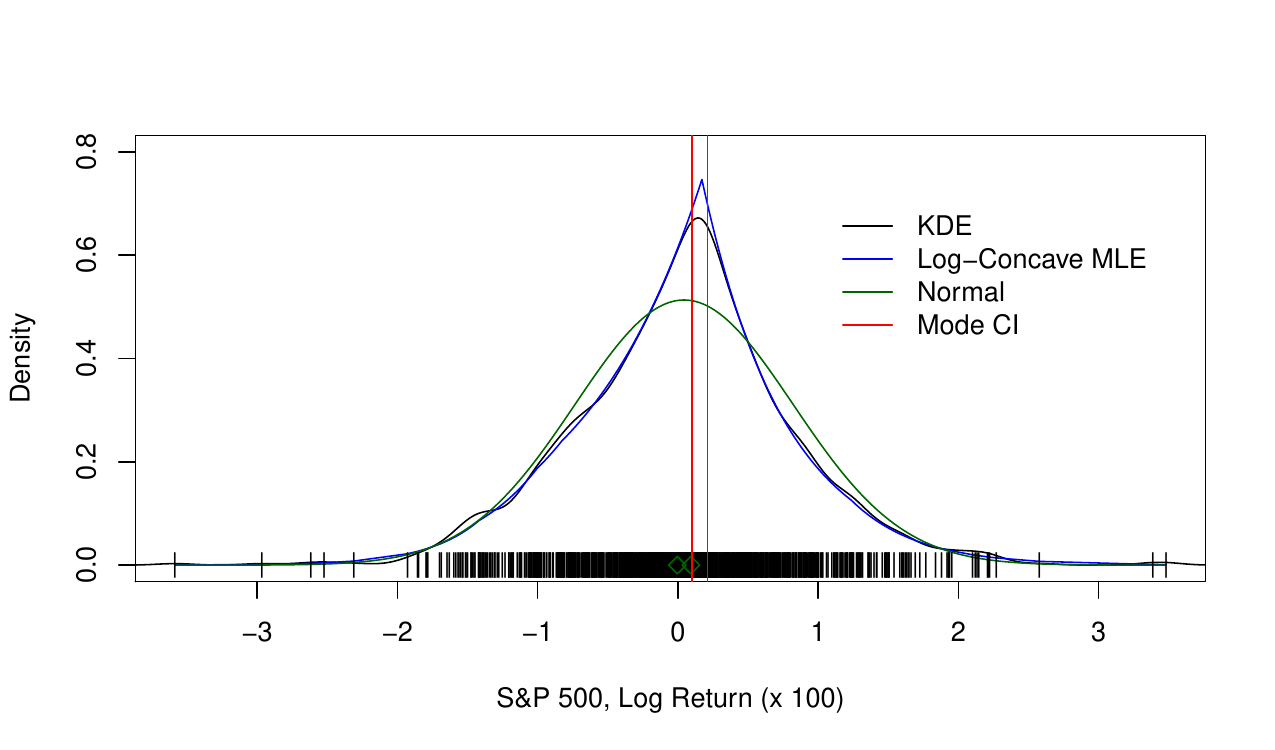}
  \caption[S\&P $500$ for years $2003$-$2006$]{$1006$ S\&P $500$ daily log
  returns for the years $2003$--$2006$}
  \label{fig:SP500_2002_KDE-LC}
\end{figure}

\section{Further Problems and potential developments}
\label{sec:conclusions}

\subsection{Uniformity and rates of convergence}

There is a long line of research giving negative results concerning nonparametric estimation,
starting with
\cite{MR0084241}, %
\cite{MR0205414},  %
\cite{MR0155394,MR0232487}, %
and continuing with
\cite{Donoho1988oneSidedInference}
and  %
\cite{MR1671670,MR1747496}.   %
In particular,
\cite{MR1747496} %
considers a general setting involving estimators or confidence limits with
optimal convergence rate $n^{-\rho}$ with $0 < \rho < 1/2$.
He shows, under weak additional conditions, that:
(i) there do not exist estimators which converge locally uniformly to a limit distribution; and
(ii) there are no confidence limits with locally uniform asymptotic coverage probability.
As an example he considers the mode of probability distributions $P$ on $\RR$ with
corresponding densities $p$ having a unique mode $M(p)$ and continuous second derivative
in a neighborhood of $M(p)$.
\cite{MR1747496} %
also reproves the result of
\cite{Hasminskii1979mode}  %
to the effect that the optimal rate of
convergence of a mode estimator for such a class is $n^{-1/5}$.  In this respect, we note that
\cite{BRW2007LCasymp} %
established a comparable lower bound
for estimation of the mode in the class of log-concave densities with continuous second derivative
at the mode; they obtained a constant which matches (up to absolute constants)
the pointwise (fixed $P$) behavior of the plug-in log-concave MLE of the mode.
\cite{MR947566}  %
gives a detailed treatment of minimax lower bounds for estimation of the mode under smoothness and curvature assumptions: assuming a bounded
derivative of order $p$ in a neighborhood of the mode
$M(f_0)$, Romano shows in his Theorem 3.1 that the minimax rate for estimation of $M(f_0)$ is $n^{-r}$ where
$r = (p-1)/(2p+1)$. He also shows that when $p=3$, the rate $n^{-2/7}$ can be achieved by a kernel density
estimator.

Our approach here has been to %
construct reasonable confidence intervals with pointwise (in $P$ or density $p$)
correct asymptotic coverage without proof of any local uniformity properties.  In view of the recent uniform rate results of
\cite{Kim-Samworth:16}  %
we suspect that our new confidence intervals {\sl will} (eventually) be shown to have some uniformity of
convergence in their coverage probabilities over appropriate subclasses of the class of log-concave densities,
but we leave the uniformity issues to future work.

\subsection{Some further directions and open questions}

But we now turn to discussion of some difficulties and potential for further work.

\subsubsection{Relaxing the second derivative assumption:}
\label{subsubsec:relaxing-curvature-assm}
As noted in the previous subsection, most of the available research concerning
inference for $M(f)$ has assumed $f \in C^2 (m, \mbox{loc})$ and  $f^{(2)} (M(f)) <0$.
Second derivative type assumptions of this type are made in
\cite{MR0143282},  %
\cite{Hasminskii1979mode},  %
\cite{Eddy1980mode}, %
\cite{MR1105839},   %
\cite{MR964293,MR947566},  and  %
\cite{MR1747496}.  %
Exceptions include
\cite{MR1015137},  %
\cite{Ehm1996cuspMode}, %
\cite{HerrmannZiegler2004mode}, %
\cite{BRW2007LCasymp}. %

What happens if the second derivative curvature assumption does not hold,
but instead is replaced by something either stronger or weaker, such as
\begin{equation*}
  f(m) - f(x) \le C |x - m|^r
\end{equation*}
for some $C$ where $1 \le r < 2$ (in the ``stronger'' case)
or $2 < r < \infty$ (in the ``weaker'' case)?
It is natural to expect that it is easier to form confidence intervals for $m$ when $1\le r< 2$ holds,
but that it is harder to form confidence intervals for $m$ when $2<r< \infty$.  In fact,
\cite{BRW2007LCasymp} page 1313 %
gives the following result:  if $f = \exp (\varphi)$ with $\varphi$ concave and
where $\varphi^{(j)} (m) = 0$ for $j=2, \ldots , k-1$ but
$\varphi^{(k)} $ exists and is continuous in a neighborhood of $m$ with $\varphi^{(k)} (m) \not= 0$,
then
$$
n^{1/(2k+1)} ( \widehat{M}_n - m ) \rightarrow_d  C_k (f(m), \varphi^{(k)} (m) ) M( H_k^{(2)} ) .
$$
Thus the convergence rate of the log-concave MLE of the
mode is slower as $k$ increases.
[On the other hand, by Theorem 2.1
of \cite{BRW2007LCasymp}, page 1305, the convergence rate
of the MLE $\widehat{f}_n $ of $f$ at $m$ (and in a local neighborhood of $m$) is {\sl faster}:
$$
n^{k/(2k+1)} ( \widehat{f}_n (m) - f(m)) \rightarrow_d c_k (m, f) H_k^{(2)} (0) .]
$$
Furthermore the sketch of the proof of the limiting distribution of the likelihood ratio statistic in
Section~\ref{ssec:PfSketchThm1} (ignoring any remainder terms) together with the results of
\cite{BRW2007LCasymp}, suggest
that $2\log \lambda_n \rightarrow_d  \DD_k $ under $f \in {\cal P}_m \cap {\cal Z}_k$
where
$$
{\cal Z}_k = \{ f \in {\cal P} : \ \varphi^{(j)} (m) = 0, \  j=2, \ldots , k-1, \ \varphi^{(k)} (m) \not= 0,\ \varphi \in C^k (m , \mbox{loc} ) \}
$$
and where with $ \varphi_k$ and $\varphi_k^0$ denoting the local limit processes in the white noise model (\ref{WhiteNoiseCanonicalConcave})
with drift term $g_0 (t) = - 12t^2$ replaced by $- (k+2)(k+1) | t |^k$,
$$
\DD_k \equiv \int \{ (\widehat{\phi}_k (v))^2 - (\widehat{\phi}_k^0(v) )^2 \} dv .
$$
We provide Monte Carlo evidence in support of this conjecture, by simulating $2 \log \lambda_n$ based on some parent distributions with $k \ne 2$.  The results are given in
Figure~\ref{fig:LRasymptotics-all}.
Figure~\ref{fig:LRasymptotics-all}  contains empirical distributions of $2 \log \lambda_n$ (with $n=10^4$ and $M=5\times 10^3$) for $9$ parent distributions, as well as a plot of the df of a $\chi_1^2$ random variable;
all of the curves from Figure~\ref{fig:figure1} are present, including $F_{\chi_1^2}$,
the Laplace (with $k=1$),
the standard normal,
Gamma$(3,1)$,
Beta$(2,3)$,
and Weibull$(1.5,1)$ (all four having $k=2$).
We also add four parent distributions with $k > 2$.  We include parent densities
proportional to $\exp\lb -|x|^j / j \rb$ for $x \in \RR$,
labelled ``Subbotin$(j)$,'' $j=3,4$  (having $k=j$).
We also include parent densities proportional to
$1-|x|^j$ for $x \in [-1,1]$,
labelled ``Bump$(j)$,'' $j=3,4$ (with $k=j$).
The (Monte Carlo estimators of) dfs based on the parent distributions with $k=3$ (estimating $\DD_3$) are grouped together, and the dfs  based on the parent distributions with $k=4$ (estimating $\DD_4$) are similarly grouped together.
Note that the   (Monte Carlo estimator of)
the distribution of $\DD_3$ seems to be stochastically larger than the (Monte-Carlo estimator) of the distribution
of $\DD_2 \equiv \DD$, and that the distribution of $\DD_4$ is apparently stochastically larger than that of $\DD_3$.
\begin{figure}[htb!]
  \centering
  \includegraphics[width=.9\textwidth]{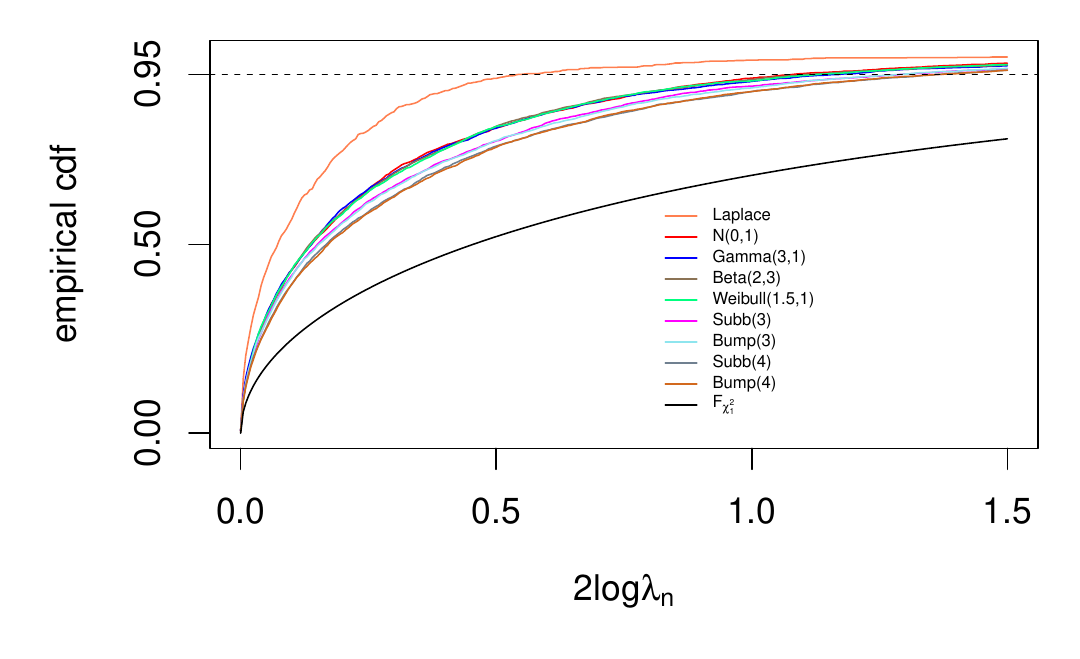}    %
  \caption{Empirical distributions of $2 \log \lambda_n$ for $f \in {\cal P}_m \cap {\cal Z}_k$, $k \in \{ 2,3,4 \}$ with  $n=10^4$ and $M = 5\times 10^3$ replications.}
  \label{fig:LRasymptotics-all}
\end{figure}
This raises several possibilities:
\begin{description}
\item[Option 1:]
  It seems likely that by choosing a critical value from the distribution of $\DD_6$ (say), that the resulting confidence
  intervals will have correct coverage for $f \in {\cal P} \cap {\cal Z}_6$ with conservative coverage if we happen to
  have $f \in {\cal P}\cap {\cal Z}_2$ (in which case critical points from $\DD = \DD_2$ would have sufficed), and anti-conservative
  coverage if the true $f$ belongs to $ {\cal P}\cap ({\cal Z}_k \setminus {\cal Z}_6)$ for some $k \ge 8$.
\item[Option 2:]
  Try to construct an adaptive procedure which first estimates $k$ (the degree of ``flatness'') of the true $f \in {\cal P}$
  (by $\hat{k}$ say),
  and then choose a critical point from the distribution of $\DD_{\hat{k}}$.
\end{description}
We leave the investigation of both of these possibilities to future work.

\subsubsection{Relaxing the assumption of log-concavity}
It would also be of interest to relax the assumption of log-concavity used in the developments here.
It would be very desirable to allow $f_0$ to be a completely arbitrary unimodal density, and allow
the smoothness at the mode $M(f_0)$ to vary as noted in the previous subsection.
As a more realistic replacement for this ambitious goal, we might instead consider
enlarging from the class of log-concave densities to some class of $s-$concave densities, ${\cal P}_s$
with $-1<s<0$;  i.e. densities of the form $f_0  = \phi_0^{1/s}$ with $\phi_0 > 0  $ convex; see e.g.\
\cite{KM2010quasiconcave},
\cite{DossWellner:2016a},
and
\cite{han:2015tr}.
Extensions in this direction will likely require further study of the R\'enyi divergence
estimators studied in
\cite{han:2015tr}
and mode-constrained versions thereof.
An interesting possible connection is that for the classes of $\alpha-$stable densities ${\cal S}_{\alpha}$
with $0 < \alpha \le 2$, we know that $f_0 \in {\cal S}_{\alpha}$ is unimodal.
Moreover, it is also known from
\cite{MR771431} %
that for the symmetric $\alpha-$stable distributions
$f_0^{\prime\prime}$ exists
in a neighborhood of the mode $m=M(f_0)$, and $f_0^{\prime\prime} (m) < 0$.
It is apparently not known if the $\alpha-$stable densities are $s-$concave for some $s \in (-1,0]$, even
though this obviously holds in the (few) examples for which an explicit formula for the density $f_0 \in {\cal S}_{\alpha}$
is available:  for example for $f_0 = $ Cauchy,  $f_0 \in {\cal S}_1 \cap {\cal P}_{-1/2}$, while for
L\'evy's completely asymmetric stable law, $f_0 \in {\cal S}_{1/2} \cap {\cal P}_{-2/3}$,
and of
course, for $f_0 = \ $Gaussian, $f_0 \in {\cal S}_2 \cap {\cal P}_0$.

\subsubsection{Mode inference in other contexts}

The methods developed in this paper raise several questions about other settings in which inference about %
a convex function may be of interest.

\begin{enumerate}[label=(\alph*),leftmargin=*]
\item
Can we do inference for the maxima or minima
in the contexts of estimation of intensity functions,  of (bathtub-shaped) hazard functions \citep{Jankowski:2009bx}, or of regression functions?
For instance, let $Y_i = r(x_i) + \epsilon_i$ where $\epsilon_i$ are mean zero i.i.d.\
observations and $x_i$ are fixed numbers in $\RR$.
If  we assume $r$ to be convex, then much is known about the least-squares estimator $\widehat{r}_n$ of $r$; see, e.g.,
\cite{Hildreth1954concave},
\cite{HP1976concave},
\cite{Mammen1991convex}, and
\cite{MR1891741,MR1891742}.
Can an argmin-constrained estimator $\widehat{r}_n^0$ be developed, in analogy
with the estimator $\ffna$, and used to develop likelihood ratio-based (or rather, residual sum-of-squares) tests and intervals for the location of the minimum of $r$?
In such a problem, we conjecture that the universal component of the
limit distribution of $\widehat{r}_n^0(m)$
will be the same as that studied in
Theorem~\ref{LRasympNullDistribution}.

\item Can the techniques used here be applied to form tests and intervals for the {\em value} (or height) of a concave function, $f_0$, rather than argmax?  Here, $f_0$ could be a log-concave density or a concave regression function (and other settings could be of interest).  That is, can we develop an estimator $\widehat{f}_n^0$ based on the constraint that $f$ satisfies $f(x_0) =y_0$ for $x_0, y_0$ fixed, and use $\widehat{f}_n^0$ with an unconstrained estimator $\ffn$ to form a likelihood ratio test for $f_0(x_0)$?  In the case where $f_0$ is a concave regression function, such a program has recently been studied by \cite{Doss:2018uh}.
Can this be extended to the density case, where $f_0$ is log-concave?

\item Can inference for the mode be extended to semiparametric settings?  For example can we form tests/intervals for the location of the minimum of an unknown convex ``link'' function $m_0$ in a single index model, $Y= m_0(\theta_0'X) + \epsilon$, where $X \in \RR^d$, $Y \in \RR$, $E(\epsilon|X)=0$, and $m_0$ is assumed to be convex \citep{Kuchibhotla:2017ug,MR3534348}?  Can we form tests/intervals for a modal regression function, i.e.\ for $m_0$ where $Y = m_0(X) + \epsilon$ where $\epsilon$ has {\em mode} $0$?  %

\end{enumerate}

\subsubsection{Beyond dimension $d=1$:}
It seems natural to consider generalizations of the present methods
to the case of multivariate log-concave and $s-$concave densities.
While there is a considerable amount of work on estimation of multivariate modes, mostly via kernel density estimation,
much less seems to be available in terms of confidence sets or other inference tools.
For some of this, see e.g.
\cite{MR1051586}, %
\cite{MR1985502},   %
\cite{MR1327618},   %
\cite{MR2112688}, %
\cite{MR0336874},  %
\cite{MR491553},  %
\cite{MR0331618}.  %
On the other hand apparently very little is known about the multivariate mode estimator
$M(\widehat{f}_n)$ where $\widehat{f}_n$ is the log-concave density estimator for $f $ on $\RR^d$
studied by
\cite{MR2758237} and  %
\cite{MR2645484}.  %
Further study of this estimator will very likely require considerable development of new methods
for study of the pointwise and local properties of the log-concave density estimator $\widehat{f}_n$ when $d \ge 2$.

\section*{Acknowledgements}
We  owe thanks to Lutz D\"umbgen for several helpful conversations.
This work was partially supported by a grant to the second author from the Simons Foundation and
was carried out in part during a visit to the Isaac Newton Institute.

\appendix

\section{Proofs}
\label{sec:ProofsPart2}

We now deal  with the remainder terms defined in (\ref{BasicDecompositionFirstForm}) in the course of our ``proof sketch'' for Theorem~\ref{LRasympNullDistribution}.
We first deal with the ``local'' remainder terms $R_{n,j}, \ R_{n,j}^0 $ with $j \in \{ 2, 3 \}$ in Subsection~\ref{ssec:ProofsLocalRTs}.
The analysis of these local remainder terms depends crucially on
Theorem~\ref{thm:GlobalConsistencyWithRates}.
 Subsection~\ref{ssec:GlobalRTs} is dedicated to the proofs for the ``non-local'' remainder terms.

For a function $f \colon \RR \to \RR$, we let $\| f \| := \sup_{x \in \RR} |f(x)|$, and for a set $J \subset \RR$ we let $\| f \|_J := \sup_{x \in J} |f(x)|$.

\subsection{The local remainder terms $R_{n,j}, R_{n,j}^0$, $j\in \{ 2, 3\}$}
\label{ssec:ProofsLocalRTs}

We first deal with the (easy) local remainder terms.
\begin{proposition}
  \label{prop:localerrorterms}
  Let $t_{n,1} = m - M n^{-1/5}$ and $t_{n,2} = m + M n^{-1/5}$ for $M >
  0$. Then the remainder terms $R_{n,2}$, $R_{n,2}^0$, $R_{n,3}$, and
  $R_{n,3}^0$ satisfy $n R_{n,j} = o_p (1)$ and $nR_{n,j}^0 = o_p (1)$ for $j
  \in \{2,3\}$.
\end{proposition}

\begin{proof}
  Recall that
  the remainder terms $R_{n,2}$, $R_{n,2}^0$, $R_{n,3}$, and $R_{n,3}^0$
  given by (\ref{RemainTermTwoUnconstrained}), (\ref{RemainTermTwoConstrained}), (\ref{RemainTermThreeUnconstrained}),
  and (\ref{RemainTermThreeConstrained}) are all of the form a constant times
  \begin{eqnarray*}
    && \tilde{R}_n \equiv \int_{D_n} e^{\tilde{x}_{n} (u) } (\widehat{\varphi}_n (u) - \varphi_0 (m) )^3 du, \ \ \ \mbox{or}\\
    && \tilde{R}_n^0 \equiv \int_{D_n} e^{\tilde{x}_{n}^0 (u) } (\widehat{\varphi}_n^0 (u) - \varphi_0 (m) )^3 du,
  \end{eqnarray*}
  where $D_n$ is a (possibly random) interval of length $O_p (n^{-1/5})$ and $\tilde{x}_{n,j}$ converges in probability,
  uniformly in $u \in D_n$, to zero.  But by Assumption 1 and by
  Theorem~\ref{thm:GlobalConsistencyWithRates} Part~\ref{thm:GlobalConsistencyWithRatesUnconstrained}
  it follows that for any $M>0$ we have
  \begin{eqnarray*}
    \lefteqn{\sup_{|t| \le M} | \widehat{\varphi}_n ( m + n^{-1/5} t) - \varphi_0 (m) | } \\
    & \le & \sup_{|t| \le M} \left \{ |  \widehat{\varphi}_n ( m + n^{-1/5} t) - \varphi_0 (m + n^{-1/5} t) |
      + | \varphi_0 (m + n^{-1/5} t) - \varphi_0 (m) | \right \}\\
    & = & O_p ( (n^{-1} \log n)^{2/5} ) + O ( n^{-2/5} )  = O_p ( ( n^{-1} \log n)^{2/5} ),
  \end{eqnarray*}
  and hence
  \begin{eqnarray*}
    n | \tilde{R}_n | = n O_p ( (n^{-1} \log n)^{6/5} \cdot n^{-1/5} ) = O_p ( n^{-2/5} (\log n)^{6/5} ) = o_p (1) ,
  \end{eqnarray*}
  By Assumption 1 and by
  Theorem~\ref{thm:GlobalConsistencyWithRates}  Part~\ref{thm:GlobalConsistencyWithRatesConstrained}
  it follows that for any $M>0$ we have
  \begin{eqnarray*}
    n | \tilde{R}_n^0 | = n O_p ( (n^{-1} \log n)^{6/5} \cdot n^{-1/5} ) = O_p ( n^{-2/5} (\log n)^{6/5} ) = o_p (1) .
  \end{eqnarray*}
  Note that the $o_p$ terms do not depend on $M$, by
  Theorem~\ref{thm:GlobalConsistencyWithRates}.  This completes the proof of
  negligibility of the local error terms $R_{n,j}, R_{n,j}^0$, $j\in \{ 2,
  3\}$.
\end{proof}

\subsection{The global remainder terms $R_{n,1}$ and $R_{n,1}^c$}
\label{ssec:GlobalRTs}
Recall that the remainder terms $R_{n,1}$ and $R_{n,1}^c$ are given by
(\ref{RemainTermOne}) and (\ref{RemainTermOneCompl}).  Note that the integral
in the definition of (\ref{RemainTermOneCompl}) is over $[X_{(1)}, X_{(n)}]
\setminus D_n$, and hence this term in particular has a global character.  We
will see later that $R_{n,1}$ also can be seen as having a global nature.

{\bf Outline:} From now on, we will focus our analysis on the portion of
$R_{n,1,\twoArgs}^c$ given by integrating over the left side, $[X_{(1)},
t_1]$.  Arguments for the integral over $[t_2, X_{(n)}]$ are analogous.
Thus, by a slight abuse of notation, define the one-sided counterpart to
$\RnTwo^c$ from \eqref{RemainTermOneCompl} for any $t < m $ by
\begin{equation}
  \label{eq:defn:LRS-remainder-4}
  \begin{split}
    R_{n,1,\oneArg}^c & \equiv \int_{[X_{(1)}, t]} \vvn d\FFn - \vvna d\FFna -
    \int_{[X_{(1)}, t]} (e^{\vvn} -    e^{\vvna}) \, d\lambda.
  \end{split}
\end{equation}
Here $\lambda$ is Lebesgue measure (and is unrelated to the likelihood ratio $\lambda_n$).
The analysis of $R_{n,1,t}^c$ is the greatest difficulty in understanding $2 \log
\lambda_n$.  The proof that $R_{n,1,t_n}^c$ is $o_p(n^{-1})$ when $b \to \infty$ where
$t_n = m - b n^{-1/5}$ is somewhat lengthy so we provide an outline here.
\begin{enumerate}
\item \label{item:rem:outline-1} {\bf Step 1, Decomposition of
    $R_{n,1,t}^c$:} Decompose $R_{n,1,t}^c$, to see that
  \begin{equation}
    R_{n,1,t}^c = A_{n,t}^1 + E_{n,t}^1 - T_{n,t}^1
    = A_{n,t}^2 + E_{n,t}^2 + T_{n,t}^2,
  \end{equation}
  where the summands $A_{n,t}^i, E_{n,t}^i, T_{n,t}^i$ are defined below
  (see \eqref{eq:rem:Rn1tc-main-decomposition} and the preceding text).
\item \label{item:rem:outline-2} {\bf Step 2, Global $O_p(n^{-1})$
    conclusion:}  %
  In this section we use the fact that away from the mode, the
  characterizations of $\vvn $ and $\vvna$ are identical to study $T_n^i$,
  $i=1,2$, which are related to $\int_{[X_{(1)}, t]} (\vvn - \vvna)^2 \ffn d\lambda$.
  and $\int_{[X_{(1)}, t]} (\vvn - \vvna)^2 \ffna d\lambda$.  We will show
  ${T}_n^i = O_p(n^{-1})$, $i=1,2$.  Note $O_p(n^{-1})$ would be the
  size of the integral if it were over a local interval of length
  $O_p(n^{-1/5})$ (under our curvature assumptions), but here the integral is
  over an interval of constant length or larger, so this result is global in nature.
\item \label{item:rem:outline-3} {\bf Step 3, Convert global $O_p$ to local
    $o_p$ to global $o_p$:} Convert the global $O_p(n^{-1})$ conclusion over
  $T_{n}^i$ into an $o_p(n^{-1})$ conclusion over a interval of length
  $O_p(n^{-1/5})$ local to $m$.  Feed this result back into the argument in
  Step~\ref{item:rem:outline-2}, yielding $T_{n,t}^i = o_p(n^{-1}),$ $i=1,2$.
  Apply Lemma~\ref{lem:rem:local-to-global-square-integral} to show
  additionally that there exist knots of $\vvn$ and $\vvna$ that are
  $o_p(n^{-1/5})$ apart in an $O_p(n^{-1/5})$ length interval on which $\Vert
  \vvna - \vvn \Vert = o_p(n^{-2/5})$,
  $\Vert
  \ffna - \ffn \Vert = o_p(n^{-2/5})$,
  and $\Vert \FFna - \FFn \Vert =
  o_p(n^{-3/5})$.
\item \label{item:rem:outline-4} {\bf Step 4, Concluding arguments:} Return
  to the decomposition of $R_{n,1,t}^c$ given in
  Step~\ref{item:rem:outline-1}; the terms given there depend on
  $\vvna-\vvn$, $\ffna -\ffn$, and $\FFna-\FFn$.  Thus, using the results of
  Step~\ref{item:rem:outline-3} we can show $n R_{n,1,t}^c = o_p(1)$ as desired.
\end{enumerate}
To finalize the argument, in Section~\ref{subsubsec:final-arguments}, we take
$t_n = m - bn^{-1/5}$, but we also need to let $b \to \infty$.  Thus, the
$O_p$ and $o_p$ statements above need to hold uniformly in $b$.

\subsubsection{Decomposition of $R_{n,1,t}^c$}

We begin by decomposing $R_{n,1,t}^c$ for fixed $t < m$.
By \eqref{eq:f-to-varphi-taylor-2} with $\vp_{1n} = \vvn$ and  $\vp_{2n} = \vvna$,  we see that
\begin{align}
  R_{n,1,t}^c
  & =
  \int_{[X_{(1)}, t]} \lp \vvn \ffn - \vvna \ffna
  - \lp \vvn - \vvna + \frac{(\vvn - \vvna)^2}{2} e^{\epsilon_{n}^1} \rp \ffna
  \rp \, d \lambda \nonumber  \\
  & = \int_{[X_{(1)}, t]} \lp \vvn \ffn - \vvn \ffna
  -  \frac{(\vvn - \vvna)^2}{2} e^{\epsilon_{n}^1} \ffna \rp d\lambda \nonumber
  \\
  & = \int_{[X_{(1)}, t]} \lp \vvn (\ffn - \ffna )
  -  \frac{(\vvn - \vvna)^2}{2} e^{\epsilon_{n}^1} \ffna \rp \, d\lambda,
  \label{eq:rem:Rn1tc-ffna-decomp1}
\end{align}
where $\lambda $ is Lebesgue measure and $\epsilon_{n}^1(x)$ lies between $0
$ and $\vvn(x)- \vvna(x)$.
Again applying  \eqref{eq:f-to-varphi-taylor-2} now with $\vp_{1n} = \vvna$ and  $\vp_{2n} =
\vvn$,  we see that
\begin{align}
  R_{n,1,t}^c
  & =
  \int_{[X_{(1)}, t]} \lp \vvn \ffn - \vvna \ffna
  + \lp  e^{\vvna - \vvn} - 1 \rp \ffn
  \rp \, d \lambda \nonumber  \\
  & = \int_{[X_{(1)}, t]} \lp \vvn \ffn - \vvna \ffna
  + \lp   \vvna - \vvn + \frac{(\vvna - \vvn)^2}{2} e^{\epsilon_{n}^2} \rp \ffn
  \rp \, d \lambda \nonumber \\
  & = \int_{[X_{(1)}, t]} \lp \vvna (\ffn - \ffna)
  +  \frac{(\vvna - \vvn)^2}{2} e^{\epsilon_{n}^2} \ffn \rp \,
  d\lambda, \label{eq:rem:Rn1tc-ffn-decomp1}
\end{align}
where $\epsilon_n^2$ lies between $0$ and $\vvna(x) - \vvn(x)$.
For a function $f(x)$, recall the notation $f_{s}(x) = f(x) -
f(s)$ for $x \le s$ and $f_s(x)=0$ for $x \ge s$.  Now define $A^i_{n,t}
\equiv A^i_n$, $i=1,2$ by
\begin{align}
  A_n^1
  \equiv   \int_{[X_{(1)}, t]} \widehat{\varphi}_{n,t} \, d\lp \FF_n - \FFna \rp
  \quad   \mbox{ and } \quad
  A_n^2 \equiv
  \int_{[X_{(1)}, t]}  \widehat{\varphi}_{n,t}^0 d\lp \FFn - \FF_n \rp
\end{align}
and
define $    E_{n,t}^1  \equiv E_n^1 $ to be
\begin{equation}
  \begin{split}
    \MoveEqLeft \int_{(\tau,t]} \widehat{\varphi}_{n,t} \, d\lp \FFn - \Fn \rp
    + \vvn(t) ( \FFn(t) - \FFna(t) ) \\
    & + (\vvn(\tau)    - \vvn(t)) ( \FFn(\tau) - \Fn(\tau) ) \label{eq:rem:defn:Ent1}
  \end{split}
\end{equation}
and $    E_{n,t}^2  \equiv E_n^2 $ to be
\begin{equation}
  \begin{split}
    \MoveEqLeft \int_{(\tau^0, t]} \widehat{\vp}_{n,t}^0  d\lp \Fn - \FFna \rp
    + \vvna(t) ( \FFn(t) - \FFna(t)) \\
    & + (\vvna(\tau^0)
    - \vvna(t)) (\Fn(\tau^0) - \FFna(\tau^0)), \label{eq:rem:defn:Ent2}
  \end{split}
\end{equation}
where $\tau \equiv \tau_-(t) = \sup S(\vvn) \cap (-\infty,t]$
and $\tau^0 \equiv \tau_-^0(t) = \sup S(\vvna) \cap (-\infty, t]$.
We will assume that
\begin{equation*}
  \tau \le \tau^0
\end{equation*}
without loss of generality, because the arguments are symmetric in $\vvn$ and
$\vvna$, since we will be arguing entirely on one side of the mode.

Our next lemma will decompose the first terms in
\eqref{eq:rem:Rn1tc-ffna-decomp1} and \eqref{eq:rem:Rn1tc-ffn-decomp1}, into
$A_{n}^i + E_n^i$, $i=1,2$.  The crucial observation is that
$A_n^1 \le 0 $ and $A_n^2 \ge 0$, by taking $\Delta = \vvnt{t}$ and $\Delta =
\vvnat{t}$ in the characterization theorems for the constrained and
unconstrained MLEs,
Theorem 2.2 A and B of
\cite{Doss-Wellner:2016ModeConstrained}. \label{page:Ani-inequalities}  Note that since $t \le m$,
$\vvnt{t}$ has modal interval containing $m$.

\begin{lemma}
  Let all terms be as defined above. We then have
  \begin{equation}
    \label{eq:rem:rn1c:1}
    \begin{split}
      \int_{[X_{(1)}, t]} \vvn ( \ffn - \ffna ) d\lambda
      & = A_{n,t}^1
      + E_{n,t}^1
    \end{split}
  \end{equation}
  and
  \begin{equation}
    \label{eq:rem:rn1c:2}
    \begin{split}
      \int_{[X_{(1)}, t]} \vvna (\ffn - \ffna) d\lambda
      & = A_{n,t}^2
      + E_{n,t}^2.
    \end{split}
  \end{equation}
\end{lemma}
\begin{proof}
  We first show \eqref{eq:rem:rn1c:1}.  We can see
  $    \int_{[X_{(1)}, t]} \vvn ( \ffn - \ffna)$ equals
  \begin{align}
    \int_{[X_{(1)}, t]} (  \widehat{\varphi}_{n,\tau_-} + \vvn -
    \widehat{\varphi}_{n,\tau_-} ) \ffn d\lambda
    - \int_{[X_{(1)}, t]} ( \widehat{\vp}_{n,t} + \vvn - \widehat{\vp}_{n,t}
    )
    \ffna d\lambda, \nonumber
  \end{align}
  and since $\int \widehat{\varphi}_{n,\tau_-} \, d\lp \Fn - \FFn \rp = 0$,
  this equals
  \begin{align}
    \MoveEqLeft \int_{[X_{(1)}, \tau_-]} \widehat{\vp}_{n,\tau_-} d\Fn
    + \int_{[X_{(1)}, \tau_-]} (\vvn - \widehat{\vp}_{n,\tau_-} ) \ffn
    d\lambda
    + \int_{(\tau_-,t]} \vvn \ffn d\lambda \nonumber \\
    &  \quad - \lp \int_{[X_{(1)}, t]} \widehat{\vp}_{n,t} \ffna d\lambda
    + \vvn(t) \FFna(t) \rp \nonumber \\
    & = \int_{[X_{(1)}, t]} \widehat{\vp}_{n,t} d\Fn
    + \int_{[X_{(1)}, \tau_-]}  (\widehat{\vp}_{n,\tau_-} -
    \widehat{\vp}_{n,t} ) d\Fn
    - \int_{(\tau_-,t]} \widehat{\vp}_{n,t} d\Fn \nonumber \\
    & \quad  + \int_{[X_{(1)}, \tau_-]} \vvn(\tau_-) \ffn d\lambda
    + \int_{(\tau_-,t]} \vvn \ffn d\lambda
    - \lp \int \widehat{\vp}_{n,t} \ffna d\lambda
    + \vvn(t) \FFna(t) \rp \nonumber \\
    & = \int_{[X_{(1)}, t]} \widehat{\vp}_{n,t} \, d\lp \Fn - \FFna \rp
    + \lp \vvn(t) - \vvn(\tau_-) \rp \Fn(\tau_-)
    - \int_{(\tau_-, t]} \vvn d\Fn \nonumber \\
    & \quad + \vvn(t) ( \Fn(t) - \Fn(\tau_-))
    + \vvn(\tau_-) \FFn(\tau_-) + \int_{(\tau_-, t]} \vvn \ffn d\lambda
    - \vvn(t) \FFna(t), \nonumber
  \end{align}
  which equals
  \begin{equation*}
    \begin{split}
      \MoveEqLeft \int \widehat{\vp}_{n,t} d \lp \Fn - \FFna \rp
      + \int_{(\tau_-, t]} \vvn d \lp \FFn - \Fn \rp \\
      & + \vvn(t) ( \Fn(t) - \FFna(t))
      + \vvn(\tau_-) ( \FFn(\tau_-) - \Fn(\tau_-))
    \end{split}
  \end{equation*}
  which equals
  \begin{equation*}
    \begin{split}
      \MoveEqLeft \int \widehat{\vp}_{n,t} d\lp \Fn - \FFna \rp
      + \int_{(\tau_-,t]} \widehat{\vp}_{n,t} d \lp \FFn-\Fn \rp \\
      & + \vvn(t) (\FFn(t) - \FFna(t))
      + (\vvn(\tau_-) - \vvn(t)) ( \FFn(\tau_-) - \Fn(\tau_-)),
    \end{split}
  \end{equation*}
  as desired.

  Now we show \eqref{eq:rem:rn1c:2}.  We see
  $\int_{[X_{(1)}, t]} \vvna (\ffn - \ffna) d\lambda$ equals
  \begin{align*}
    \int_{[X_{(1)}, t]} \lp \vvnat{t} + \vvna - \vvnt{t} \rp \ffn d\lambda
    - \int_{[X_{(1)}, t]} \lp \vvnat{\tau^0_-} + \vvna - \vvnat{\tau^0_-} \rp
    \ffna d\lambda
  \end{align*}
  and since $\int \vvnat{\tau^0_-} d(\Fn - \FFna) = 0$, this equals
  \begin{align*}
    \MoveEqLeft
    \int_{[X_{(1)}, t]} \vvnat{t}\ffn d\lambda
    + \vvna(t) \FFn(t)
    - \bigg[ \int \vvnat{\tau^0_-} d \Fn
    + \int_{[X_{(1)}, \tau^0_-]}
    \vvna(\tau^0_-) \ffna d\lambda
    \\
    & \hspace{13em} + \int_{(\tau^0_-, t]} \vvna \ffna
    d\lambda \bigg],
  \end{align*}
  which equals
  \begin{align*}
    \MoveEqLeft \int_{[X_{(1)}, t]} \vvnat{t} \ffn d\lambda
    + \vvna(t) \FFn(t)
    - \bigg[ \int_{[X_{(1)}, t]} \vvnat{t} d \Fn \\
    & + \int_{[X_{(1)}, t]} (\vvnat{\tau^0_-} - \vvnat{t} ) d\Fn
    + \vvna(\tau^0_-) \FFna(\tau^0_-)
    + \int_{(\tau^0_-, t]} \vvna \ffna d\lambda
    \bigg]
  \end{align*}
  which equals
  \begin{align*}
    \MoveEqLeft     \int_{[X_{(1)}, t]} \vvnat{t} d(\FFn - \Fn )
    + \vvna(t) \FFn(t)
    - \bigg[
    \int_{[X_{(1)}, \tau^0_-]}    (\vvna(t) - \vvna(\tau^0_-) ) d\Fn \\
    & \quad - \int_{(\tau^0, t]} \vvna d\Fn
    + \int_{(\tau^0_-, t]} \vvna(t) d\Fn
    + \vvna(\tau^0_-) \FFna(\tau^0_-)
    + \int_{(\tau^0_-, t]} \vvna \ffna d\lambda
    \bigg]
  \end{align*}
  which equals
  \begin{align*}
    \MoveEqLeft    \int_{[X_{(1)}, t]} \vvnat{t} d(\FFn - \Fn)
    + \int_{(\tau^0_-, t]} \vvna d (\Fn - \FFna)
    + \vvna(t) \FFn(t)
    +  \vvna(\tau^0_-) \Fn(\tau^0_-) \\
    & \quad - \vvna(t) \Fn(\tau^0_-)
    - \vvna(t) (\Fn(t) - \Fn(\tau^0_-))
    - \vvna(\tau^0_-) \FFna(\tau^0_-)
  \end{align*}
  which equals
  \begin{align*}
    \MoveEqLeft
    \int_{[X_{(1)}, t]} \vvnat{t} d(\FFn - \Fn) + \int_{(\tau^0_-,t]} \vvna
    d(\Fn - \FFna) \\
    & + \vvna(t) (\FFn(t) - \Fn(t)) + \vvna(\tau^0_-) (
    \Fn(\tau^0_-) - \FFna(\tau^0_-))
  \end{align*}
  which equals
  \begin{align*}
    \MoveEqLeft    \int_{[X_{(1)}, t]} \widehat{\vp}_{n,t}^0 d \lp \FFn - \Fn \rp +
    \int_{(\tau^0_-, t]} \widehat{\vp}_{n,t}^0 d\lp \Fn - \FFna \rp +
    \vvna(t) ( \FFn(t) - \FFna(t))  \\
    & \quad + (\vvna(\tau^0_-)  - \vvna(t) ) (\Fn(\tau^0_-) -
    \FFna(\tau^0_-)),
  \end{align*}
  as desired.
\end{proof}

\medskip

Define
$T_{n,t}^i \equiv T_n^i$, $i=1,2$, by \label{rem:defn:Tni}
\begin{align}
  \label{eq:rem:defn:Tni}
  T_n^1 =
  \int_{[X_{(1)}, t]} \frac{ (\vvn - \vvna)^2}{2}
  e^{\epsilon_{n}^1}  \ffna
  d\lambda
  \quad \mbox{ and } \quad
  T_n^2 =
  \int_{[X_{(1)}, t]} \frac{(\vvn - \vvna)^2}{2}
  e^{\epsilon_{n}^2} \ffn d\lambda,
\end{align}
so that
\begin{equation}
  \label{eq:rem:Rn1tc-main-decomposition}
  R_{n,1,t}^c = A_{n,t}^1 + E_{n,t}^1 - T_{n,t}^1
  = A_{n,t}^2 + E_{n,t}^2 + T_{n,t}^2.
\end{equation}
by \eqref{eq:rem:Rn1tc-ffna-decomp1} and \eqref{eq:rem:Rn1tc-ffn-decomp1}.
Recall (from page~\pageref{page:Ani-inequalities}) that $A_n^1 \le 0 \le A_n^2$.
Thus
\begin{equation}
  \label{eq:rem:main-An2-ineq}
  E_n^1 - E_n^2 \ge E_n^1 - E_n^2 - T_n^2 - T_n^1
  = A_n^2 - A_n^1 \ge
  \begin{cases}
    & A_n^2 \ge 0,\\
    & -A_n^1 \ge 0.
  \end{cases}
\end{equation}
To see that $R_{n,1}^c  = O_p(n^{-1})$ we need to see that $E_n^i, A_n^i,$
and $T_n^i$ are each $O_p(n^{-1})$ (for, say, $i=1$).
We can see already that $E_n^1 - E_n^2 = O_p(n^{-1})$ (by direct analysis of
the terms in $E_n^1-E_n^2$ from \eqref{eq:rem:defn:Ent1} and
\eqref{eq:rem:defn:Ent2}), which yields that $A_n^1$ and $A_n^2$ are both
$O_p(n^{-1})$.  However, it is clear that we also need to analyze $T_n^1 +
T_n^2$ to understand $R_{n,1}^c$.  We need to show that $T_{n,t}^1 +
T_{n,t}^2$ is $O_p(n^{-1})$ to see that $R_{n,1,t}^c = O_p(n^{-1})$; but we
will also be able to use that $T_{n,t}^1+T_{n,t}^2=O_p(n^{-1})$ to then find
$t^*$ values such that $T_{n,t^*}^1+T_{n,t^*}^2=o_p(n^{-1})$, which will
allow us to argue in fact that $E_{n,t^*}^1+E_{n,t^*}^2 = o_p(n^{-1})$
(rather than just $O_p(n^{-1})$), and thus that $R_{n,1,t^*}^c =
o_p(n^{-1})$, as is eventually needed.  Thus, we will now turn our attention
to studying $T_n^1 + T_n^2$.  Afterwards, we will study
\begin{equation}
  \label{eq:rem:Rn1tc-average-decomp}
  R_{n,1,t}^c = (A_{n,t}^1 + E_{n,t}^1 - T_{n,t}^1
  + A_{n,t}^2 + E_{n,t}^2 + T_{n,t}^2)/2,
\end{equation}
from \eqref{eq:rem:Rn1tc-main-decomposition}.  From seeing
$T_{n,t^*}^1+T_{n,t^*}^2 = o_p(n^{-1})$, we
will be able to conclude that $A_{n,t^*}^1+A_{n,t^*}^2$ and
$E_{n,t^*}^1+E_{n,t^*}^2$ are also $o_p(n^{-1})$, as desired.
Then we can conclude $R_{n,1, t^*}^c = o_p(n^{-1})$.

\subsubsection{Show  ${T}_n^i = O_p(n^{-1})$, $i=1,2$}

The next lemma shows that terms that are nearly identical to $T_{n}^i$ are
$O_p(n^{-1})$.  The difference between the integrand in the terms in the
lemma and the integrand defining $T_n^i$ is that $\epsilon_n^i$ is replaced
by a slightly different $\tilde \epsilon_n^i$.
Previously, we considered $t$ to be fixed, whereas now we will have it vary
with $n$.
\begin{lemma}
  \label{lem:rem:step2:Tni-Opn-1}
  Let $t_n  < m $ be a (potentially random) sequence such that
  \begin{equation}
    \label{eq:rem:2}
    t_n \le \max
    \lp S(\vvn) \cup S(\vvna)  \rp \cap (-\infty, m).
  \end{equation}
  Let
  \begin{equation}
    \label{eq:rem:defn:tilde-Tni}
    \tilde T_{n,t}^1 = \int_{[X_{(1)}, t_n]}
    \lp \vvna - \vvn \rp^2 e^{\tilde \epsilon_{n}^1} \ffna d
    \lambda
    \;     \mbox{ and } \;
    \tilde T_{n,t}^2
    = \int_{[X_{(1)}, t_n]} \lp \vvna - \vvn \rp^2 e^{ \tilde \epsilon_{n}^2} \ffn d
    \lambda,
  \end{equation}
  where $\tilde \epsilon_{n}^1(x)$ lies between $\vvn(x)-\vvna(x)$ and $0$,
  and $\tilde \epsilon_{n}^2(x)$ lies between $\vvna(x)-\vvn(x)$ and $0$, and
  are defined in \eqref{eq:rem:second-I-plus-II} in the proof.  Then we have
  \begin{equation}
    \label{eq:rem:3}
    \tilde T_{n,t_n}^i =
    O_p(n^{-1})
    \quad \text{ for } i =1,2.
  \end{equation}
\end{lemma}
\begin{proof}
  For a
  function $f(x)$, recall the notation $f_{s}(x) = f(x) - f(s)$ for $x \le s$
  and $f_s(x)=0$ for $x \ge s$.
  Let $\tau \in S(\vvn)$ and $\tau^0 \in S(\vvna)$,
  and assume that
  \begin{equation}
    \label{eq:rem:defn:tau-tau0}
    \tau \le \tau^0  < m.
  \end{equation}
  (The argument is symmetric in $\vvn$ and $\vvna$, so we may assume this
  without loss of generality.)  We will show the lemma holds for the case $t_n =
  \tau^0$, and then the general $t_n \le \tau^0$ case follows since the integral is
  increasing in $t_n$.
  Now, because $\vvnat{\tau}$ is concave, for
  $\epsilon \le 1$, the function $\vvn(x) + \epsilon(\vvnat{\tau}(x) -
  \vvnt{\tau}(x))$ is concave. So, by Theorem 2.2, page 43, of
  \cite{DR2009LC}, we have
  \begin{equation}
    \label{eq:rem:old-I}
    \int (\vvnat{\tau} -  \vvnt{\tau}) d(\Fn - \FFn) \le 0.
  \end{equation}
  Similarly, if $\tau^0 $ is a knot of $\vvna$ and is less than the mode,
  then since $\vvna(x) + \epsilon \lp \vvnt{\tau^0}(x) -\vvnat{\tau^0}(x)
  \rp$ is concave with mode at $m$ for $\epsilon$ small (since
  $\vvnt{\tau^0}(x) -\vvnat{\tau^0}(x)$ is only nonzero on the left side of
  the mode), by the characterization
  Theorem 2.2 B of
  \cite{Doss-Wellner:2016ModeConstrained},
  we have
  \begin{equation*}
    \int (\vvnt{\tau^0} - \vvnat{\tau^0}) d(\Fn-\FFna) \le 0.
  \end{equation*}
  Then setting \label{page:def-II-L-tau} $II_{n,\tau^0}^L := \int_{[X_{(1)}, \tau^0]} \lp \vvn - \vvna \rp d\lp
  \Fn - \FFna \rp$, we have
  \begin{align}
    \label{eq:rem:4}
    0 &  \ge \int_{[X_{(1)}, \tau^0]} \lp   \vvnt{\tau^0} - \vvnat{\tau^0}\rp d\lp \Fn -
    \FFna \rp   \\
    & = %
    II_{n,\tau^0}^L
    - \lp \vvn(\tau^0) - \vvna(\tau^0) \rp \lp \Fn(\tau^0) - \FFna(\tau^0)
    \rp. \label{eq:rem:5}
  \end{align}
  And setting $I_{n,\tau^0}^L := \int_{[X_{(1)}, \tau^0]} \lp \vvna - \vvn \rp d\lp \Fn -
  \FFn \rp $, we have
  \begin{equation}
    \label{eq:rem:decompose-IL-tau}
    \begin{split}
      I_{n,\tau^0}^L
      & =  \int_{[X_{(1)}, \tau]} \lp \vvna(u) - \vvna(\tau) \rp d \lp \Fn - \FFn
      \rp(u) \\
      & \quad - \int_{[X_{(1)}, \tau]} \lp \vvn(u)-\vvn(\tau)  \rp d\lp
      \Fn-\FFn\rp(u) \\
      & \quad + \lp \vvna(\tau) - \vvn(\tau) \rp \int_{[X_{(1)}, \tau]} d \lp \Fn - \FFn \rp \\
      & \quad +   \int_{(\tau, \tau^0]} \lp \vvna - \vvn \rp d\lp \Fn - \FFn \rp ,
    \end{split}
  \end{equation}
  and, since the first two summands together yield the left hand side of
  \eqref{eq:rem:old-I},
  we have
  \begin{equation}
    \label{eq:rem:InL-upperbound}
    I_{n,\tau^0}^L \le
    \lp \vvna(\tau) - \vvn(\tau) \rp \int_{[X_{(1)}, \tau]} d \lp \Fn - \FFn \rp
    +   \int_{(\tau, \tau^0]} \lp \vvna - \vvn \rp d\lp \Fn - \FFn \rp.
  \end{equation}
  \label{page:new-analysis-squared-integral}
  Now, we apply \eqref{eq:f-to-varphi-taylor-1} of
  Lemma~\ref{lem:f-to-varphi-taylor} to see that
  \begin{align}
    I_{n,\tau^0}^L +II_{n,\tau^0}^L
    & =
    \int_{[X_{(1)}, \tau^0]} \lp \vvna - \vvn \rp d \lp \FFna - \FFn
    \rp \nonumber \\
    & =
    \begin{cases}
      & \int_{X_{(1)}}^{\tau^0} \lp \vvna - \vvn \rp^2 e^{\tilde \epsilon_{n}^1} \ffna d
      \lambda \ge 0,  \\
      & \int_{X_{(1)}}^{\tau^0} \lp \vvna - \vvn \rp^2 e^{\tilde \epsilon_{n}^2} \ffn d
      \lambda \ge 0,
    \end{cases} \label{eq:rem:second-I-plus-II}
  \end{align}
  where $\tilde \epsilon_{n}^2(x)$ lies between $\vvna(x)-\vvn(x)$ and $0$
  and $\tilde \epsilon_{n}^1(x)$ lies between $\vvn(x)-\vvna(x)$ and $0$.  By
  \eqref{eq:rem:5} and \eqref{eq:rem:InL-upperbound}, \eqref{eq:rem:second-I-plus-II} is
  bounded above by
  \begin{equation}
    \label{eq:rem:I-plus-II-error}
    \begin{split}
      & \lp \vvna(\tau) - \vvn(\tau) \rp \int_{[X_{(1)}, \tau]} d \lp \Fn - \FFn \rp
      +   \int_{(\tau, \tau^0]} \lp \vvna - \vvn \rp d\lp \Fn - \FFn \rp   \\
      & \quad +   \lp \vvn(\tau^0) - \vvna(\tau^0) \rp \lp \Fn(\tau^0) - \FFna(\tau^0) \rp.
    \end{split}
  \end{equation}
  By
  Proposition~7.1 of \cite{Doss-Wellner:2016ModeConstrained}
  and Lemma~4.5 of \cite{BRW2007LCasymp},  %
  $\sup_{t \in [\tau, \tau^0]}
  \left| \vvna(t)-\vvn(t) \right| = O_p(n^{-2/5})$.  By Corollary~2.5 of
  \cite{DR2009LC},
  $$\left| \int_{[X_{(1)}, \tau]} d \lp \Fn - \FFn \rp
  \right| \le 1/n, $$
  so the first term in the above display is
  $O_p(n^{-7/5})$.  Similarly, by
  Corollary~2.7B of \cite{Doss-Wellner:2016ModeConstrained},
  $\left| \Fn(\tau^0) - \FFna(\tau^0) \right| \le 1/n$, so the last term in
  the previous display is $O_p(n^{-7/5})$.  \label{page:Op-seven-fifths}
  We
  can also see that the middle term in the previous display equals
  \begin{equation}
    \label{eq:rem:square-integral-errorterms}
    \begin{split}
      \MoveEqLeft  ( \widehat{\varphi}_n^0 - \widehat{\varphi}_n )(\tau^0) ( \FF_n - \widehat{F}_n )(\tau^0)
      - ( \widehat{\varphi}_n^0 - \widehat{\varphi}_n )(\tau) ( \FF_n - \widehat{F}_n )(\tau)\\
      & \ \  -  \  \int_{(\tau, \tau^0]} ( \FF_n - \widehat{F}_n ) (  \widehat{\varphi}_n^0 - \widehat{\varphi}_n )^{\prime}  d \lambda.
    \end{split}
  \end{equation}
  Now the middle term in the previous display is $O_p(n^{-7/5})$.  For the
  last two terms, we %
  apply Lemma~\ref{lem:rem:FFn-Fn-empirical-proc-arg}
  taking  $I = [\tau, \tau^0]$ to see that
  $$\sup_{t \in (\tau, \tau^0]} n^{3/5} \left| \int_{(\tau,t]} d  \lp     \Fn
    - \FFn \rp \right| = O_p(1).$$
  Thus, using Proposition~7.1 of \cite{Doss-Wellner:2016ModeConstrained} and
  Lemma~4.5 of \cite{BRW2007LCasymp}, we have
  \begin{equation}
    \label{eq:rem:lemma-error-integ-by-parts}
    \int_{(\tau, \tau^0]} ( \FF_n - \widehat{F}_n )
    (  \widehat{\varphi}_n^0 - \widehat{\varphi}_n )^{\prime}  d \lambda
    = O_p(n^{-4/5}) \int_{(\tau,\tau^0]} d\lambda = O_p(n^{-1}),
  \end{equation}
  so we have now shown that \eqref{eq:rem:square-integral-errorterms} is
  $O_p(n^{-1})$, so the middle term in \eqref{eq:rem:I-plus-II-error} is
  $O_p(n^{-1})$.  Thus, \eqref{eq:rem:I-plus-II-error} is $O_p(n^{-1})$, and
  since \eqref{eq:rem:I-plus-II-error} bounds \eqref{eq:rem:second-I-plus-II}
  we can conclude that
  \begin{equation}
    \label{eq:rem:6}
    \int_{X_{(1)}}^{\tau^0} \lp \vvna - \vvn \rp^2 e^{\tilde \epsilon_{n}^1} \ffna d
    \lambda =
    \int_{X_{(1)}}^{\tau^0} \lp \vvna - \vvn \rp^2 e^{\tilde \epsilon_{n}^{2}} \ffn d
    \lambda  = O_p(n^{-1}),
  \end{equation}
  and so we are done.
\end{proof}

\begin{remark}
  Note that if we computed the integrals
  in \eqref{eq:rem:6} %
  over an interval
  of length $O_p(n^{-1/5})$, by using that the corresponding integrand is
  $O_p(n^{-4/5})$ (under smoothness/curvature assumptions), the integrals would
  be $O_p(n^{-1})$.  However, \eqref{eq:rem:6} shows that the integrals are
  $O_p(n^{-1})$ over a larger interval whose length is constant or larger, with high probability. Thus we can use
  \eqref{eq:rem:6} to show that $\vvna - \vvn$ must be of order smaller than $
  O_p(n^{-2/5})$ somewhere, and this line of reasoning will in fact show that
  $T_{n,t}^1$ and $T_{n,t}^2$ are $o_p(n^{-1})$ for certain $t$ values.
\end{remark}

\begin{remark}
  Having shown \eqref{eq:rem:3}, it may seem that we can easily find a
  subinterval over which the corresponding integrals are $o_p(n^{-1})$ (or
  smaller), and that this should allow us to quickly finish up our proof.
  There is an additional difficulty, though, preventing us from naively
  letting $|t| \to \infty$: we need to control the corresponding integrals
  actually within small neighborhoods of $m$ (of order $O_p(n^{-1/5})$), not
  just arbitrarily far away from $m$.  This is because our asymptotic results
  for the limit distribution take place in $n^{-1/5}$ neighborhoods of $m$.
\end{remark}
To connect the result about $\tilde{T}_n^i$ to the title of this section
(which states $T_{n}^i = O_p(n^{-1})$), note that by
Lemma~\ref{lem:rem:epsilon-inequality}, $0 \le T_n^i \le 2 \tilde{T}_n^i =
O_p(n^{-1})$.

\subsubsection{Local and Global $o_p(n^{-1})$ Conclusion}

We will now find a subinterval $I$ such that
\begin{equation*}
  \int_{I} \lp \vvna - \vvn \rp^2 e^{\tilde \epsilon_{n}^2} \ffn d
  \lambda  = o_p(n^{-1}).
\end{equation*}
We will argue by partitioning a larger interval over which the above integral
is $O_p(n^{-1})$ into smaller subintervals.  Let $\epsilon > 0$.
Let $L > 0$ be such that
intervals of length $L n^{-1/5}$ whose endpoints converge to $m$ contain
a knot from each of $\vvn$ and $\vvna$ with probability $ 1- \epsilon$.
Also let  $\delta > 0$ and $\zeta = \delta / L$
which we take without loss of
generality to be the reciprocal of an integer. By
Proposition~7.3 of \cite{Doss-Wellner:2016ModeConstrained}, fix $M \ge L $ large
enough such that with probability
$1 - \epsilon \zeta$
for any random
variable $\xi_n \to_p m$, $[\xi_n - M n^{-1/5}, \xi_n + M n^{-1/5}]$ contains
knots of both $\vvn$ and of $\vvna$, when $n$ is large enough.
Now, each of the
intervals
\begin{equation*}
  I_{jn} := ( \tau^0 - Mjn^{-1/5},  \tau^0 - M(j-1)n^{-1/5}] \, \mbox{ for
  } \, j=1, \ldots, 1/\zeta
\end{equation*}
contains  a knot of $\vvn$ and of $\vvna$ by
taking $\xi_n$ to be
$\tau^0 - M j n^{-1/5}.$ %
There are $1 / \zeta $ such
intervals so the probability that all intervals contain a knot of both $\vvn$
and $\vvna $ is $1 - \epsilon$.
Now, let $K= O_p(1)$ be such that
$\int_{X_{(1)}}^{\tau^0}
\lp \vvna - \vvn
\rp^2 e^{\tilde \epsilon_{n}^2} \ffn d \lambda \le K n^{-1}$ for $\tau^0 < m$,
by Lemma~\ref{lem:rem:step2:Tni-Opn-1}.
In particular,
\begin{equation}
  \label{eq:rem:14}
  \int_{I_{j^*}} \lp \vvna - \vvn \rp^2 e^{\tilde \epsilon_{n}^2} \ffn d
  \lambda  :=
  \min_{j=1,\ldots,1/\zeta} \int_{I_j} \lp \vvna - \vvn \rp^2 e^{\tilde \epsilon_{n}^2} \ffn d
  \lambda  \le \zeta K n^{-1}.
\end{equation}
We next conclude by Lemma~\ref{lem:rem:local-to-global-square-integral},
since $\zeta = \delta / L$,  that
there exists a subinterval $J^* \subset I_{j^*}$ containing knots $\eta \in
S(\vvn)$ and $\eta^0 \in S(\vvna)$, such that
\begin{equation}
  \label{eq:rem:local-op-bound}
  \sup_{x \in J^*} |\vvn(x)-\vvna(x)| \le c \delta_1 n^{-2/5}
  \quad \mbox{ and } \quad
  | \eta^0 - \eta | \le c \delta_1 n^{-1/5}
\end{equation}
for a universal constant $c>0$ and where $\delta_1 \to 0$ as $\delta \to 0$.

We can now re-apply the proof of Lemma~\ref{lem:rem:step2:Tni-Opn-1}, this
time taking as our knots $\eta$ and $\eta^0$, and again assuming without loss
of generality $\eta \le \eta^0$. We again see that
\eqref{eq:rem:second-I-plus-II} is bounded above by
\eqref{eq:rem:I-plus-II-error}, and the middle term of
\eqref{eq:rem:I-plus-II-error} is bounded by
\eqref{eq:rem:square-integral-errorterms}.  Using
\eqref{eq:rem:local-op-bound}, we can conclude by
\eqref{eq:rem:lemma-error-integ-by-parts} that
\eqref{eq:rem:square-integral-errorterms} is bounded by $\delta_2
O_p(n^{-1})$, so \eqref{eq:rem:I-plus-II-error} is also, and so
\eqref{eq:rem:second-I-plus-II} is also, where $\delta_2 \to 0$ as $\delta
\to 0$.  We can conclude
\begin{equation*}
  \int_{X_{(1)}}^{\eta^0} \lp \vvna - \vvn \rp^2 e^{\tilde \epsilon_{n}^2} \ffn d
  \lambda
  \le  \tilde \delta n^{-1}.
\end{equation*}
Now $\eta^0 \ge \tau^0 - M n^{-1/5} / \zeta$, the endpoint of
$I_{1/\zeta,n}$. Thus,  take $b n^{-1/5} \ge \tau^0 - M n^{-1/5} / \zeta$, let $t_n =
m - bn^{-1/5}$  and now let
$J^* = [t_n - \tilde L n^{-1/5}, t_n]$ where $\tilde L = \max( L , 8 D /
\vvo^{(2)}(m))$, chosen so that we can apply
Lemma~\ref{lem:rem:local-to-global-square-integral}.
Then
\begin{equation}
  \label{eq:rem:Tni-op-epsilon-tildes}
  \int_{X_{(1)}}^{t_n} \lp \vvna - \vvn \rp^2 e^{\tilde \epsilon_{n}^2} \ffn d
  \lambda
  \le  \tilde \delta n^{-1}
\end{equation}
with high probability.  Analogously,
\begin{equation}
  \label{eq:rem:Tn1-op-epsilon-tildes}
  \int_{X_{(1)}}^{t_n} \lp \vvna - \vvn \rp^2 e^{\tilde \epsilon_{n}^1} \ffna d
  \lambda
  \le  \tilde \delta n^{-1}
\end{equation}
as $n \to \infty$ with high probability.
And
we can apply Lemma~\ref{lem:rem:local-to-global-square-integral}
to the interval $J^*$ to see
\begin{align}
  \Vert \vvn - \vvna \Vert_{J^*}  = \delta O_p(n^{-2/5}) , &
  \quad
  \quad
  & \Vert \ffn - \ffna
  \Vert_{J^*} \le \delta K n^{-2/5},
  \label{eq:rem:vv-ff-op-statements} \\
  \Vert \FFn - \FFna \Vert_{J^*} \le
  \delta K n^{-3/5}, &
  \quad
  \mbox{ and }
  \quad
  & |\tau - \tau^0 | \le \delta
  K n^{-1/5} , \label{eq:rem:FF-knot-op-statements}
\end{align}
where $\tau \in S(\vvn) \cap J^*$ and $\tau^0 \in S(\vvna) \cap
J^*$,  \label{page:rem:defn:tau-tau0}
and
\begin{equation}
  \label{eq:rem:vv-prime-op-statements}
  \Vert (\vvn - \vvna)' \Vert_{[\max(\tau,\tau^0) + \delta O_p(n^{-1/5}), t_n
    - \delta O_p(n^{-1/5})]} = \delta K n^{-1/5};
\end{equation}
here, $K  = O_p(1)$ and depends on $\epsilon$ and $\tilde L \equiv \tilde
L_\epsilon$, but not on $\delta$ or $t_n$.
Thus when we eventually let $\tilde \delta \to 0$, so $b
\equiv b_{\tilde \delta} \to \infty$, we can still conclude $\tilde \delta K \to 0$.
\begin{mylongform}
  \begin{longform}
    Note we cannot assume that $\vvn$ or $\vvna$ are linear on
    $[\max(\tau,\tau^0), \sup J^*]$.
  \end{longform}
\end{mylongform}
We also continue to assume, without loss of generality, that
\begin{equation*}
  \tau \le \tau^0.
\end{equation*}
Thus, here is the sense in which we mean $o_p$, for the remainder of the proof:
if we say, e.g., $E_{n,t_n}^1-E_{n,t_n}^2=o_p(n^{-1})$ we mean for any $\tilde
\delta > 0$, we may set $t_n = m-bn^{-1/5}$ and choose $b$ large enough that
$| E_{n,t_n}^1-E_{n,t_n}^2| \le  \tilde \delta K n^{-1}$ where $K$ does not
depend on $t_n$.

We can now conclude that
\begin{align}
  \label{eq:rem:tilde-Tni-op}
  \tilde T_{n,t_n}^i =
  o_p(n^{-1})
  \quad \mbox{ for } \quad
  i=1,2,
\end{align}
The difference in the definitions of $T_n^i$ (defined in
\eqref{eq:rem:defn:Tni}) and $\tilde T_n^i$ (defined in
\eqref{eq:rem:defn:tilde-Tni}), for $i=1,2$, is only in the
$e^{\epsilon^i_n}$'s and $e^{\tilde \epsilon_n^i}$'s.  These arise from
Taylor expansions of the exponential function. The definition of $T_n^i$
arises from the expansions of $R_{n,1,t}^c$ (see
\eqref{eq:rem:Rn1tc-ffna-decomp1} and \eqref{eq:rem:Rn1tc-ffn-decomp1}).
Thus, if we let $ e^x = 1 + x + 2^{-1} x^2 e^{\epsilon(x)}$ we can see that
$\epsilon_n^1(x) = \epsilon( \vvn(x) - \vvna(x))$ and $\epsilon^2_n(x) =
\epsilon(\vvna(x)-\vvn(x))$.  Let $ e^x = 1 + x e^{\tilde \epsilon(x)}.$ Then
we can see that $\tilde \epsilon_n^1(x) =\tilde \epsilon( \vvn(x) -
\vvna(x))$ and $\tilde \epsilon^2_n(x) = \tilde \epsilon(\vvna(x)-\vvn(x))$.
Now, by Lemma~\ref{lem:rem:epsilon-inequality}, for all $x \in \RR$, $
e^{\epsilon(x)} \le 2 e^{\tilde \epsilon(x)},$ so that
\begin{equation}
  \label{eq:rem:Tni-op-switch-epsilons}
  0 \le   T_{n,t_n}^i \le 2 \tilde T_{n,t_n}^i = o_p(n^{-1}), \mbox{ for } i =1,2,
\end{equation}
by  \eqref{eq:rem:tilde-Tni-op}.

\subsubsection{Return to $R_{n,1,t}^c$}

We take $t_n$ and $J^*$ as defined at the end of the previous section.
Now, if we could show that $E_{n,t_n}^1 - E_{n,t_n}^2 = o_p(n^{-1})$ then from
\eqref{eq:rem:main-An2-ineq} we could conclude that $A_{n,t_n}^i$, $i=1,2$, are
both $o_p(n^{-1})$.  If, in addition, we can show $E_{n,t_n}^1 + E_{n,t_n}^2 =
o_p(n^{-1})$, then since
\begin{equation}
  \label{eq:rem:2}
  R_{n,1,t_n}^c = (E_{n,t_n}^1 + E_{n,t_n}^2 + A_{n,t_n}^1 + A_{n,t_n}^2 + T^2_{n,t_n} -
  T_{n,t_n}^1 ) / 2
\end{equation}
by \eqref{eq:rem:Rn1tc-main-decomposition}, we could conclude $R_{n,1,t_n}^c =
o_p(n^{-1})$.  Unfortunately it is difficult to get any results about
$E_{n,t_n}^1-E_{n,t_n}^2$.
We can analyze $E_{n,t_n}^1 + E_{n,t_n}^2$, though.
The next lemma shows that the difficult terms in
$E_{n,t_n}^1 + E_{n,t_n}^2$ are $o_p(n^{-1})$.

\begin{lemma}
  \label{lem:rem:Fni-op}
  Let all terms be as defined above. For any $t < m $ let $F_{n,t}^1 =
  \int_{(\tau, t]} \vvnt{t} d(\FFn - \Fn)$ and $F_{n,t}^2 = \int_{(\tau^0,t]}
  \vvnat{t} d(\Fn - \FFna)$.  Then
  \begin{equation*}
    F_{n,t_n}^1 + F_{n,t_n}^2 = o_p(n^{-1}).
  \end{equation*}
\end{lemma}
\begin{proof}
  For the proof, denote $t \equiv t_n$ and recall that we assume $\tau \le
  \tau^0$.  We see
  \begin{align*}
    \MoveEqLeft \int_{(\tau, t]} \vvnt{t} d(\FFn - \Fn) + \int_{(\tau^0, t]}
    \vvnat{t} d(\Fn - \FFna) \\
    & = \int_{(\tau^0,t]} \vvnt{t} d(\FFn - \Fn) + \int_{(\tau, \tau^0]}
    \vvnt{t} d(\FFn -\Fn) \\
    & \quad - \lp \int_{(\tau^0,t]} \vvnt{t}  d(\FFna - \Fn)
    + \int_{(\tau^0,t]} (\vvnat{t}-\vvnt{t} ) d(\FFna-\Fn) \rp
  \end{align*}
  which equals
  \begin{align}
    \label{eq:int_tau0-t-vvntt}
    \int_{(\tau^0, t]} \vvnt{t} (\ffn-\ffna) d\lambda
    - \int_{(\tau^0, t]} (\vvnat{t} - \vvnt{t} ) d(\FFn - \Fn)
    + \int_{(\tau, \tau^0]} \vvnt{t} d(\FFn - \Fn).
  \end{align}
  Note $\Vert \vvnat{t} \Vert_{J^*} = O_p(n^{-2/5})$.  This follows because
  $n^{1/5}(\tau^0 - t) = O_p(1)$ by Proposition~7.3
  of
  \cite{Doss-Wellner:2016ModeConstrained}, and because $\Vert (\vvna)'
  \Vert_{J^*} = n^{-1/5} O_p(1)$ by Corollary~7.1 of
  \cite{Doss-Wellner:2016ModeConstrained}, since $\vvo'(m)=0$. In both cases
  the $O_p(1)$ does not depend on $t_n$. Thus, the first term in
  \eqref{eq:int_tau0-t-vvntt} is $o_p(n^{-1})$, since $ \Vert \ffn - \ffna
  \Vert_{J^*} = o_p(n^{-2/5})$.  We will rewrite the other two terms of
  \eqref{eq:int_tau0-t-vvntt} with integration by parts.  The negative of the
  middle term, $ \int_{(\tau^0,t]} (\vvnat{t} - \vvnt{t}) d(\FFn - \Fn)$,
  equals
  \begin{equation}
    \label{eq:ffn-fnvvn-vvnttt}
    ((\FFn-\Fn)(\vvnat{t}-\vvnt{t}))(\tau^0,t]
    - \int_{(\tau^0,t]} (\FFn - \Fn ) ((\vvna)' - \vvn')
    d\lambda.
  \end{equation}
  Note $\Vert \FFn - \Fn \Vert_{J^*} = O_p(n^{-3/5})$ by
  Lemma~\ref{lem:rem:FFn-Fn-empirical-proc-arg}.  Thus the first term in
  \eqref{eq:ffn-fnvvn-vvnttt} is $o_p(n^{-1})$ because $\Vert \vvna - \vvn
  \Vert_{J^*} = o_p(n^{-2/5})$. The second term in
  \eqref{eq:ffn-fnvvn-vvnttt} is $o_p(n^{-1})$ because
  \eqref{eq:rem:vv-prime-op-statements} implies that $\int_{[\tau^0,t]} |
  (\vvn - \vvn)'| d\lambda = o_p(n^{-2/5})$, and, as already noted, $\Vert
  \FFn - \Fn \Vert_{J^*} = O_p(n^{-3/5})$.
  Thus \eqref{eq:ffn-fnvvn-vvnttt} is
  $o_p(n^{-1})$.

  We have left the final term of \eqref{eq:int_tau0-t-vvntt}.  This can be
  bounded by
  \begin{align}
    \label{eq:lv-ffn-fn}
    \lv (\FFn - \Fn)( \vvnt{t}) (\tau, \tau^0]
    - \int_{(\tau, \tau^0]} (\FFn - \Fn) \vvn' d\lambda \rv,
  \end{align}
  and the second term above bounded by $\Vert \FFn - \Fn \Vert_{[\tau,
    \tau^0]} \vvn'(\tau+) (\tau^0-\tau) = o_p(n^{-1})$ because $\Vert \FFn -
  \Fn \Vert_{J^*} = O_p(n^{-3/5})$, $\vvn'(\tau+) = O_p(n^{-1/5})$ (since
  $\vvo'(m)=0$, and recall $\vvn$ is linear on $[\tau, \tau^0]$), and
  $(\tau^0-\tau) = o_p(n^{-1/5})$.  The first term of \eqref{eq:lv-ffn-fn} is
  $o_p(n^{-1})$ because in fact $\Vert \FFn - \Fn \Vert_{[\tau, \tau^0]} =
  o_p(n^{-3/5})$, by Lemma~\ref{lem:rem:FFn-Fn-empirical-proc-arg} (since
  $|\tau-\tau^0| = o_p(n^{-1/5})$), and $\Vert \vvnt{t} \Vert_{[\tau,
    \tau^0]} = O_p(n^{-2/5})$.  Thus \eqref{eq:int_tau0-t-vvntt} is
  $o_p(n^{-1})$ so we are done.
\end{proof}

For $t < m$, define
\begin{equation*}
  \label{eq:rem:defn:Gnt1}
  G_{n,t}^1 =
  \vvn(t) ( \FFn(t) - \FFna(t) )
  + (\vvn(\tau)    - \vvn(t)) ( \FFn(\tau) - \Fn(\tau) )
\end{equation*}
and
\begin{equation*}
  \label{eq:rem:defn:Gnt2}
  G_{n,t}^2 =
  \vvna(t) ( \FFn(t) - \FFna(t))
  + (\vvna(\tau^0)
  - \vvna(t)) (\Fn(\tau^0) - \FFna(\tau^0)),
\end{equation*}
so that $G_{n,t_n}^i = E_{n,t_n}^i -F_{n,t_n}^i$ for $i=1,2$ (recalling
the definitions of $E_{n,t_n}^i$ in
\eqref{eq:rem:defn:Ent1}
and
\eqref{eq:rem:defn:Ent2}).
The key idea now is that the first term in $G_{n,t}^i$ matches up with
$R_{n,1,\twoArgs}$.  To make this explicit, we need to define a one-sided
version of $R_{n,1,\twoArgs}$.  Since both $\ffn$ and $\ffna$
integrate to $1$, note for any $t_1 \le m \le t_2$, that
\begin{equation*}
  \label{eq:rem:Rn1-integral-identity}
  R_{n,1,\twoArgs} = - \vvo(m) \int_{D_{n,\twoArgs}^c} ( \ffn - \ffna) d\lambda;
\end{equation*}
thus, define
\begin{equation}
  \label{eq:rem:defn:Rn1t-one-sided}
  R_{n,1,t_1} = - \vvo(m) (\FFn(t_1) - \FFna(t_1)).
\end{equation}
The corresponding definition for the right side is $-\vvo(m)
\int_{t_2}^{X_{(n)}} (\ffn - \ffna ) d\lambda$, which when summed with
\eqref{eq:rem:defn:Rn1t-one-sided} yields $R_{n,1,\twoArgs}$.

\begin{lemma}
  \label{lem:rem:Gni-Rn1-op}
  Let all terms be as defined above.  We then have for $i=1,2$,
  \begin{equation*}
    G_{n,t_n}^i + R_{n,1,t_n} = o_p(n^{-1}).
  \end{equation*}
\end{lemma}
\begin{proof}
  The second terms in the definitions of $G_{n,t_n}^i$, $i=1,2$, are both
  $O_p(n^{-7/5})$ since $ \vvn(\tau) - \vvn(t)$ and $\vvna(\tau^0) -
  \vvna(t)$ are both $O_p(n^{-2/5})$ by Lemma~4.5 of \cite{BRW2007LCasymp}
  and Corollary~5.4 of \cite{Doss-Wellner:2016ModeConstrained}, and the terms
  $\FFn(\tau) -\Fn(\tau)$ and $\FFna(\tau^0) - \Fn(\tau^0)$ are both
  $O_p(n^{-1})$ by Corollary~2.4 and Corollary~2.12 of
  \cite{Doss-Wellner:2016ModeConstrained}.

  Thus we consider the first terms of $G_{n,t_n}^i$, in sum with
  $R_{n,1,t_n}$.   Consider the case $i=1$; we see that
  \begin{equation}
    \label{eq:rem:vvnt_n-ffnt_n-ffnat}
    \vvn(t_n) (\FFn(t_n) - \FFna(t_n)) + R_{n,1,t_n}
    = (\vvn(t_n) - \vvo(m)) (\FFn(t_n) - \FFna(t_n)),
  \end{equation}
  and $(\vvn(t_n) - \vvo(m)) = O_p(n^{-2/5})$ by Lemma~4.5 of
  \cite{BRW2007LCasymp} since $\vvo'(m)=0$.
  Crucially, we are not making a claim that $\vvn$ is close to $\vvo(m)$
  uniformly  over an interval, just a claim at the point $t_n$ satisfying
  $t_n \to m$, so the $O_p$ statement does not depend on $b$.
  Since  $\FFn(t_n) - \FFna(t_n) =
  o_p(n^{-3/5})$ by \eqref{eq:rem:FF-knot-op-statements}, we conclude that
  \eqref{eq:rem:vvnt_n-ffnt_n-ffnat} is $o_p(n^{-1})$.  Identical reasoning
  applies to the case $i=2$, using Corollary~5.4 of
  \cite{Doss-Wellner:2016ModeConstrained}. Thus we are done.
\end{proof}

Thus by Lemmas~\ref{lem:rem:Fni-op} and \ref{lem:rem:Gni-Rn1-op},
\begin{equation}
  \label{eq:rem:E1-+-E2-op}
  R_{n,1,t_n} +   (E^1_{n,t_n}  +  E^2_{n,t_n})/2 = o_p(n^{-1}).
\end{equation}

Now we decompose the $A^i_{n,t_n}$ terms.  Let
\begin{align*}
  & B_{n,t_n}^1 = \int \vvnt{\tau} d(\Fn - \FFna), \\
  & B_{n,t_n}^2 = \int \vvnat{\tau^0} d(\FFn - \Fn).
\end{align*}
where $\tau$ and $\tau^0$ are as previously defined (on
page~\pageref{page:rem:defn:tau-tau0})
and
\begin{align*}
  & C_{n,t_n}^1 = \int \vvn'(t_n-) (x-t_n)_- \, d(\Fn - \FFna) \\
  & C_{n,t_n}^2 = \int (\vvna)'(t_n-) (x-t_n)_- \, d(\FFn -  \Fn).
\end{align*}
Then, for $i=1,2$, by the definitions of $\tau$ and $\tau^0$,
\begin{align*}
  A_n^i = B_n^i + C_n^i,
\end{align*}
and note that
\begin{equation}
  \label{eq:rem:1}
  B_{n,t_n}^1 = \int (\vvnt{\tau} - \vvnat{\tau^0}) d(\Fn - \FFna)
  \quad \mbox{ and } \quad
  B_{n,t_n}^2 = \int (\vvnat{\tau^0} - \vvnt{\tau} ) d(\FFn  - \Fn),
\end{equation}
by the characterization theorems, Theorem~2.2 of
\cite{Doss-Wellner:2016ModeConstrained} with $\Delta = \pm \vvnat{\tau^0}$
and Theorem~2.8 of \cite{Doss-Wellner:2016ModeConstrained} with $\Delta = \pm
\vvnt{\tau}$.  Perhaps strangely, it seems it is easier to analyze $B_n^1-
B_n^2$ than $B_n^1+B_n^2$, and $C_n^1 + C_n^2$ rather than $C_n^1 - C_n^2$.
Perhaps more strangely, this will suffice.  Again by Theorem~2.2 of
\cite{Doss-Wellner:2016ModeConstrained} with $\Delta = \vvnt{\tau}$ and
Theorem~2.8 of \cite{Doss-Wellner:2016ModeConstrained} with $\Delta =
\vvnat{\tau^0}$, $B_{n,t_n}^1 \le 0$ and $ B_{n,t_n}^2 \ge 0 ,$
so
\begin{equation}
  \label{eq:b_n-t_n1-b_n}
  B_{n,t_n}^1  - B_{n,t_n}^2 \le
  \begin{cases}
    -B_{n,t_n}^2 \le 0 \\
    B_{n,t_n}^1 \le 0.
  \end{cases}
\end{equation}
Thus, if we can show $B_{n,t_n}^1 - B_{n,t_n}^2 = o_p(n^{-1})$ then $B_{n,t_n}^i=o_p(n^{-1})$,
$i=1,2,$ so $B_{n,t_n}^1+B_{n,t_n}^2 = o_p(n^{-1})$. We do this in the
following lemma.

\begin{lemma}
  \label{lem:rem:Bni-op}
  With all terms as defined above,
  \begin{equation*}
    B_{n,t_n}^1 - B_{n,t_n}^2 = o_p(n^{-1}),
  \end{equation*}
  and thus
  \begin{equation*}
    B_{n,t_n}^1+B_{n,t_n}^2 = o_p(n^{-1}).
  \end{equation*}
\end{lemma}
\begin{proof}

  Now
  by \eqref{eq:rem:1}
  \begin{equation*}
    \begin{split}
      B_{n,t_n}^1-B_{n,t_n}^2 & = \int (\vvnt{\tau} - \vvnat{\tau} ) d(\FFn - \FFna)
      + \int (\vvnat{\tau} - \vvnat{\tau^0}) d( \FFn - \FFna)
    \end{split}
  \end{equation*}
  which equals
  \begin{equation}
    \label{eq:b_n1-b_n2-T-type-term}
    \int (\vvnt{\tau} - \vvnat{\tau}) d(\FFn - \FFna)
    + (\vvna(\tau^0)-\vvna(\tau)) (\FFn - \FFna)(\tau)
    + \int_\tau^{\tau^0} \vvnat{\tau^0} d(\FFn - \FFna).
  \end{equation}
  The first term in \eqref{eq:b_n1-b_n2-T-type-term} equals, applying
  \eqref{eq:f-to-varphi-taylor-1},
  \begin{equation}
    \label{eq:int_x_1tau-vvn-vvna2}
    \int_{X_{(1)}}^{\tau} (\vvn - \vvna)^2
    e^{\tilde{\epsilon}_n^1} \ffna d\lambda -
    ((\vvn-\vvna)(\FFn-\FFna))(\tau),
  \end{equation}
  where $\tilde{\epsilon}_n^1$ is identical to $\tilde{\epsilon}_n^1$ in
  Lemma~\ref{lem:rem:step2:Tni-Opn-1}, and thus by
  \eqref{eq:rem:Tni-op-epsilon-tildes} the first term in
  \eqref{eq:int_x_1tau-vvn-vvna2} is $o_p(n^{-1})$.  The second term is also
  $o_p(n^{-1})$ since $\Vert \vvn - \vvna \Vert_{J^*} = o_p(n^{-2/5})$ and
  $\Vert \FFn - \FFna \Vert_{J^*} = o_p(n^{-3/5})$.  This also shows that the
  middle terms in \eqref{eq:b_n1-b_n2-T-type-term} is $o_p(n^{-1})$.  To see
  the last term is $o_p(n^{-1})$, recall $\Vert \vvnat{\tau^0} \Vert_{J^*} =
  O_p(n^{-2/5})$ by Corollary~5.4 of \cite{Doss-Wellner:2016ModeConstrained},
  using that $\vvo'(m)=0$.  Since $|\tau^0 - \tau| = o_p(n^{-1/5})$ and $\Vert
  \ffn - \ffna \Vert_{J^*} = o_p(n^{-2/5})$ we see the last term of
  \eqref{eq:b_n1-b_n2-T-type-term} is $o_p(n^{-1})$, so
  \eqref{eq:b_n1-b_n2-T-type-term} is $o_p(n^{-1})$, so $B^1_{n,t_n} -
  B^2_{n,t_n} = o_p(n^{-1})$.  By \eqref{eq:b_n-t_n1-b_n}, $B^1_{n,t_n} +
  B^2_{n,t_n} = o_p(n^{-1})$, so we are done.
\end{proof}

We now turn our attention to
$  C_{n,t_n}^1 + C_{n,t_n}^2$.
\begin{lemma}
  With all terms as defined above,
  \begin{equation*}
    C^1_{n,t_n} + C^2_{n,t_n} = o_p(n^{-1}).
  \end{equation*}
\end{lemma}
\begin{proof}
  Note that $ C_{n,t_n}^1 + C_{n,t_n}^2$ equals
  \begin{align*}
    & \int (\vvn'(t_n-) - (\vvna)'(t_n-)) (x-t_n)_- \, d(\Fn - \FFna) \\
    &\quad + \int (\vvna)'(t_n-) (x-t_n)_- \, d(\Fn - \FFna)
    + \int (\vvna)'(t_n-) ( x-t_n)_- \, d(\FFn - \Fn)
  \end{align*}
  which equals
  \begin{equation}
    \label{eq:rem:Cn1-plus-Cn2-first-identity}
    \int (\vvn'(t_n-) - (\vvna)'(t_n-)) (x-t_n)_- \, d(\Fn - \FFna)
    + \int (\vvna)'(t_n-) (x-t_n)_- \, d(\FFn - \FFna).
  \end{equation}
  Since $(\FFn-\FFna)(X_{(1)}) = 0$,
  the second term in
  \eqref{eq:rem:Cn1-plus-Cn2-first-identity} equals
  \begin{equation}
    \label{eq:rem:Cn1-plus-Cn2-second-term-1}
    -(\vvna)'(t_n-) \int_{-\infty}^{t_n} (\FFn - \FFna) d\lambda
    =   -(\vvna)'(t_n-) \int_{\upsilon}^{t_n} (\FFn - \FFna) d\lambda
  \end{equation}
  for a point $\upsilon \in [\tau, \tau^0]$
  which exists by the proof of Proposition~2.13 of
  \cite{Doss-Wellner:2016ModeConstrained}.
  By
  \eqref{eq:rem:FF-knot-op-statements},
  since $t_n - \upsilon = O_p(n^{-1/5})$,
  \begin{equation}
    \label{eq:rem:int_-ffn-ffna}
    \int_{\upsilon}^{t_n} (\FFn - \FFna) d\lambda = o_p(n^{-4/5}).
  \end{equation}
  Since $t_n \to m$, by
  Corollary~5.4 of
  \cite{Doss-Wellner:2016ModeConstrained},
  $(\vvna)'(t_n-) = O_p(1) n^{-1/5}.$
  As in previous cases, by taking $t_n = \xi_n$ and $C = 0$ in that
  corollary, the $O_p(1)$ does not depend on $t_n$.
  Thus we have shown
  \eqref{eq:rem:Cn1-plus-Cn2-second-term-1}
  is $o_p(n^{-1})$.

  Since $\Fn(-\infty)-\FFna(-\infty)=0$,
  the first term in \eqref{eq:rem:Cn1-plus-Cn2-first-identity} equals
  \begin{equation*}
    - (\vvn - \vvna)'(t_n-) \int_{X_{(1)}}^{t_n} (\Fn - \FFna) d\lambda,
  \end{equation*}
  which equals
  \begin{equation}
    \label{eq:rem:-vvn-vvnat}
    -(\vvn - \vvna)'(t_n-) \int_{\tau^0}^{t_n} (\Fn - \FFna) d\lambda
  \end{equation}
  because $\widehat{H}_{n,L}^0(\tau^0) = \mathbb{Y}_{n,L}(\tau^0)$ by
  Theorem~2.10 of \cite{Doss-Wellner:2016ModeConstrained}.
  The  absolute value of \eqref{eq:rem:-vvn-vvnat}
  is bounded above by
  \begin{equation}
    \label{eq:rem:Cn1-plus-Cn2-first-term-bound}
    |\vvn'(t_n-) - (\vvna)'(t_n-)| (t_n - \tau^0) \sup_{u \in [\tau^0, t_n]} |\Fn(u) - \FFna(u)|.
  \end{equation}
  We know $ |\vvn'(t_n-) - (\vvna)'(t_n-)| = O_p(n^{-1/5})$ but unfortunately it is
  not necessarily $o_p(n^{-1/5})$.  However,
  \begin{equation*}
    |\vvn'(t_n-) - (\vvna)'(t_n-)| (t_n - \tau^0)
    = | (\vvn-\vvna)(t_n) - (\vvn-\vvna)(\tau^0)|
    = o_p(n^{-2/5}),
  \end{equation*}
  so \eqref{eq:rem:Cn1-plus-Cn2-first-term-bound} is $o_p(n^{-1})$, and thus
  so also is \eqref{eq:rem:Cn1-plus-Cn2-first-identity}; that is,
  $C_{n,t_n}^1 + C_{n,t_n}^2 = o_p(n^{-1})$.
\end{proof}

Thus we have shown that $C^1_{n,t_n} + C^2_{n,t_n} = o_p(n^{-1})$ and
$B^1_{n,t_n} + B^2_{n,t_n} = o_p(n^{-1})$, so we can conclude
\begin{equation*}
  A^1_{n,t_n} +
  A^2_{n,t_n}  = o_p(n^{-1}).
\end{equation*}
Together with
\eqref{eq:rem:E1-+-E2-op}, \eqref{eq:rem:Tni-op-switch-epsilons},
and \eqref{eq:rem:2}, we can conclude that
\begin{equation}
  \label{eq:rem:Rn1c-op}
  R_{n,t_n} + R_{n,1,t_n}^c = o_p(n^{-1}).
\end{equation}

\subsection{Proof completion / details:  the main result}
\label{subsubsec:final-arguments}

The preceding one-sided arguments apply symmetrically to the error terms on
the right side of $m$.  Thus, we now return to handling simultaneously the
two-sided error terms.
We have thus shown for any $\delta > 0$ we can find a $b \equiv b_\delta$,
such that,
letting $t_{n,1} = m - b n^{-1/5}$, $t_{n,2} = m + b n^{-1/5}$, we have
\begin{equation}
  \label{eq:rem:16}
  |R_{n,1,t_{n,1}, t_{n,2}} + R_{n,1,t_{n,1}, t_{n,2}}^c | \le  \delta K n^{-1}
\end{equation}
where $K  = O_p(1)$ does not depend on $b$ (i.e., on $\delta$).
Now by
Proposition~\ref{prop:localerrorterms}
\begin{equation*}
  |R_{n,2,t_{n,1},t_{n,2}}| + |R_{n,2,t_{n,1},t_{n,2}}^0|
  + |R_{n,3,t_{n,1},t_{n,2}}| + |R_{n,3,t_{n,1},t_{n,2}}^0| = o_p(n^{-1}).
\end{equation*}
Let
\begin{align*}
  R_{n,t_{n,1},t_{n,2}} & \equiv
  2 n(R_{n,1,t_{n,1},t_{n,2}} + R_{n,1,t_{n,1},t_{n,2}}^c +
  R_{n,2,t_{n,1},t_{n,2}} \\
  & \qquad - R_{n,2,t_{n,1},t_{n,2}}^0 + R_{n,3,t_{n,1},t_{n,2}} -R_{n,3,t_{n,1},t_{n,2}}^0).
\end{align*}
Then by \eqref{BasicDecompositionFirstForm}, write $2 \log \lambda_n =
\mathbb{D}_{n,\twoArgs[t_{n,1},t_{n,2}]} + R_{n,t_{n,1},t_{n,2}}$ (slightly
modifying the form of the subscripts).  Now fixing any subsequence of $\lb n
\rb_{n=1}^\infty$, we can find a subsubsequence such that
$R_{n,t_{n,1},t_{n,2}} \to_d \delta R$ for a tight random variable $R$ by
\eqref{eq:rem:16}.
For $b > 0$ let $ \mathbb{D}_b \equiv \int_{-b/ \gamma_2}^{b/\gamma_2}
(\widehat{\varphi}^2(u) -(\widehat{\varphi}^0)^2(u))du$, as in
\eqref{eq:DDb-formula-gammas}, %
which lets us conclude that
\begin{equation}
  \label{eq:rem:17}
  2 \log \lambda_n = \mathbb{D}_{n, t_{n,1},t_{n,2}} + R_{n,t_{n,1},t_{n,2}}
  \to_d \mathbb{D}_{b} + \delta R
\end{equation}
along the subsubsequence. Taking say $\delta = 1$ shows that there exists a
(tight) limit random variable, which we denote by $\DD$.  Then, since $R$ does
not depend on $\delta$, we can let $\delta \searrow 0$ so $b_\delta \equiv b
\nearrow \infty$, and see that $\lim_{b \to \infty} \DD_b = \DD$, which can
now be seen to be pivotal.  Thus along this subsubsequence, $ 2 \log
\lambda_n \to_d \DD.$ This was true for an arbitrary subsubsequence, and so
the convergence in distribution holds along the original sequence. Thus,
\begin{equation*}
  2 \log \lambda_n \to_d \mathbb{D}
  \quad \mbox{ as } \quad
  n \to \infty.
\end{equation*}

\subsection{Proofs for global consistency}
\label{ssec:ConsistencyProofs}

\begin{proof}[Proof of Theorem~\ref{thm:GlobalConsistencyWithRates} Part B]
  We now indicate the changes to the arguments of \cite{DR2009LC} which are
  needed to prove an analog of Theorem~4.1 of \cite{DR2009LC} for
  $\widehat{\varphi}_n^0$.
  Note that our Theorem~\ref{thm:GlobalConsistencyWithRates} part B is only a partial analogue of Theorem~\ref{thm:GlobalConsistencyWithRates} part A since we only consider the case $\beta =2$ and require $m$ to be unique.  We assume $f_0 \in {\cal P}_{m} := \{ e^{\varphi} : \int e^{\varphi(x)}dx= 1$, $\varphi \in {\cal C}_{m} \}$ where for $m$ fixed ${\cal C}_{m}$ is the class of concave, closed, proper functions with $m$ as a maximum.  We need to study the allowed `caricatures' of the Lemmas A.4 and A.5 of \cite{DR2009LC}, which differ for $\widehat{\varphi}_n^0$ from those for $\widehat{\varphi}_n$.  Let $\rho_n \equiv n^{-1} \log n$. %
Note, we define here a function $\Delta$ to be ``piecewise linear (with $q$ knots)'' to mean that
  \begin{eqnarray}
    \label{eq:4}
    && \ \ \RR \mbox{ may be partitioned into $q+1$ non-degenerate intervals }\\
    && \ \ \mbox{on each of which } \Delta  \mbox{ is linear.} \nonumber
  \end{eqnarray}
  In particular, $\Delta$ may be discontinuous.  Let $\mathcal{D}_k$ be the
  family of piecewise linear functions on $\RR$ with at most $k$ knots.  Let
  \begin{align}
    \label{eq:defn:MI}
    & \widehat{M} := \overline{\{ x \in \RR  : (\widehat{\varphi}_n^0)'(x) = 0 \}}
    & \mbox{ and }  &
    & \widehat{N} := [\tau_L, \tau_R],
  \end{align}
  where $\tau_L$ is the greatest knot of $\widehat{\varphi}_n^0$ strictly smaller than $m$
  and $\tau_R$ is the smallest knot of $\widehat{\varphi}_n^0$ strictly larger than $m$.
  Note that $\widehat{M}$, the (closed) modal interval of $\widehat{\varphi}_n^0$, is contained in
  $\widehat{N}$, and may or may not be strictly contained in $\widehat{N}$.
  Let ${\cal S}_n(\widehat{\varphi}_n^0)$
  denote the set of knots of $\widehat{\varphi}_n^0$.

  \begin{lemma}
    \label{lem:MC-caricatures}
    Let $\widehat{M}$ be as in \eqref{eq:defn:MI}.  Let $\Delta: \RR \to \RR$ be
    piecewise linear in the sense of \eqref{eq:4}, such that
    \begin{equation}
      \label{eq:1}
      \Delta 1_{\widehat{M}}  +
      (-\infty) \times 1_{\widehat{M}^c} \mbox{ is concave with mode at } m,
    \end{equation}
    and assume for each knot $q$ of $\Delta$ that one of the following holds:
    \begin{align}
      &   q \in {\cal S}_n(\widehat{\varphi}_n^0) \setminus \{ m \}
      &\mbox{ and } &
      & \Delta(q) = \liminf_{x \to q} \Delta(x) \label{eq:knot-caricature} \\
      & \Delta(q) = \lim_{r \to q} \Delta(r)
      &\mbox{ and }&
      & \Delta'(q-) \ge \Delta'(q+) \label{eq:concave-caricature} \\
      & q=m, \Delta(q) = \lim_{r \in \widehat{M}, r \to q} \Delta(r)
      & \mbox{ and } &
      & \Delta(q) = \liminf_{x \to q} \Delta(x). \label{eq:mode-caricature}
    \end{align}
    Then
    \begin{equation}
      \label{eq:3}
      \int \Delta d \FF_n \le \int \Delta d \widehat{F}_n^0.
    \end{equation}
  \end{lemma}

  Note that if $m$ is not a knot of $\widehat{\varphi}_n^0$, so is interior to $\widehat{M}$, then
  \eqref{eq:mode-caricature} implies that $\Delta$ must be continuous at $m$
  and \eqref{eq:1} implies that $m$ is a local mode of $\Delta$.
  If $m$ is, e.g.,  a right knot of $\widehat{\varphi}_n^0$ so $(\widehat{\varphi}_n^0)'(m+) < 0$, then
  \eqref{eq:mode-caricature} allows $\Delta$ to be discontinuous at $m$ but
  forces $\Delta(m-) \le \Delta(m+)$.

  \begin{proof}
    We show we can construct a sequence $\Delta_k$
    converging pointwise to $\Delta$, with $|\Delta_k| \le |\Delta|$, and such
    that $\widehat{\varphi}_n^0 + \epsilon \Delta_k$ is concave with mode at $m$ for small
    enough $\epsilon$.  We first show this holds on the interval $\widehat{N}$.

    If $m $ is not a knot of $\widehat{\varphi}_n^0$ then by \eqref{eq:1} $\widehat{\varphi}_n^0 + \epsilon
    \Delta$ is concave with mode at $m$ on $\widehat{N}$.  Now if $m$ is a knot of
    $\widehat{\varphi}_n^0$, either \eqref{eq:concave-caricature} holds or
    \eqref{eq:mode-caricature} holds.  In the former case, again for $\epsilon>
    0$ small enough, $\widehat{\varphi}_n^0 + \epsilon \Delta$ is concave with mode at $m$ on
    $\widehat{N}$.

    Thus assume \eqref{eq:mode-caricature} holds. For concreteness, assume
    $m$ is a left knot of $\widehat{\varphi}_n^0$, so $(\widehat{\varphi}_n^0)'(m-) < 0 = (\widehat{\varphi}_n^0)'(m+)$.
    For $x \in [m - 1/k, m]$ define $\Delta_k(x)$ to be the linear function
    connecting $\Delta(m)$ to $\Delta(m-1/k)$ for $k = 1, \ldots,$ and let
    $\Delta_k(x) =\Delta(x) $ for $x \in \widehat{N} \setminus [m-1/k, m]$.  Then for
    $k$ large,
    \begin{equation}
      \label{eq:2}
      \widehat{\varphi}_n^0 + \epsilon \Delta_k
      \mbox{    is concave with mode } m
    \end{equation}
    on $\widehat{N}$,
    and $\Delta_k(x)$ is monotonically increasing to $\Delta(x)$ (again, for $k
    \ge$ some $K$), by \eqref{eq:mode-caricature}.

    For knots $q$ of $\Delta$ with $q \ne m$, similar arguments can be made;
    one can define $\Delta_k(x)$ such that $|\Delta_k(x)| \le |\Delta(x)|$
    where the knots $q_k$ of $\Delta_k$ are either knots of $\widehat{\varphi}_n^0$ or satisfy
    $\Delta_k'(q_k-) > \Delta_k'(q+)$ so that for $\epsilon >0 $ small
    \eqref{eq:2} holds globally.  Thus, by the dominated convergence theorem,
    and the characterization theorem for $\widehat{F}_n^0$,
    \begin{equation*}
      \int \Delta d \FF_n
      = \lim_{k \to \infty} \int \Delta_k d \FF_n
      \le \lim_k \int \Delta_k d\widehat{F}_n^0
      = \int \Delta d \widehat{F}_{n}^0.
    \end{equation*}
  \end{proof}

  For the next lemma, we define for a function $\Delta : \RR \to \RR$,
  \begin{align*}
    &W(\Delta) = \sup_{x \in \RR} \frac{ | \Delta(x)|}{1 \vee |\varphi_0(x)|}
    & \mbox{ and } &
    & \sigma^2(\Delta) =  \int_{\RR} \Delta^2(x) dF_0(x).
  \end{align*}
  Also, for a point $x \in \RR$, let $\tau^0_+(x) = \min
  S_n(\widehat{\varphi}_n^0) \cap [x, \infty)$ and $\tau^0_-(x) = \max
  S_n(\widehat{\varphi}_n^0) \cap (-\infty,x ]$.

  \begin{lemma}
    \label{lem:UC-caricature-properties}
    Let  $T = [A,B]$ be a compact subinterval strictly contained in $\{ f_0 > 0
    \}$.  Let $\varphi_0 - \widehat{\varphi}_n^0 \ge \epsilon$ or
    $\widehat{\varphi}_n^0 - \varphi_0 \ge \epsilon$ on some interval
    $[c,c+\delta] \subset T$ with length $ \delta > 0$ and suppose $X_{(1)} < c$
    and $X_{(n)} > c+\delta$.  Suppose $[\tau^0_-(c), \tau^0_+(c+\delta)] \cap
    \hat{N} = \emptyset$. Then there exists a piecewise linear function
    $\Delta$ with at most three knots, each of which satisfies one of
    conditions
    \eqref{eq:knot-caricature} or
    \eqref{eq:concave-caricature}
    and a positive constant $K' = K'(f_0, T)$ such that
    \begin{eqnarray}
      && | \varphi_0 - \widehat{\varphi}^0 | \ge \epsilon | \Delta | ,  \label{A6}\\
      && \Delta (\varphi_0 - \widehat{\varphi}^0 ) \ge 0, \label{A7}\\
      && \Delta \le 1, \label{A8}\\
      && \int_c^{c+\delta} \Delta^2 (x) dx \ge \delta /3,  \label{A9}\\
      && W(\Delta) \le K' \delta^{-1/2} \sigma(\Delta) \label{A10}.
    \end{eqnarray}
  \end{lemma}
  \begin{proof}
    The proof is identical to the proof of Lemma A.5 in
    \cite{DR2009LC}; %
    the condition $[\tau^0_-(c), \tau^0_+(c+\delta)] \cap \hat{N} = \emptyset$
    allows us to use identical perturbations for $\widehat{\varphi}_n^0$ that
    one can use for $\widehat{\varphi}_n$.
  \end{proof}
  Next, we need an adaptation of the above lemma for the more difficult case where we have to
  accomodate the modal constraint.  We assume here that the length of
  $\widehat{N}$ is shorter than $\delta$, which will be true with high
  probability when we apply the lemma to the case where $\widehat{N}$ is of
  order $n^{-1/5}$ and $\delta$ of order $(\log n / n)^{1/5}$.
  \begin{lemma}
    \label{lem:MC-caricature-properties}
    Let $T = [A,B]$ be a compact interval strictly contained in $\{ f_0 > 0
    \}$.  Let $\varphi_0 - \widehat{\varphi}_n^0 \ge \epsilon$ or $\widehat{\varphi}_n^0 - \varphi_0 \ge \epsilon$ on
    some interval $[c, c+ \delta] \subset T$ with length $\delta > 0$ and
    suppose that $X_{(1)} < c$ and $ X_{(n)} > c + \delta$.
    Suppose also  that
    $[\tau^0_-(c), \tau^0_+(c+\delta)] \cap \hat N \ne \emptyset$ and
    $|
    \widehat{N} |$, the length of $\widehat{N}$, is no larger than $\delta / 4$, and suppose $T
    \setminus [c, \infty)$ and $T \setminus (-\infty,c+\delta]$ both contain a
    knot of $\widehat{\varphi}_n^0$.  Then there
    exists a piecewise linear (in the sense of \eqref{eq:4}) function $\Delta$
    with at most $4$ %
    knots, satisfying the conditions of Lemma~\ref{lem:MC-caricatures}, and
    there exists a positive $K' \equiv K'(f_0, T)$ such that
    \begin{align}
      & | \varphi_0 - \widehat{\varphi}_n^0| \ge \epsilon |\Delta|, \label{eq:delta-upper-bound} \\
      & \Delta ( \varphi_0 - \widehat{\varphi}_n^0) \ge 0, \label{eq:delta-aligned} \\
      & \Delta \le 1 \label{eq:delta-1-upper-bound} \\
      & \int_c^{c+\delta} \Delta^2(x)dx \ge \delta / 6,
      \label{eq:delta-size-bound} \\
      & W(\Delta) \le K' \delta^{-1/2} \sigma(\Delta). \label{eq:W-sigma-bound}
    \end{align}

  \end{lemma}

  \begin{proof}
    We argue by several different cases.  We focus only on the cases where
    $\widehat{N}$ is near to $[c,c+\delta]$ in the sense that we now assume that either $\widehat{N} \cap
    [c,c+\delta] \ne \emptyset$ or there are no knots of $\widehat{\varphi}_n^0$ between $\widehat{N}$
    and $[c,c+\delta]$.  In any other case, the proof Lemma~A.5 of
    \cite{DR2009LC} applies without modification.

    We begin with the cases where $\widehat{\varphi}_n^0
    - \varphi_0 \ge \epsilon$ on $[c,c+\delta]$.  There are separate subcases
    depending on how $\widehat{N}$ relates to $[c,c+\delta]$ and the (non-)existence of
    other knots in $[c,c+\delta].$ In all cases, we will first verify conditions
    \eqref{eq:delta-upper-bound}--\eqref{eq:delta-size-bound} and put off
    verifying \eqref{eq:W-sigma-bound} until later.

    {\bf Case 1.1} Assume $\widehat{\varphi}_n^0 - \varphi_0 \ge \epsilon$ on $[c,c+\delta]$ and
    $\widehat{N} \subset [c,c+\delta]$.  Let $\Delta \in \mathcal{D}_4$ be continuous
    (and piecewise linear)
    with knots at $c, \tau_L, \tau_R, $ and $c+ \delta$, and let $\Delta $ be
    equal to $-1$ on $\widehat{N}$ and $0$ on $[c,c+\delta]^c$.  Thus $\Delta $
    satisfies conditions \eqref{eq:1} and \eqref{eq:concave-caricature} of
    Lemma~\ref{lem:MC-caricatures}.  Then $|\Delta| \le 1$ on $[c,c+\delta]$
    and is $0$ on $[c,c+\delta]^c$, so \eqref{eq:delta-upper-bound} is
    satisfied, and so is \eqref{eq:delta-aligned}. Since $\Delta$ is always
    nonpositive, \eqref{eq:delta-1-upper-bound} is trivially satisfied.
    We see
    that
    \begin{eqnarray}
      \label{eq:delta-1.1-lowerbound}
      \int_c^{c+\delta} \Delta^2 (x) dx
      &  \ge  & \int_0^{x_0} \left ( \frac{x}{x_0} \right )^2 dx + \int_{x_0}^\delta \left ( \frac{\delta-x}{\delta-x_0} \right )^2 dx \\
      & = &  \frac{x_0}{3} + \frac{\delta - x_0}{3} = \frac{\delta}{3},  \nonumber
    \end{eqnarray}
    so \eqref{eq:delta-size-bound} is
    satisfied.

    The next two cases assume $\widehat{\varphi}_n^0 - \varphi_0 \ge \epsilon$ on $[c,c+\delta]$ and
    $\widehat{N} \not\subset [c, c+ \delta]$. Recall $\widehat{N} = [\tau_L, \tau_R]$.

    {\bf Case 1.2} Assume $\widehat{\varphi}_n^0 - \varphi_0 \ge \epsilon$ on $[c,c+\delta]$ and
    $\widehat{N} \not\subset [c, c+ \delta]$.  Additionally, assume there exists $ \tau
    \in S_n(\widehat{\varphi}_n^0) \cap ( [c,c+\delta] \setminus \widehat{N}) $.  We now again let
    $\Delta \in \mathcal{D}_3 \subset \mathcal{D}_4$ be continuous, now with
    $\Delta(\tau) = -1$. If $\tau_R < c+ \delta$, then set the knots of
    $\Delta$ at $c \vee \tau_R, \tau, $ and $c+ \delta$, and set $\Delta$ to be
    $0$ on $[c \vee \tau_R, c+\delta]^c$.  If $\tau_L > c$, then set the knots
    at $c, \tau, $ and $(c+\delta) \wedge \tau_L$ and set $\Delta$ to be $0$ on
    $[c, (c+\delta) \wedge \tau_L]^c$.  Consider the case where $\tau_R < c+
    \delta$, and the other case is identical.  Since $\widehat{N} \not\subset
    [c,c+\delta]$ and $| \widehat{N} | \le \delta / 4$, $\tau_R - c < \delta /4$.
    Again, $\Delta$ satisfies conditions \eqref{eq:1} and
    \eqref{eq:concave-caricature} of Lemma~\ref{lem:MC-caricatures}.
    Conditions \eqref{eq:delta-upper-bound}--\eqref{eq:delta-1-upper-bound} can
    be immediately verified, as before.  Condition~\eqref{eq:delta-size-bound}
    can be verified as in the previous case, replacing $\delta$ by $3\delta/4$,
    since $\tau_R - c < \delta /4$, and this yields
    \begin{equation}
      \label{eq:5}
      \int_c^{c+\delta}
      \Delta^2(x)dx \ge \delta / 4.
    \end{equation}

    {\bf Case 1.3} Assume $\widehat{\varphi}_n^0 - \varphi_0 \ge \epsilon$ on $[c,c+\delta]$ and
    $\widehat{N} \not\subset [c,c+\delta]$.  Additionally, assume that $S_n(\widehat{\varphi}_n^0) \cap
    ([c,c+\delta] \setminus \widehat{N}) = \emptyset.$ We define $\tilde \Delta$ to be
    an affine function either with $\tilde \Delta(c) = -\epsilon$ and $\tilde
    \Delta$ nonincreasing or $\tilde \Delta(c+\delta) = - \epsilon$ and $\tilde
    \Delta $ nondecreasing. Thus, $\tilde \Delta \le -\epsilon$ on
    $[c,c+\delta]$. We take $\tilde \Delta$ to be tangent to $\varphi_0 - \widehat{\varphi}_n^0$
    (but this is not essential).  Next let $(c_1, d_1) := \{ \tilde \Delta < 0
    \} \cap (c_0,d_0)$ where $[c_0,d_0] \supset ([c,c+\delta] \setminus \widehat{N})$
    is defined to be the maximal interval on which $\varphi_0 - \widehat{\varphi}_n^0$ is concave,
    so $\widehat{\varphi}_n^0$ is linear.  Define $\Delta \in \mathcal{D}_2 \subset
    \mathcal{D}_4$ via
    \begin{align*}
      \Delta (x) :=
      \begin{cases}
        0 & x \notin [c_1,d_1], \\
        \tilde \Delta(x) / \epsilon  & x \in [c_1,d_1].
      \end{cases}
    \end{align*}
    Now, \eqref{eq:1} of Lemma~\ref{lem:MC-caricatures} is seen to be satisfied
    since $\Delta$ is $0$ on $\widehat{N}$, and since by definition $\tau_L \ne m \ne
    \tau_R$, \eqref{eq:knot-caricature} is satisfied at $c_1$ and $d_1$.  Since
    $\tilde \Delta$ is tangent to $\varphi_0 - \widehat{\varphi}_n^0$, condition
    \eqref{eq:delta-upper-bound} is verified, and \eqref{eq:delta-aligned} and
    \eqref{eq:delta-1-upper-bound} are also seen to be verified.  Condition
    \eqref{eq:delta-size-bound} holds easily since in fact $\Delta \le -1$ on
    $[c,c+\delta]$.

    {\bf Case 2} Now assume $\varphi_0 - \widehat{\varphi}_n^0 \ge \epsilon$ on $[c,c+\delta]$.\\
    {\bf Case 2.1} Assume $\widehat{N} \subset (c, c+ \delta)$. Then if $c + \delta/2
    \le m \le c + \delta$ then set $c_0 = \tau_{-1}(c)$, the largest knot of
    $\widehat{\varphi}_n^0$ not larger than $c$, set $x_0 = m$, and set $d_0 = \tau_R$.  If $
    c \le m < c + \delta / 2$, set $c_0 = \tau_L$, set $x_0 = m$, and let
    $d_0 = \tau_1(c+\delta)$, the smallest knot of $\widehat{\varphi}_n^0$ not smaller than $c+
    \delta$.  \\
    {\bf Case 2.2} Assume $\widehat{N} \not\subset (c,c+\delta)$. Then $(c,c+\delta)
    \setminus \widehat{N}$ is an interval, and we set $x_0$ %
    to be the midpoint of this interval; if $m \ge c + \delta / 2$ then set
    $c_0 = \tau_{-1}(c)$ and and set $d_0 = \tau_1(c+\delta) \wedge \tau_L$.
    Similarly, if $m < c + \delta / 2$, set $c_0 = \tau_{-1}(c) \vee \tau_R$
    and set $d_0 = \tau_1(c+\delta)$.  Since $|\widehat{N}| \le \delta / 4$, $c_0 - c
    \le \delta /4 $ and $c+ \delta - d_0 < \delta / 4$ (where only one of the
    previous inequalities is relevant, depending on whether $m < c + \delta /
    4$ or $m > c + 3 \delta/ 4$).

    For both Case 2.1 and 2.2 we then define $\Delta \in \mathcal{D}_3 \subset
    \mathcal{D}_4$ by
    \begin{align}
      \Delta(x) :=
      \begin{cases}
        0, & x \in [c_0,d_0]^c,\\
        1 + \beta_1(x- x_0 ) & x \in [c_0, x_0],\\
        1 + \beta_2(x- x_0 ), & x \in [x_0, d_0],
      \end{cases}
    \end{align}
    where $\beta_1 \ge 0$ is chosen such that if
    \begin{align}
      & (\varphi_0 - \widehat{\varphi}_n^0)(c_0) \ge 0 & \mbox{ then } &   &  \Delta(c_0) = 0, \mbox
      { and if } \label{eq:case-1L} \\
      & (\varphi_0 -\widehat{\varphi}_n^0)(c_0) < 0 & \mbox{ then }& & \mbox{sign}(\Delta) = \mbox{sign}(\varphi_0-\widehat{\varphi}_n^0)
      \mbox{ on } [c_0,x_0], \label{eq:case-2L}
    \end{align}
    where $\mbox{sign}( y )$ is $1$ if $ y \ge 0$ and $-1$ if $y < 0$.  Similarly,
    $\beta_2 \le 0$ is chosen such that if
    \begin{align}
      &  (\varphi_0 - \widehat{\varphi}_n^0)(d_0) \ge 0   & \mbox{ then } & & \Delta(d_0) = 0, \mbox
      { and if } \label{eq:case-1R} \\
      & (\varphi_0 -\widehat{\varphi}_n^0)(d_0) < 0 & \mbox{ then }& & \mbox{sign}(\Delta) = \mbox{sign}(\varphi_0-\widehat{\varphi}_n^0)
      \mbox{ on } [x_0, d_0]. \label{eq:case-2R}
    \end{align}
    That is, $\Delta$ is defined to be $1$ at $x_0$ and, if $\varphi_0 - \widehat{\varphi}_n^0$
    crosses below $0$ on $[c_0, c] \cup [c+\delta, d_0]$ at potential points
    $\tilde c$ or $\tilde d$, then
    $\Delta$ crosses below $0$ at the same point(s).  We note also for future
    reference in Case 2.1 that if $c + \delta /2 \le m$ then $(\varphi_0-\widehat{\varphi}_n^0)(d_0)
    = 0$ since $d_0=\tau_R \le c + \delta$, so we are in case
    \eqref{eq:case-1R} and $\Delta(d_0)=0$.  Thus
    $W(\Delta) = W( \Delta 1_{[c_0,x_0]})$, because we have thus forced
    $W(\Delta 1_{[x_0,d_0]}) = 1$
    (and $1 \le W( \Delta 1_{[c_0,x_0]})$).  Similarly, if $m < c+ \delta/2$ then we are in
    case \eqref{eq:case-1L}, and
    $W(\Delta) = W(\Delta 1_{[x_0,d_0]})$. Now we check that the conditions of
    Lemma~\ref{lem:MC-caricatures}
    hold.\\
    {\bf Case 2.1 (continued)} If $m \in \widehat{N} \subset [c, c+ \delta]$ then
    $\Delta$ is continuous at $m$ (so \eqref{eq:mode-caricature} holds),
    \eqref{eq:1} holds, and at $c_0$ and $d_0$ \eqref{eq:knot-caricature} holds
    (possibly with one discontinuity) since these are both knots.   \\
    {\bf Case 2.2 (continued)} Note, if $\widehat{N} \not\subset [c,c+\delta]$, then
    $\Delta$ is $0$ on $\widehat{N} \supseteq \widehat{M}$: if $\widehat{N} \cap [c,c+\delta] =
    \emptyset$ then this is immediate (since the endpoint of $\widehat{N}$ is the
    nearest knot to $[c,c+\delta]$), and if one of $\tau_L$ or $\tau_R$ lies in
    $(c,c+\delta)$, then $\varphi_0 - \widehat{\varphi}_n^0$ is greater or equal to $\epsilon$ at
    that point, so $\Delta$ will be $0$ at that point.  Now
    \eqref{eq:knot-caricature} holds at $c_0, d_0$, and $\Delta$ is $0$ on $\widehat{N}
    \supseteq \widehat{M}$ so \eqref{eq:mode-caricature} holds.

    Now we check the remaining conditions.  Conditions \eqref{eq:delta-aligned}
    and \eqref{eq:delta-1-upper-bound} hold by construction for both Case 2.1
    and 2.2.  We check Condition~\eqref{eq:delta-size-bound} holds for the two
    cases.  \\
    {\bf Case 2.1 (continued)} Define $\Delta_*(x)$ to be the triangle
    function with $\Delta_*(x_0)=1$ and $\Delta_*(c) = \Delta_*(c+\delta) = 0$.
    We assume without loss of generality that $ m \ge c + \delta / 2$.  Then,
    $\Delta 1_{[c, x_0]} \ge \Delta_* 1_{[c, x_0]}$, so by
    \eqref{eq:delta-1.1-lowerbound},
    $\int \Delta^2(x)dx \ge \int_c^{m} \Delta_*^2(x)dx \ge (m-c)/3 \ge \delta / 6$. \\
    {\bf Case 2.2 (continued)} Define $\Delta_*$ to be the triangle function
    with $\Delta_*(x_0)=1$,
    $\Delta_*(c_0 \vee c) = 0 = \Delta_*(d_0 \wedge (c+\delta))$,
    and $\Delta_*(x)=0$ for $x \notin [c,c+\delta]$.  Then again
    $\Delta 1_{[c_0 \vee c, d_0 \wedge( c+\delta)]} \ge \Delta_*1_{[c_0 \vee c, d_0 \wedge( c+\delta)]}$.
    Since $c_0 - c \le \delta / 4$ and $ c+ \delta - d_0 < \delta / 4$,
    $d_0 \wedge( c+\delta) - ( c_0 \vee c) \ge 3 \delta / 4$.  Thus, as in \eqref{eq:5},
    $\int_c^{c+\delta} \Delta(x)^2 dx \ge \int_c^{c+\delta} \Delta_*(x)^2 dx \ge \delta / 4$.

    Next we check \eqref{eq:delta-upper-bound} for both Case 2.1 and 2.2.  If
    $\varphi_0 - \widehat{\varphi}_n^0 \ge 0 $ on $[c_0,d_0]$ there is nothing to check (since then
    $|\Delta| = \Delta \le 1$).  Assume that there is thus a point $\tilde d$
    with $c + \delta < \tilde d < d_0$ such that $\varphi_0 - \widehat{\varphi}_n^0 \le 0$ on
    $[\tilde d, d_0]$.  (An analogous argument holds for a point $\tilde c <
    c$).  By construction, $(\varphi_0 - \widehat{\varphi}_n^0)(c + \delta ) \ge \epsilon \ge
    \epsilon \Delta(c + \delta) \ge 0$, and $(\varphi_0-\widehat{\varphi}_n^0)(\tilde d) = \epsilon
    \Delta (\tilde d) = 0$; on $[c+\delta, d_0]$, $\varphi_0-\widehat{\varphi}_n^0$ is concave by
    the definition of $d_0$.  Thus, $|(\varphi_0 - \widehat{\varphi}_n^0)'(x+)| \ge \epsilon
    |\beta_2|$ for any $x \in [\tilde d, d_0]$.  Thus $(\varphi_0 - \widehat{\varphi}_n^0)(x) \le
    \epsilon \Delta(x) \le 0$ for $x \in [\tilde d, d_0]$.  Thus we have shown
    \eqref{eq:delta-upper-bound}.

    \medskip

    Lastly, we check \eqref{eq:W-sigma-bound} in all cases.  Note that since
    $T$ is a compact interval strictly contained in $\{ f > 0 \}$, there
    exists a constant $C_0$ such that $f(x) \ge C_0$ for $x \in T$.  Now, in
    Case 1.1, $W(\Delta) \le \Vert \Delta \Vert = 1$ where $\Vert \Delta \Vert
    = \sup_{x \in \RR} |\Delta(x)|$.  And we have $\sigma(\Delta)^2 \ge C_0
    \int_{\RR} \Delta(x)^2 dx \ge C_0 \delta / 3$ by
    \eqref{eq:delta-1.1-lowerbound}.  So let $K' \ge (3 / C_0)^{1/2}$, and then
    \eqref{eq:W-sigma-bound} holds.

    Similarly, in Case 1.2, $W( \Delta) \le 1$ and $\sigma(\Delta)^2 \ge C_0
    \delta / 4$ by \eqref{eq:5}, so let $K' \ge (4/C_0)^{1/2}$ and then
    \eqref{eq:W-sigma-bound} holds.

    To handle the remaining cases, we consider $h(x)$ defined by
    $h(x) = 1_{Q}(x) (\alpha + \gamma x)$ for $\alpha, \gamma \in \RR$ where
    $Q = [x_0,y_0]$ is a nondegenerate interval, $Q \subseteq T$.  We always
    have
    \begin{align*}
      W( h) \le \Vert h \Vert
      & & \mbox{ and } & &
      \sigma(h)^2 \ge C_0 \int_{x_0}^{y_0} h(x)^2 dx.
    \end{align*}
    Now $\int_{\RR} h(x)^2 dx$ is invariant under translations of $h$, sign
    changes of $h$, and replacing $h$ by $h(-x)$.  Thus, we assume that
    $h(y_0) > 0$, by replacing $h$ by $-h$ if necessary.  If $\min_{y \in Q}
    h(y) \ge 0$ then let $[x_0,y_0] = [0,y_0]$, taking $x_0 = 0$ by
    translation.  Otherwise, take $x_0 < 0 < y_0$ and $h(0)=0$, by translation.
    Furthermore, we assume $h(y_0) = \Vert h \Vert$ by replacing $h(x)$ by
    $-h(-x)$ if $h(x_0) < -h(y_0) < 0$ (so $h(y_0) > 0$, still),
    or by $h(-x)$ if $h(x_0) > h(y_0) > 0$.   Note that we have forced $h$ to
    be nondecreasing so $\gamma \ge 0$.

    Now if we are in the case $\inf_{y \in Q} h(y) = h(0) = \alpha > 0$ with
    $x_0 = 0$, then
    \begin{eqnarray*}
      \int_0^{y_0} (\alpha + \gamma x)^2 dx
      & = &\frac{1}{ 3 \gamma} \left ( (\alpha + \gamma y_0)^3 - \alpha^3 \right ) \\
      & = & \frac{1}{3} y_0 \left ( (\alpha + \gamma y_0)^2 + \alpha ( \alpha + \gamma y_0) + \alpha^2 \right )
      = \frac{1}{3} (y_0-x_0) \Vert h \Vert^2,
    \end{eqnarray*}
    since $\Vert h \Vert = \alpha + \gamma y_0$ in this case.
    If we are in the case $h(x_0) < 0 < h(y_0)$ with $x_0 < 0 < y_0$, then
    \begin{align*}
      \int_{x_0}^{y_0} (\gamma x)^2 dx
      = \frac{\gamma^2}{3} y_0^3 - \frac{\gamma^2}{3} x_0^3
      = \frac{1}{3}(y_0-x_0) \gamma^2 ( y_0^2 - y_0x_0 + x_0^2)
      \ge \frac{1}{3}  (y_0-x_0) \Vert h \Vert^2
    \end{align*}
    since $\Vert h \Vert^2 = \gamma^2 y_0^2$ in this case. Thus, by
    \eqref{eq:6},
    \begin{equation}
      \label{eq:8}
      \left ( \frac{3}{C_0(y_0-x_0)} \right)^{1/2}     \sigma(h)  \ge  \Vert h \Vert.
    \end{equation}

    Now we apply these computations to the remaining cases.  In Case 1.3,
    $\Delta$ is of the form of $h$ defined above and the corresponding
    $x_0,y_0$ satisfy $y_0-x_0 \ge 3 \delta / 4$ since $[c_1,d_1 ] \supseteq
    [c,c+\delta] \setminus \widehat{N}$ and $| \widehat{N}| \le \delta /4 $.  Note that we can
    take $Q \subset T$ by the assumption that there are knots of $\widehat{\varphi}_n^0$ above
    and below $[c,c+\delta]$, and these bound the support of $\Delta$.  Thus
    this case is complete since $\sigma(\Delta)^2 \ge (C_0/3) (y_0-x_0) W(
    \Delta)^2$.

    For Case 2, $\Delta$ equals $h_1 + h_2$ where $h_1,h_2$ are of the type
    considered above and have disjoint support, where both supports are
    contained in $T$ again by the assumption that $\widehat{\varphi}_n^0$ has knots above and
    below $[c,c+\delta]$.  \\
    {\bf Case 2.1 (continued)} Assume without loss of generality that
    $m \ge c + \delta / 2$.  Then, as noted after display \eqref{eq:case-2R},
    $W(\Delta) = W(\Delta 1_{[c_0,x_0]}) \equiv W(h_1)$.  For $h_1$, the
    corresponding $x_0,y_0$ satisfy $y_0-x_0 \ge \delta / 2$. Thus,
    \begin{align*}
      W(\Delta) = W(h_1) \le \frac{6^{1/2}}{C_0^{1/2} \delta^{1/2}} \sigma(h_1)
      \le \frac{6^{1/2}}{C_0^{1/2} \delta^{1/2}}  \sigma(\Delta).
    \end{align*}
    {\bf Case 2.2 (continued)} In this case, for both $h_1, h_2$, the
    corresponding $x_0,y_0$ satisfy $y_0-x_0 \ge 3 \delta / 8$ (again using
    $|\widehat{N}| \le \delta /4 $).  Thus,
    \begin{align*}
      W(\Delta) = \max(W(h_1),W(h_2))
      \le \frac{ 8^{1/2}}{C_0^{1/2}
        \delta^{1/2}}  \max(\sigma(h_1),\sigma(h_2))
      \le \frac{ 8^{1/2}}{C_0^{1/2}
        \delta^{1/2}}  \sigma(\Delta).
    \end{align*}
    This completes the proof.
  \end{proof}

  Now we complete the proof of Theorem 4.7.B.  We treat the case $m \in K$.  We
  can always enlarge $K$ so this holds.  Now, since
  $\vp_0''(m)<0$,  there is an interval $K^0$ containing $m$ such that
  $\widehat{\varphi}_n^0$ has knots above and below $K^0$ with high
  probability for large $n$ since $\widehat{\varphi}_n^0$ is uniformly
  consistent
  by Proposition~7.2 of \cite{Doss-Wellner:2016ModeConstrained}.  %
  Thus,
  $K^0$ satisfies the condition needed for
  Lemma~\ref{lem:MC-caricature-properties}.  Now suppose that
  \begin{equation*}
    \sup_{t \in K} (\widehat{\varphi}_n^0  - \varphi_0)(t) \ge C \epsilon_n
    \quad \mbox{ or } \quad
    \sup_{t \in [A + \delta_n, B - \delta_n]} (\varphi_0 -
    \widehat{\varphi}_n^0(t) \ge C \epsilon_n
  \end{equation*}
  for some $C > 0$ where $\epsilon_n = \rho_n^{2/5}$ and $\delta_n =
  \rho_n^{1/5} = \epsilon_n^{1/2}$.  By Lemma~A.3 of \cite{DR2009LC} (stated
  below as Lemma~\ref{lem:A.3-DR2009} for convenience) with $\epsilon = C
  \epsilon_n$, if $C \ge K(2,L)^{-2}$ and $n$ is large it follows that there is
  a random interval $[c_n, c_n + \delta_n]$ either contained in $K^0$ or
  contained in $K \setminus K^0$ on which either $\widehat{\varphi}_n^0 -
  \varphi_0 \ge C \epsilon_n / 4$ or $\varphi_0 - \widehat{\varphi}_n^0 \ge C
  \epsilon_n / 4$.  In the case $[c_n,c_n + \delta_n] \subset K^0$, then since
  $\tau_R - \tau_L = O_p(n^{-1/5})$ by
  Proposition~7.3 of
  \cite{Doss-Wellner:2016ModeConstrained}.  %
  (since
  we assumed $\varphi_0''(m) < 0$) so for $n$ large
  \begin{equation*}
    |\widehat{N}| = \tau_R-\tau_L \le \delta_n /4  =  (\log n / n)^{1/5}
    /4,
  \end{equation*}
  so we can find a random function $\Delta_n$ with no more than four knots
  which satisfies the conditions of Lemma~\ref{lem:MC-caricature-properties}.
  If $[c_n,c_n +\delta_n] \subset K \setminus K^0$ then we can find a random
  function $\Delta_n$ with no more than three knots satisfying the conditions
  of Lemma~\ref{lem:UC-caricature-properties}.
  Now, calculating as in the proof of Theorem~4.1 of
  \cite{DR2009LC}, %
  we can see that for a constant $G_0$,
  \begin{equation*}
    C^2 \le \frac{ 16 G_0^2 (1 + o(1)) \epsilon_n^{-2} \rho_n}{
      \sigma^2(\Delta_n)}
    = \frac{ 16 G_0^2 (1 + o(1))}{ \delta_n^{-1} \sigma^2(\Delta_n)}
    \le \frac{48 G_0^2 (1 + o(1))}{\inf_{t \in K} f_0(t)}.
  \end{equation*}
  But if we choose $C$ strictly larger than the constant on the right side we
  find that the set is empty, and hence has probability $0$ on an event with
  probability increasing to $1$.

\end{proof}

\begin{lemma}[Lemma~A.3,   \cite{DR2009LC}] %
  \label{lem:A.3-DR2009}
  For any $\beta \in [1,2]$ and $L > 0$ there exists a constant $K =
  K(\beta,L) \in (0,1]$ with the following property: suppose that $g$ and
  $\hat g$ are concave and real-valued functions on $T = [A,B]$ where
  $g \in {\cal H}^{\beta,L}(T).$  Let $\epsilon > 0 $ and
  $0 < \delta < K \min(B-A, \epsilon^{1/\beta})$.  Then
  \begin{equation*}
    \sup_{t \in T} (\hat g-g)(t) \ge \epsilon
    \quad \mbox{ or } \quad
    \sup_{ t \in [A + \delta, B - \delta]} (g - \hat g)(t) \ge \epsilon
  \end{equation*}
  implies that %
  for some $c \in [A, B - \delta]$ %
  \begin{equation*}
    \inf_{t \in [c,c+\delta]} (\hat g - g)(t) \ge \epsilon / 4
    \quad \mbox{ or } \quad
    \inf_{t \in [c,c+\delta]} (g - \hat g)(t) \ge \epsilon / 4.
  \end{equation*}
\end{lemma}

\section{Local asymptotic distribution theory near the mode}
\label{sec:AsymptoticDistTheorNearMode}

\subsection{Limit processes and scaling relations}
\label{ssec:PrepLimitProcessesScaling}

From
Theorems~5.1 and 5.2 of
\cite{Doss-Wellner:2016ModeConstrained}, %
we know that the processes $H$ and $H^{(2)} = \widehat{\varphi}$
and $H^0$ and $(H^0)^{(2)} = \widehat{\varphi}^0$ exist and are unique in the limiting
Gaussian white noise problem described by (\ref{WhiteNoiseCanonicalConcave}).
We now
introduce further notation and basic scaling results that are needed in the proof of Theorem~\ref{LRasympNullDistribution}.
As in
\cite{MR1891741} %
Appendix A, Proposition A.1,   and
Theorem 4.6 of
\cite{BRW2007LCasymp}
(noting the corrections indicated in Subsection~\ref{ssec:BRW2009Corrections} below), let
$\sigma \equiv 1/\sqrt{f_0 (m)}$,  $a = | \varphi_0^{(2)} (m)|/4!$,
and let
\begin{eqnarray*}
  && Y_{a,\sigma}(t) \equiv \sigma \int_0^t W(s) ds - a t^4 \stackrel{d}{=}  \sigma (\sigma/a)^{3/5} Y( (a/\sigma)^{2/5} t) , \\
  && Y_{a,\sigma}^{(1)} (t) = \sigma W(t) - 4a t^3 \stackrel{d}{=}  \sigma (\sigma/a)^{1/5} Y^{(1)} ( (a/\sigma)^{2/5} t),
\end{eqnarray*}
where $Y \equiv Y_{1,1}$.  These processes arise as the limits of
appropriate (integrated) localized empirical processes.
Similar relations are satisfied by the unconstrained and constrained invelope processes $H_{a,\sigma}$, $H_{a,\sigma}^0$,
and their
derivatives:  with $H\equiv H_{1,1}$ and $H^0 \equiv H_{1,1}^0$, where  $H^0$ can be either $H_L$ or $H_R$,
\begin{eqnarray*}
  && H_{a,\sigma}(t)   \stackrel{d}{=}  \sigma (\sigma/a)^{3/5} H( (a/\sigma)^{2/5} t) , \\
  && H_{a,\sigma}^0 (t)\stackrel{d}{=}  \sigma (\sigma/a)^{3/5} H^0 ( (a/\sigma)^{2/5} t), \\
  && H_{a,\sigma}^{(1)} (t)  \stackrel{d}{=}  \sigma (\sigma/a)^{1/5} H^{(1)} ( (a/\sigma)^{2/5} t),\\
  && (H_{a,\sigma}^0 )^{(1)} (t)  \stackrel{d}{=}  \sigma (\sigma/a)^{1/5} (H^0)^{(1)} ( (a/\sigma)^{2/5} t),\\
\end{eqnarray*}
and
\begin{eqnarray}
  \widehat{\varphi}_{a,\sigma}
  & = & H_{a,\sigma}^{(2)} \stackrel{d}{=} \sigma^{4/5} a^{1/5} H^{(2)} ((a/\sigma)^{2/5} \cdot ) \nonumber \\
  &  = & \frac{1}{\gamma_1 \gamma_2^2} H^{(2)} ( \cdot / \gamma_2)
  \equiv \frac{1}{\gamma_1 \gamma_2^2} \widehat{\varphi} (\cdot / \gamma_2),  %
  \label{ScalingRelationUnConstrained}
\end{eqnarray}
and, similarly,
\begin{eqnarray}
  \widehat{\varphi}_{a,\sigma}^0
  & = & (H_{a , \sigma}^0)^{(2)} \stackrel{d}{=} \sigma^{4/5} a^{1/5} (H^0)^{(2)} ((a/\sigma)^{2/5} \cdot ) \nonumber \\
  & = & \frac{1}{\gamma_1 \gamma_2^2} (H^0)^{(2)} ( \cdot / \gamma_2)
  \equiv \frac{1}{\gamma_1 \gamma_2^2} \widehat{\varphi}^0 (\cdot / \gamma_2),
  \label{ScalingRelationConstrained}
\end{eqnarray}
Here
\begin{eqnarray}
  && \gamma_1 = \left ( \frac{f_0(m)^4 | \varphi_0^{(2)} (m) |^3}{(4!)^3}  \right )^{1/5} = \frac{1}{\sigma} \left ( \frac{a}{\sigma} \right )^{3/5} , \\
  && \gamma_2 = \left ( \frac{(4!)^2}{f_0 (m) | \varphi_0^{(2)} (m) |^2 } \right )^{1/5}  = \left ( \frac{\sigma}{a} \right )^{2/5},
  \label{GammaDefnsLogConcaveAtMode}
\end{eqnarray}
and we note that
\begin{eqnarray}
  && \gamma_1 \gamma_2^{3/2} = \sigma^{-1} = \sqrt{f_0 (m)},  \ \ \ \gamma_1 \gamma_2^4 = a^{-1} = \frac{4!}{| \varphi_0^{(2)} (m)|} ,
  \label{GammaRelationsPart1} \\
  && \gamma_1 \gamma_2^2 = \frac{1}{C(m, \varphi_0)} \equiv \left ( \frac{4! f_0 (m)^2}{| \varphi_0^{(2)} (m)|} \right )^{1/5} .
  \label{GammaRelationsPart2}
\end{eqnarray}

\subsection{Corrections for \cite{BRW2007LCasymp}}  %
\label{ssec:BRW2009Corrections}
In (4.25) of
\cite{BRW2007LCasymp}, %
replace
$$
Y_{k,a,\sigma} (t) := a \int_0^t W(s) ds - \sigma t^{k+2}
$$
by
$$
Y_{k,a,\sigma} (t) := \sigma \int_0^t W(s) ds - a t^{k+2}
$$
to accord with \cite{MR1891741}, page 1649, line -4, when $k=2$.
In (4.22) of
\cite{BRW2007LCasymp}, %
page 1321, replace
the definition of $\gamma_1$ by
\begin{eqnarray*}
  \gamma_1 = \left ( \frac{f_0 (x_0)^{k+2} | \varphi_0 (x_0)|^3}{(4!)^3} \right )^{1/(2k+1)} .
\end{eqnarray*}
In (4.23) of
\cite{BRW2007LCasymp}, %
page 1321, replace the definition of $\gamma_2$ by
\begin{eqnarray*}
  \gamma_2 = \left ( \frac{((k+2)!)^2}{f_0 (x_0) | \varphi_0^{(k)} (x_0) |^2 } \right )^{1/(2k+1)} .
\end{eqnarray*}
When $k=2$ and $x_0 = m$, these definitions of $\gamma_1, \gamma_2$ reduce to $\gamma_1 $ and $\gamma_2$
as given in (\ref{GammaDefnsLogConcaveAtMode}).
One line above (4.24) of
\cite{BRW2007LCasymp}, %
page 1321, change $Y_{a,k,\sigma} $ to $Y_{k,a,\sigma}$.

\section{Lemmas}

Below are some useful lemmas.
\begin{lemma}
  \label{lem:f-to-varphi-taylor}
  Let $f_{in} = e^{\varphi_{in}}$ for $i=1,2$, and let $x$ be such that
  $|f_{1n}(x) - f_{2n}(x)| \to 0 $ as $n \to \infty$.  Then
  \begin{align}
    f_{1n}(x)-f_{2n}(x)
    & = \lp \varphi_{1n}(x) - \varphi_{2n}(x) \rp
    e^{\tilde \epsilon_{n}(x)} e^{\varphi_{2n}(x)} \label{eq:f-to-varphi-taylor-1} \\
    & = \lp \varphi_{1n}(x) - \varphi_{2n}(x)  + \frac{\lp \varphi_{1n}(x) -
      \varphi_{2n}(x)  \rp^2 }{2} e^{\epsilon_n(x)} \rp  e^{\varphi_{2n}(x)}
    \label{eq:f-to-varphi-taylor-2}
  \end{align}
  where $\tilde \epsilon_n(x)$ and $\epsilon_n(x)$ both lie between $0$ and
  $\varphi_{1n}(x) - \varphi_{2n}(x)$, and thus converge to $0$ as $n \to
  \infty$.
\end{lemma}
\begin{proof}
  Taylor expansion shows
  \begin{align*}
    f_{1n}(x)-f_{2n}(x) = e^{\varphi_{1n}(x)} - e^{\varphi_{2n}(x)}
    & = \lp e^{\varphi_{1n}(x) -\varphi_{2n}(x)} - 1 \rp e^{\varphi_{2n}(x)}\\
    & = \lp \varphi_{1n}(x) - \varphi_{2n}(x) \rp e^{\epsilon_n(x)} e^{\varphi_{2n}(x)},
  \end{align*}
  yielding \eqref{eq:f-to-varphi-taylor-1}.  The second
  expression, \eqref{eq:f-to-varphi-taylor-2}, follows from a similar
  (two-term) expansion.
\end{proof}

\begin{lemma}
  \label{lem:rem:local-to-global-square-integral}
  Assume $\vvo$ is twice continuously differentiable in a neighborhood of $m$
  and $\vvo''(m) < 0$.
  Let $I$ be a random interval whose endpoints are in an $O_p(n^{-1/5})$
  neighborhood of $m$.  Let $D$
  be such that for any $\xi_n \to m$, $|(\vvna)'(\xi_n)- \vvo'(\xi_n)| \le D n^{-1/5}$ with
  probability $1- \epsilon$ for large $n$ by
  Corollary~5.4 of   \cite{Doss-Wellner:2016ModeConstrained}. %
  for $\epsilon > 0$.
  Assume $n^{1/5} \lambda(I) \ge 8 D / \vvo^{(2)}(m)$.
  Let $L>0$,  $\epsilon > 0$, and $\check \delta > 0$.  Suppose
  there exists $\check K>0$ such that
  \begin{equation}
    \label{eq:knot-lemma:integral-ineq}
    \int_{I} (\vvna - \vvn)^2 d\lambda \le \frac{\check \delta}{L} \check K n^{-1}
  \end{equation}
  with probability $1- \epsilon$ for $n$ large.  Then for any interval $J
  \subset I$ where $\lambda(J) = Ln^{-1/5}$, we have with probability $1 -
  \epsilon$ for $n$ large
  \begin{enumerate}[label=(\Alph*)] %
  \item \label{item:lem:integral-to-sup:vv}
    $\Vert \vvn - \vvna \Vert_{J} \le  \delta O_p(n^{-2/5})$,
  \item \label{item:lem:integral-to-sup:ff}
    $\Vert \ffn - \ffna \Vert_J \le \delta O_p(n^{-2/5})$,
  \item \label{item:lem:integral-to-sup:FF}
    $\Vert \FFn - \FFna \Vert_J \le \delta O_p(n^{-3/5})$, and
  \item \label{item:lem:integral-to-sup:knots} there exist knots $\eta \in
    S(\vvn) \cap J$ and $\eta^0 \in S(\vvna) \cap J$ %
    such that $| \eta - \eta^0 | \le \delta O_p(n^{-1/5})$,
  \item   \label{item:lem:integral-to-sup:vv-prime}
    and, letting $K  = [\max(\eta, \eta^0) + \delta O_p(n^{-1/5}), \sup J -
    \delta O_p(n^{-1/5})]$, we have $\Vert \vvn' - (\vvna)'
    \Vert_K \le \delta O_p(n^{-1/5})$,
  \end{enumerate}
  where $ \delta \to 0$ as $\check \delta \to 0$.
\end{lemma}
\begin{proof}
  First we prove the first three statements.
  By Taylor expansion, $\vvo'(x) = \vvo^{(2)}(\xi) (x - m)$ where $\xi$ is
  between $x$ and $m$. Let $J = [j_1, j_2] \subseteq I$.  Then $\vvo'(j_2) -
  \vvo'(j_1) \ge |\vvo^{(2)}(m)| L n^{-1/5} 2$ for $n$ large enough since
  $\vvo^{(2)}$ is continuous near $m$.  Note with probability $1 - \epsilon$,
  $|(\vvna)'(j_1) - \vvo'(j_1)|$ and $|(\vvna)'(j_2) - \vvo'(j_2) | $ are
  less than $n^{-1/5} D$ by applying
  Corollary~5.4 of   \cite{Doss-Wellner:2016ModeConstrained} twice taking $\xi_n = j_1$
  and $\xi_n = j_2$, and $C=0$ (not $C = \lambda(J)$).
  Here $(\vvna)'$ may be the right or left derivative.  Now apply
  Lemma~\ref{lem:rem:L2-to-sup-1}  (taking $I$ in that
  lemma to be our $J$) with
  $\varphi_U' - \varphi_L' = n^{-1/5}\lp 2D + |\vvo^{(2)}(m)| L /2 \rp$
  and $\epsilon  = \check \delta \check K n^{-1}  / L$.
  Then for small enough $\check \delta$,
  \begin{equation*}
    \lp \frac{ \check \delta \check K}{L (2D + |\vvo^{(2)}(m)|L / 2)^2}
    \rp^{1/3} n^{-1/5}
    \le L n^{-1/5} = \lambda(J),
  \end{equation*}
  as needed.  Thus, Lemma~\ref{lem:rem:L2-to-sup-1} allows us to conclude
  \begin{equation*}
    \Vert \vvna - \vvn \Vert
    \le \lp 8 \frac{ n^{-6/5} \check \delta \check K }{L}
    \lp 2D +  \vvo^{(2)}(m) L / 2 \rp \rp^{1/3}.
  \end{equation*}
  Thus taking $\check \delta $ so that $\check \delta \check K D \to 0$, we
  see that
  Lemma~\ref{lem:rem:local-to-global-square-integral}~\ref{item:lem:integral-to-sup:vv}
  holds.
  Then  \ref{item:lem:integral-to-sup:ff}
  follows by the
  delta method (or Taylor expansion of $\exp$).
  Note that $\check K$ and $D$  depend only on $\epsilon$, not on $I$.

  We show \ref{item:lem:integral-to-sup:FF} and
  \ref{item:lem:integral-to-sup:knots} next.  Note that
  \ref{item:lem:integral-to-sup:FF} follows from
  \ref{item:lem:integral-to-sup:knots} and
  \ref{item:lem:integral-to-sup:ff}.  This is because
  \begin{equation*}
    \FFn(x) - \FFna(x) = \FFn(\eta) - \FFna(\eta) + \int_{\eta}^x
    (\ffn(x)-\ffna(x)) dx.
  \end{equation*}
  By \ref{item:lem:integral-to-sup:ff}, the second term above is $\delta
  O_p(n^{-3/5})$ since $x \in J$ satisfies $|x - \eta| \le L n^{-1/5}$.  We
  can next see that the first term in the previous display is $\delta^{1/2}
  O_p(n^{-3/5})$.  Notice that $\sup_{t \in [\eta,\eta^0]} |\int_\eta^t
  d(\FFn(u) - \Fn(u))| = \delta^{1/2} O_p(n^{-3/5})$ by
  Lemma~\ref{lem:rem:FFn-Fn-empirical-proc-arg}, where the random variable
  implicit in the $O_p$ statement depends on $J$ only through $L$.
  Since $|\FFn(\eta) - \Fn(\eta)| \le 1/n$ by Corollary~2.4 of
  \cite{Doss-Wellner:2016ModeConstrained}, we see that $\sup_{t \in
    [\eta,\eta^0]} | \FFn(t) - \Fn(t)| = \delta^{1/2} O_p(n^{-3/5})$.
  Similarly, since $|\FFna(\eta^0) - \Fn(\eta^0)| \le 1/ n$ by Corollary~2.12
  of \cite{Doss-Wellner:2016ModeConstrained}, $\sup_{t \in [\eta,\eta^0]} |
  \FFna(t) - \Fn(t)| = \delta^{1/2} O_p(n^{-3/5})$ by analogous computations.
  Together, these let us conclude that $\sup_{t \in [\eta,\eta^0]} | \FFna(t)
  - \FFn(t)| = \delta^{1/2} O_p(n^{-3/5})$.
  \begin{mylongform}
    \begin{longform}
      an alternative proof approach would be just to use the point $s$ where
      $\FFn(s) = \FFna(s)$

    \end{longform}
  \end{mylongform}

  Now we prove \ref{item:lem:integral-to-sup:knots}.  Let $ \delta > 0$ and
  define $i_{1}$ and $i_{2}$ by $I = [i_1 - \delta n^{-1/5}, i_2 + \delta
  n^{-1/5}]$, taking $\delta$ small enough that $i_1 < i_2$. Let
  $ i_2-i_1 = \tilde M n^{-1/5}$.  Then, by the Taylor expansion of $\vvo'$
  (see the beginning of this proof), $\vvo'( i_2 ) - \vvo'(i_1) \ge
  |\vvo^{(2)}(m)| \tilde M n^{-1/5} / 2$ for $\delta$ small and $n$ large
  enough, since $\vvo^{(2)}$ is continuous.  Additionally, by applying
  Corollary~7.1 of   \cite{Doss-Wellner:2016ModeConstrained} %
  twice, taking $\xi_n =
  i_1$ and $\xi_n = i_2$ and $C=0$ (not $C = \tilde M$), we find $D$
  (independent of $\lambda(I)$) such that $|(\vvna)'(i_1) - \vvo'(i_{1})|$
  and $|(\vvna)'(i_{2}) - \vvo'(i_{2})|$ are, with probability $1 -
  \epsilon$, less than $n^{-1/5} D$.  Here $(\vvna)'$ may be the right or
  left derivative.  For $\delta $ small enough, by assumption $D <
  |\vvo^{(2)}(m)| \tilde M / 8$, and we then have $(\vvna)'(i_{1}+) -
  (\vvna)'(i_{2}-) \ge |\vvo^{(2)}(m)| \tilde M n^{-1/5} / 4$.  We do not
  know a priori how much $(\vvna)'$ decreases at any specific knot in
  $S(\vvna)$, but by partitioning $[i_1,i_2]$ into intervals of a fixed length, we
  can find one such interval on which $(\vvna)'$ decreases by the
  corresponding average amount.  That is, there exists
  a subinterval of $[i_1,i_2]$, denoted
  $J^* = [l^*, r^*],$ %
  of length $2 \delta n^{-1/5}$, such that
  \begin{equation*}
    (\vvna)'(l^*+) - (\vvna)'(r^*-)
    \ge \frac{|\vvo^{(2)}(m)| \tilde M n^{-1/5} / 4}{ (i_2-i_1) / 2\delta n^{-1/5} }
    = \frac{|\vvo^{(2)}(m)| \delta n^{-1/5}}{2}
    \equiv K n^{-1/5}
  \end{equation*}
  since $i_2-i_1 =  \tilde M n^{-1/5}.$ Let $x^*_l = \sup \lb x \in
  J^* : (\vvna - \vvn)'(x) \ge 0 \rb$, and let $x^*_l = l^*$ if the set is empty,
  and let $x^*_r = \inf \lb x \in J^* : (\vvna - \vvn)'(x) \le 0 \rb$, and
  $x^*_r = x^*_l = r^*$ if the set is empty.  Now $(\vvna - \vvn)'$ decreases
  by at least $K n^{-1/5} / 2$ on either $[l^*, x_l^*]$ or on $[x_r^* , r^*]$
  (since $(\vvna- \vvn)'$ is constant on $[x^*_l, x_r^*]$).
  Let $\eta^0_l = \inf S(\vvna) \cap J^*$ and $\eta^0_r = \sup S(\vvna) \cap
  J^*$.
  Now by assumption $\vvn$ is linear on $[\eta^0_l - \delta n^{-1/5},
  \eta^0_r + \delta n^{-1/5}]$ (since $\vvn$ is linear on $J^*$ and within
  $\delta n^{-1/5}$ of any knot of $\vvna$), and so
  \begin{align}
    (\vvna - \vvn)'(u) & \ge \frac{K n^{-1/5}}{2}
    &  \mbox{ for } u \in [\eta^0_l
    - \delta n^{-1/5}, \eta^0_l], &
    \quad \mbox { or }
    \label{eq:knot-lemma:slopeineq-1} \\
    (\vvna - \vvn)'(u) & \le - \frac{ Kn^{-1/5}}{2}
    & \mbox{ for } u \in
    [\eta^0_r, \eta^0_r + \delta n^{-1/5}], & \label{eq:knot-lemma:slopeineq-2}
  \end{align}
  depending on whether $ (\vvna - \vvn)'$ decreases by $K n^{-1/5}/2$ on $[l^*,
  x^*_l]$ (in which case $\eta^0_l \le x^*_l$) or on $[x^*_r, r^*]$ (in which
  case $x^*_r \le \eta^0_r$).  (In the former case,
  \eqref{eq:knot-lemma:slopeineq-1} holds because $(\vvna - \vvn)'$ is
  nonincreasing and its decrease to $0$ on $[l^*, x_l^*]$ actually happens on
  $[\eta^0_l, x_l^*]$ since $\eta^0_l$ is the last knot of $\vvna$ in $J^*$.
  Similar reasoning in the latter case yields
  \eqref{eq:knot-lemma:slopeineq-2}.)
  If \eqref{eq:knot-lemma:slopeineq-1} holds then
  \begin{equation*}
    \sup_{u \in [\eta^0_l -
      \delta n^{-1/5}, \eta^0_l]} | \vvna(u) - \vvn(u)| \ge ( Kn^{-1/5}/2)
    (\delta n^{-1/5}/2) = |\vvo^{(2)}(m) | n^{-2/5} \delta^2 / 8,
  \end{equation*}
  and if \eqref{eq:knot-lemma:slopeineq-2} holds then $\sup_{u \in [\eta^0_r,
    \eta^0_r + \delta n^{-1/5}]} | \vvna(u) - \vvn(u)| \ge |\vvo^{(2)}(m)|
  n^{-2/5} \delta^2 / 8$.  This allows us to lower bound $\int_{I} ( \vvna(u)
  - \vvn(u))^2$, to attain a contradiction with
  \eqref{eq:knot-lemma:integral-ineq}.

  Assume that \eqref{eq:knot-lemma:slopeineq-1} holds.  The case where
  \eqref{eq:knot-lemma:slopeineq-2} holds is shown analogously.  Let $ z =
  \argmin_{u \in [i_1,i_2]} | \vvna(u) - \vvn(u)|$.  Let $L$ be the affine
  function such that $L(z) = \vvna(z) - \vvn(z)$ and $L$ has slope $K
  n^{-1/5} / 2$.  Then for $x \in [\eta^0_l - \delta n^{-1/5}, \eta^0_l]$,
  $|L(x)| \le | \vvna(x) - \vvn(x)|$.  For $z \le x \le \eta^0_l$, this is
  because
  \begin{align*}
    | \vvna(x) - \vvn(x)|
    =  \vvna(x) - \vvn(x)
    & = \vvna(z) - \vvn(z)
    + \int_z^x (\vvna - \vvn)' d\lambda \\
    & \ge L(z) +  \int_z^x L' d\lambda = |L(x)|,
  \end{align*}
  where the first equality holds since $\vvna - \vvn$ is increasing on
  $[\eta^0_l - \delta n^{-1/5}, \eta^0_l]$ (by
  \eqref{eq:knot-lemma:slopeineq-1}), so by the definition of $z$, $\vvna(x)
  - \vvn(x) \ge 0$ for $x \ge z$, and the last is similar.  For $\eta^0_l -
  \delta n^{-1/5} \le x \le z$,
  \begin{align*}
    - |\vvna(x)-\vvn(x)|
    = (\vvna(x) -\vvn)(x)
    & = (\vvna-\vvn)(z) - \int_x^z (\vvna - \vvn)'d\lambda \\
    & \le L(z) - \int_x^z L' d\lambda
    = - |L(x)|,
  \end{align*}
  where the first and last equalities follow because for $\eta_l^0 \le x \le
  z$, since $\vvna - \vvn$ is increasing it must be negative by the
  definition of $z$. Thus,
  $|L(u)| \le | \vvna(u) - \vvn(u)|$ on $[\eta_l^0 - \delta n^{-1/5}, \eta_l^0]$ so
  \begin{equation*}
    \frac{\delta^3 K^2 n^{-1}}{3 \cdot 2^5}
    \le \int_{[\eta_l^0 - \delta n^{-1/5}, \eta_l^0]} L^2 d\lambda
    \le \int_{[\eta_l^0 - \delta n^{-1/5}, \eta_l^0]} (\vvna - \vvn)^2 d\lambda
    \le \int_{I}  (\vvna - \vvn)^2 d\lambda,
  \end{equation*}
  where the quantity on the far left is $\int_0^{\delta n^{-1/5}/2} (x K
  n^{-1/5}/2)^2 dx$.  This is a contradiction if $\delta$ is fixed and we let
  $\check \delta \to 0$, since $[\eta_l^0 - \delta n^{-1/5}, \eta_r^0 +
  \delta n^{-1/5}] \subset I$ by the definition of $[i_1,i_2]$, and then
  $ \int_{I} (\vvna - \vvn)^2 d\lambda \le \check \delta \check K n^{-1}$.  A
  similar inequality can be derived if \eqref{eq:knot-lemma:slopeineq-2}
  holds.  Thus $\delta \to 0$ as $\check \delta \to 0$.

  Finally, we show \ref{item:lem:integral-to-sup:vv-prime} holds with
  similar logic.  Let $\xi_1 < \xi_2$ be points such that $\vvn$ is linear on
  $[\xi_1, \xi_2]$, and let $\delta > 0$.  Then if all knots $\xi^0$ of
  $\vvna$ satisfy $| \xi_i - \xi^0 | > \sqrt{\delta} n^{-1/5}$, $i =1 ,2$,
  then we can see that $\Vert \vvn' - (\vvna)' \Vert_{[\xi_1,\xi_2]} \le
  \sqrt{\delta} n^{-1/5}$.  This is because at any $x \in [\xi_1 + \sqrt{
    \delta} n^{-1/5}, \xi_2 - \sqrt{    \delta} n^{-1/5}]$, if
  $(\vvna - \vvn)'(x) > \sqrt{ \delta } n^{-1/5}$, then $(\vvna - \vvn)' >
  \sqrt{\delta} n^{-1/5}$ on $[x- \sqrt{\delta}n^{-1/5}, x]$ (since $\vvn$ is
  linear on a $\sqrt{\delta} n^{-1/5}$ neighborhood of $x$), so $(\vvna - \vvn)(x
  - \sqrt{\delta} n^{-1/5}, x] > \delta n^{-2/5}$, a contradiction.  Here we
  use the notation $g(a,b] = g(b) - g(a)$.  Similarly if $(\vvna - \vvn)'(x)
  < \sqrt{\delta}n^{-1/5}$ then $(\vvna - \vvn)(x , x+\sqrt{\delta} n^{-1/5}]
  < - \delta n^{-2/5}$, a contradiction.
  We have thus shown that
  if we take $\eta$ and $\eta^0$ to be the largest knot pair within
  $\sqrt{\delta} n^{-1/5}$ (meaning $\max(\eta, \eta^0)$ is largest among
  such pairs) then $\Vert (\vvn - \vvna)' \Vert_{[\max(\eta, \eta^0), \sup J
    - \sqrt{\delta}n^{-1/5}]} \le \sqrt{ \delta} n^{-1/5}$, by
  Part~\eqref{item:lem:integral-to-sup:vv} and by partitioning $[\max(\eta,
  \eta^0), \sup J]$ into intervals on which $\vvn$ is linear.
\end{proof}

The proofs of Parts~\ref{item:lem:integral-to-sup:knots} and
\ref{item:lem:integral-to-sup:vv-prime} in the previous lemma could also be
completed with the roles of $\vvna$ and $\vvn$ reversed.
The next lemma provides the calculation used in
Lemma~\ref{lem:rem:local-to-global-square-integral} to translate an upper
bound on
$\int_{I} \lp \vvna - \vvn \rp^2  d \lambda$
into an upper bound on $\sup_{x \in I} \lp \vvna(x) - \vvn(x)
\rp$ for an appropriate interval $I$.
\begin{lemma}
  \label{lem:rem:L2-to-sup-1}
  Let $\epsilon > 0$.  Assume $\varphi_i$, $i=1,2$, are functions on an
  interval $I$, where
  \begin{equation}
    \label{eq:6}
    \varphi_L' \le \varphi_i'(x) \le \varphi_U' \, \mbox{ for } x \in I, i=1,2,
  \end{equation}
  where $\varphi_i'$ refers to either the left or right derivative and
  $\varphi_L',\varphi_U'$ are real numbers.  Assume
  $I$ is of length no smaller than $\lp \epsilon / \lp
  \varphi_U'-\varphi_L'\rp^2  \rp^{1/3}$.
  Assume  that $\int_I
  \lp \varphi_1(x) - \varphi_2(x) \rp^2  dx \le \epsilon$.
  Then
  \begin{equation*}
    \sup_{x      \in I} | \varphi_1(x)-\varphi_2(x) | \le \lp 8
    \epsilon \lp \varphi_U' - \varphi_L' \rp \rp^{1/3}.
  \end{equation*}
\end{lemma}
\begin{proof}
  Assume that $x$ is such that $\varphi_1(x) - \varphi_2(x) = \delta$, and
  without loss of generality, $\delta > 0$.  Then by \eqref{eq:6}, if $y$ is
  such that $|x - y | \le (\delta/2) / (\varphi_U'-\varphi_L')$ then
  $\varphi_1(y) - \varphi_2(y) \ge \delta / 2$.  Thus, if $I$ is an interval
  whose length is no smaller than $(\delta /2) / (\varphi_U'-\varphi_L')$
  then if $\varphi_1(x) - \varphi_2(x) \ge \delta$ for any
  $x \in J \subseteq I$ where the length of $J$ is equal to
  $(\delta /2) / (\varphi_U'-\varphi_L')$, then
  \begin{equation*}
    \int_I \lp \varphi_1(x) -    \varphi_2(x) \rp^2  dx
    \ge \int_J \lp \varphi_1(x) -    \varphi_2(x) \rp^2 dx
    \ge  \frac{\delta^2}{2^2} \lambda(J)
    = \frac{1}{8}  \frac{\delta^3}{\varphi_U'-\varphi_L'} .
  \end{equation*}
  Thus, substituting $\epsilon = \frac{1}{8}
  \frac{\delta^3}{\varphi_U'-\varphi_L'}$,
  we see that
  for $x \in I$, $|\varphi_1(x)-\varphi_2(x) | \le \lp 8
  \epsilon \lp \varphi_U'-\varphi_L'\rp  \rp^{1/3}$, as desired.
  \begin{mylongform}
    \begin{longform}
      If we set
      $\epsilon = \frac{\tilde{f}}{8}
      \frac{\delta^3}{\varphi_U'-\varphi_L'}$, then $\delta = \lp 8 \epsilon \lp
      \varphi_U'-\varphi_L'\rp / \tilde f \rp^{1/3}$, and then
      \begin{equation*}
        \frac{\delta/2}{ \varphi_U'-\varphi_L'} = \frac{1}{2}
        \frac{8^{1/3}}{\tilde
          f^{1/3}} \epsilon^{1/3} \lp \varphi_U'-\varphi_L' \rp^{-2/3}
        = \frac{\epsilon^{1/3} \lp \varphi_U'-\varphi_L'\rp^{-2/3}}{ \tilde f^{1/3}}
      \end{equation*}
      is the length of $I$.  (Note that $\tilde f$ is now replaced by $1$ in
      the lemma. --2016 sep 09)
    \end{longform}
  \end{mylongform}
\end{proof}

\begin{lemma}
  \label{lem:rem:FFn-Fn-empirical-proc-arg}
  Let %
  either Assumption~\ref{CurvatureAtTheMode} hold at $x_0 = m$ or
  Assumption~\ref{CurvatureAwayFromMode} hold at $x_0 \ne m$, and let $\FFna$
  and $\FFn$ be the log-concave mode-constrained and unconstrained MLEs of
  $\FFo$.  Let $I = [v_1, v_2]$ be a random interval whose dependence on $n$
  is suppressed and such that $n^{1/5}(v_j - x_0) = O_p(1), j = 1,2$. Then
  \begin{equation}
    \label{eq:rem:dFFn-dFn}
    \left \{
      \begin{array}{l} \sup_{t \in I} \lv \int_{[v_1,t]} d(\FFn(u) - \Fn(u)) \rv \\
        \ \\
        \sup_{t \in I} \lv \int_{[v_1,t]}  d(\FFna(u) - \Fn(u)) \rv
      \end{array} \right \}
    = \sqrt{ \lambda(I)} O_p(n^{-2/5}).
  \end{equation}
  The random variables implicit in the $O_p$ statements in
  \eqref{eq:rem:dFFn-dFn} depend on $I$ through its length (in which they are
  increasing) and not the location of its endpoints.
\end{lemma}
\begin{proof}
  We analyze $    \sup_{t \in I} \lv \int_{v_1}^t
  d(\FFn(u) - \Fn(u)) \rv$ first.   Note
  \begin{equation}
    \label{eq:rem:7}
    \begin{split}
      \MoveEqLeft
      \sup_{t \in I} \left| \int_{[v_1,t]} d  \lp     \Fn
        - \FFn \rp  \right| \\
      & \le  \sup_{t \in I}
      \bigg| \int_{v_1}^{t} \lp \ffn(u) - \ffo(v_1)  - (u- v_1)
      \ffo'(v_1) \rp du  \\
      & \qquad \qquad \qquad - \int_{v_1}^{t} \lp \ffo(u) - \ffo(v_1) - (u- v_1) \ffo'(v_1) \rp
      du \\
      & \qquad \qquad \qquad - \int_{[v_1,t]} d\lp \Fn-\FFo\rp \bigg|.
    \end{split}
  \end{equation}
  By (the proof of) Lemma~8.16 of  \cite{Doss-Wellner:2016ModeConstrained},
  $$
  \sup_{t \in I}  \lv \int_{[v_1, t]} d\lp \Fn - \FFo \rp \rv =
  \sqrt{\lambda(I)} O_p(n^{-2/5}).
  $$
  The supremum over the middle term in \eqref{eq:rem:7} is $\lambda(I)
  O_p(n^{-2/5})$ by a Taylor expansion of $\ffo$,
  and applying in addition Lemma~4.5 of
  \cite{BRW2007LCasymp}, we see that the supremum over the first term in
  \eqref{eq:rem:7} is also $\lambda(I) O_p(n^{-2/5})$.

  The same analysis, using Proposition~5.3 or Corollary~5.4
  of \cite{Doss-Wellner:2016ModeConstrained}, applies to $ \sup_{t \in I} \lv
  \int_{[v_1,t]} d(\FFna(u) - \Fn(u)) \rv$.  Note in all cases that the
  random variables implicit in the $O_p$ statements depend on $I$ only
  through its length (and they are increasing in the length) and not the
  location of its endpoints, since $\Vert {\ffo} \Vert = \ffo(m) < \infty$ and
  since $\ffo^{(2)}$ is continuous and so uniformly bounded in a neighborhood
  of $x_0$.
\end{proof}

When we apply the previous lemma, the length of $I$ will depend on
$  \epsilon$ which gives the probability bound implied by our $o_p$ statements
whereas its endpoints will depend on $\delta$, which gives the size bound
implied by our $o_p$ statements.

\begin{lemma}
  \label{lem:rem:epsilon-inequality}
  Let $\epsilon(x)$ and $\tilde \epsilon(x)$ be defined by $e^x = 1 + x +
  2^{-1} x^2 e^{\epsilon(x)}$ and $ e^x = 1 + x e^{\tilde \epsilon(x)}.$
  Then
  \begin{equation}
    e^{\epsilon(x)} \le 2 e^{\tilde \epsilon(x)}.
  \end{equation}
\end{lemma}
\begin{proof}
  We can see
  \begin{eqnarray*}
    e^{\epsilon(x)} = \frac{e^{x}-1 -x}{(x^2/2)} = \frac{2}{x^2} \sum_{k=2}^\infty \frac{x^k}{k!} = \sum_{k=0}^\infty \frac{2 x^k}{(k+2)!} .
  \end{eqnarray*}
  Similarly,
  \begin{eqnarray*}
    e^{\tilde{\epsilon}(x)} = \frac{e^{x}-1 }{x} = \frac{1}{x} \sum_{k=1}^\infty \frac{x^k}{k!} = \sum_{k=0}^\infty \frac{x^k}{(k+1)!} .
  \end{eqnarray*}
  Comparing coeffients in the two series, we see that
  \begin{eqnarray*}
    \frac{2}{(k+2)!} \le \frac{1}{(k+1)!}  \ \ \mbox{for all} \ k\ge 0
  \end{eqnarray*}
  since $k+2 \ge 2$ for $k\ge 0$.  It follows that $e^{\epsilon (x)} \le e^{\tilde{\epsilon} (x)} $ for all $x \ge 0$.
  This implies that $\epsilon (x) \le \tilde{\epsilon }(x)$ for all $x \ge 0$.

  Now for $x \le 0$ we have
  \begin{align*}
    2e^{\tilde{\epsilon} (x)}  - e^{\epsilon (x)}
    & =  \frac{2}{x} \left \{ e^x -1 - \frac{1}{x} (e^x -1 -x ) \right \} \\
    & =  \frac{2}{x} \left \{ \sum_{k=1}^ \infty  \frac{x^k}{k!} - \sum_{k=1}^\infty
      \frac{ x^k}{(k+1)!} \right \}
    =  \frac{2}{x} \sum_{k=1}^\infty \frac{k}{(k+1)!} x^k
  \end{align*}
  where the infinite sum is negative for all $x<0$ since the first term is negative.  It follows that
  \begin{eqnarray*}
    2e^{\tilde{\epsilon} (x)}  - e^{\epsilon (x)}  \ge 0
  \end{eqnarray*}
  for all $x\le 0$.     Combined with the result for $x\ge 0$ the claimed result holds:
  $e^{\epsilon(x)} \le 2 e^{\tilde{\epsilon}(x)}$ for all $x \in \RR$.
\end{proof}

\bibliographystyle{imsart-nameyear}
\bibliography{ModeConstrained}

\end{document}